\documentclass{article}

\usepackage{graphicx,url}
\usepackage[all]{xy}

\newtheorem{theorem}{Theorem}
\newtheorem{lemma}[theorem]{Lemma}
\newtheorem{proposition}[theorem]{Proposition}
\newtheorem{corollary}[theorem]{Corollary}

\newcommand{\proof}{\par\noindent\textbf{Proof. }}

\renewcommand{\Re}{\mathrm{Re}\,}

\usepackage{amsfonts,amssymb,amsmath}
\newcommand{\C}{\ensuremath{\mathbb{C}}}
\newcommand{\F}{\ensuremath{\mathbb{F}}}

\newcommand{\J}{\ensuremath{\mathbb{J}}}
\newcommand{\N}{\ensuremath{\mathbb{N}}}
\newcommand{\Q}{\ensuremath{\mathbb{Q}}}
\newcommand{\R}{\ensuremath{\mathbb{R}}}
\newcommand{\T}{\ensuremath{\mathbb{T}}}
\newcommand{\Z}{\ensuremath{\mathbb{Z}}}

\newcounter{ste}
\newlength\largeurLabelsuite
\newenvironment{suite}%
{ \begin{list}%
     {\stepcounter{ste}\theste\phantom{i}$^o$}%
     {\setlength{\topsep}{0pt}%
       \setlength{\labelwidth}{-\largeurLabelsuite}%
       \setlength{\leftmargin}{0pt}%
       \setlength{\itemsep}{0pt}}%
   }%
   {\setcounter{ste}{0}\end{list}}

\newlength\largeurLabelListe
\newenvironment{Liste}
{ \begin{list}%
     {---}%
     {\setlength{\topsep}{0pt}%
       \setlength{\labelwidth}{-\largeurLabelListe}%
       \setlength{\leftmargin}{0pt}%
       \setlength{\itemsep}{0pt}}%
   }%
   { \end{list} }

\newlength\largeurLabelliste
{ \begin{list}%
     {--}%
     {\setlength{\topsep}{0pt}%
       \setlength{\labelwidth}{-\largeurLabelliste}%
       \setlength{\leftmargin}{0pt}%
       \setlength{\itemsep}{0pt}}%
   }%
   { \end{list} }

\begin{document}

\title{Unchained polygons and the N-body problem}
\author{Alain Chenciner \& Jacques F{\'e}joz}
\maketitle

\begin{quote}
  \begin{flushright}
    \textit{To the memory of J. Moser, with admiration}
  \end{flushright}
\end{quote}

\begin{abstract}
  We study both theoretically and numerically the Lyapunov families
  which bifurcate in the vertical direction from a horizontal relative
  equilibrium in $\R^3$. As explained in \cite{CF1}, very symmetric
  relative equilibria thus give rise to some recently studied classes
  of periodic solutions.  We discuss the possibility of continuing
  these families globally as action minimizers in a rotating frame
  where they become periodic solutions with particular symmetries. A
  first step is to give estimates on intervals of the frame rotation
  frequency over which the relative equilibrium is the sole absolute
  action minimizer: this is done by generalizing to an arbitrary
  relative equilibrium the method used in \cite{BT} by V. Batutello
  and S. Terracini.

  In the second part, we focus on the relative equilibrium of the
  equal-mass regular $N$-gon.  The proof of the local existence of the
  vertical Lyapunov families relies on the fact that the restriction
  to the corresponding directions of the quadratic part of the energy
  is positive definite. We compute the symmetry groups
  $G_{\frac{r}{s}}(N,k,\eta)$ of the vertical Lyapunov families
  observed in appropriate rotating frames, and use them for
  continuing the families globally.

  The paradigmatic examples are the ``Eight'' families for an odd
  number of bodies and the ``Hip-Hop'' families for an even number.
  The first ones generalize Marchal's $P_{12}$ family for 3 bodies,
  which starts with the equilateral triangle and ends with the Eight
  \cite{CM,Ma2,CFM,CF1,S}; the second ones generalize the Hip-Hop
  family for 4 bodies, which starts from the square and ends with the
  Hip-Hop \cite{CV,CF1,TV}. 

  We argue that it is precisely for these two families that global
  minimization may be used. In the other cases, obstructions to the
  method come from isomorphisms between the symmetries of different
  families; this is the case for the so-called ``chain''
  choreographies (see \cite{S}), where only a local minimization
  property is true (except for $N=3$). Another interesting feature of
  these chains is the deciding role played by the parity, in
  particular through the value of the angular momentum. For the
  Lyapunov families bifurcating from the regular $N$-gon whith $N\leq
  6$ we check in an appendix that locally the torsion is not zero,
  which justifies taking the rotation of the frame as a parameter.
\end{abstract}

\clearpage  
\tableofcontents 
\listoffigures 
\clearpage

\section*{RELATIVE EQUILIBRIA\\ AND THEIR  LYAPUNOV FAMILIES}
\addcontentsline{toc}{section}{RELATIVE EQUILIBRIA AND THEIR LYAPUNOV
   FAMILIES}

\section{Lyapunov families bifurcating normally from relative
  equilibria}

We consider relative equilibria rotating in the horizontal plane $\R^2
\times \{0\} \subset \R^3$, that is, of the form
$$e^{\J \omega_1 t} \, C,$$
where $C$ is a central configuration, $\omega_1 >0$ is the frequency
and $\J$ is the horizontal-rotation operator acting diagonally on the
horizontal components by a rotation of $\pi/2$ and trivially on the
vertical ones. Along the paper, various normalizations of $C$ and
various rotating frames will be considered, and the corresponding
relative equilibrium will often be denoted by $\bar x(t)$.

\subsection{The horizontal and vertical variational equations}

The variational equations associated with a solution $x_i=x_i(t)$ of
the Newton's equations for $N$ bodies in the Euclidean space $(\R^3,
\| \cdot \|)$
$$\ddot x_i=\sum_{i\not=j}m_j\frac{x_j-x_i}{||x_j-x_i||^3},\;
i=1,\ldots,N$$
are
$$\ddot {\delta x_i}=\sum_{j\not= i}m_j\frac{\delta x_j-\delta
   x_i}{||x_j-x_i||^3}-3\sum_{j\not= i}m_j
\frac{\langle x_j-x_i,\delta x_j-\delta
  x_i\rangle}{||x_j-x_i||^5}(x_j-x_i).\eqno(VE)$$

Along a planar solution $x_i=x_i(t)$ (supposed to be horizontal), the
Pythagoras theorem implies these equations split into horizontal and
vertical parts. Namely, if $\delta x_i=h_i+z_i \in \C \oplus \R$ is
the decomposition of the variations in respectively horizontal and
vertical components, (VE) is equivalent to the following pair of
equations: the complicated {\it horizontal variational equation}
$$\ddot {h_i}=\sum_{j\not= 
i}m_j\frac{h_j-h_i}{||x_j-x_i||^3}-3\sum_{j\not= 
i}m_j \frac{\langle
   x_j-x_i,h_j-h_i\rangle}{||x_j-x_i||^5}(x_j-x_i),\eqno(HVE)$$
and the much simpler {\it vertical variational equation}
$$\ddot z_i=\sum_{j\not= i}m_j\frac{z_j-z_i}{||x_j-x_i||^3}.\eqno(VVE)$$

\subsection{Vertical variations of a relative equilibrium}
\label{sec:vvRe}

As the mutual distances $r_{ij}=||x_j-x_i||$ stay constant along a
relative equilibrium motion, the corresponding vertical variational
equation has constant coefficients:
{\small
   $$\left(\begin{array}{c}\ddot z_1\\ .\\.\\. \\ \ddot
       z_{N}\end{array}\right)
   =\left(
     \begin{array}{cccccc}
       -\sum_{j\not=1}{m_j\over r_{j1}^3}&{m_2\over
         r_{21}^3}&{m_3\over r_{31}^3}&.&.&{m_N\over r_{N1}^3}\\
       {m_1\over r_{12}^3}&-\sum_{j\not=2}{m_j\over
         r_{j2}^3}&{m_3\over r_{32}^3}&.&.&{m_N\over r_{N2}^3}\\
       .&.&.&.&.&.\cr.&.&.&.&.&.\\
       {m_1\over r_{1N}^3}&{m_2\over r_{2N}^3}&{m_3\over
         r_{3N}^3}&.&.&-\sum_{j\not=N}{m_j\over
         r_{jN}^3}
     \end{array}\right) \left(
     \begin{array}{c}
       z_1\\.\\.\\. \\  z_{N}
     \end{array}\right),$$}
or $\ddot z={\cal W}z$, where ${\cal W}$ is, up to a factor -2, the
transposed of the {\sl Wintner-Conley matrix} (i.e., up to a factor -2
and a transposition, it represents the endomorphism $A$ of the space
of codispositions in \cite{AC}). The sum of the elements of any line
of ${\cal W}$ is equal to 0. This implies that it acts on the space of
dispositions ${\cal D}=\R^N/(1,1,\ldots,1)\R$, to which $z$ rightly
belongs.

The matrix ${\cal W}$ is symmetric for the mass scalar product, which
means that
$$z'\cdot {\cal W}z''={\cal W}z'\cdot z'',\quad\hbox {where}\quad z'\cdot
z''=\sum_{i=1}^N{m_iz'_iz''_i}.$$ Hence the eigenvalues of ${\cal W}$
are real; because the Newton force is attractive, they are also
negative (see~\cite[Proposition~1]{Mo} or~\cite{AC}). We will call
their distinct values the {\sl vertical frequencies} and denote them
$-\omega_1^2,-\omega_2^2,\ldots,-\omega_{\ell}^2$, $\ell\le N-1$, {\it
  where the $\omega_k$ are chosen to be positive}; note that by making
${\cal W}$ act on ${\cal D}$ we have taken away the eigenvalue 0.

Now, let $Z_1,\ldots,Z_{(N-1)}$ be a basis of ${\cal D}$ (which can be
chosen orthogonal) consisting of eigenvectors of ${\cal W}$ with
eigenvalues $-\omega(1)^2,\ldots,-\omega(N-1)^2$, not necessarily
distinct. The general solution $Z(t)$ of (VVE) is of the form
$$Z(t)=\sum_{j=1}^{N-1} \Re(\alpha_jZ_je^{i\omega(j)t}),\quad \alpha_j\in\C,$$
that is
$$Z(t)=\sum_{k=1}^\ell{\Re(W_ke^{i\omega_k t})},$$
where each $W_k$ is a {\it complex} eigenvector of ${\cal W}$ with
eigenvalue $-\omega_k^2$.

\subsection{What is known about the vertical frequencies}
\label{sec:known}

First, one of the frequencies is that of the relative equilibrium,
$\omega_1$: it corresponds indeed to infinitesimal rotations around a
horizontal axis.

Let us now compare $\omega_1$ to the other frequencies. Let $I(x) =
|x|^2=\sum{m_i||x_i||^2}$ be the moment of inertia (i.e. the
square norm in the mass metric), and let $U(x)=\sum_{i<j}{m_im_j\over
  ||x_j-x_i||}$ be the potential function. Since a central
configuration $C$ is a critical point of the scaled potential $\tilde
U=I^{1\over 2}U$, it is natural to write $U$ in terms of $\tilde U$:
$$dU(x)\delta x=-(x\cdot\delta x)I(x)^{-{3\over 2}}\tilde
U(x)+I(x)^{-{1\over 2}}d\tilde U(x)\delta x$$
and
$$d^2U(x)(\delta x,\delta x)=3(x\cdot\delta x)^2I(x)^{-{5\over 2}}
\tilde U(x) - 2(x\cdot\delta x)I(x)^{-{3\over 2}}d\tilde U(x)\delta x$$
$$-|\delta x|^2I(x)^{-{3\over 2}}\tilde U(x)+I(x)^{-{1\over 2}}d^2\tilde
U(x)(\delta x,\delta x).$$

Let now $x=C$, the central configuration. Let us split a tangent
vector $\delta x$ as before, into horizontal and vertical components
$h\in \C$ and $z\in \R$. Since $d\tilde U(C)=0$,
\begin{eqnarray*}
   d^2U(C)(h+z,h+z)
   &=&3(C\cdot h)^2I(C)^{-{5\over 2}}\tilde U(C) -
   (|h|^2+|z|^2)I^{-\frac{3}{2}}\tilde U(C) \\
   && + I(C)^{-{1\over 2}} d^2\tilde U(C)(h+z,h+z).
\end{eqnarray*}
In particular, for a vertical variation,
$$d^2U(C)(z,z)=-|z|^2I(C)^{-{3\over 2}}\tilde U(C)+I(C)^{-{1\over
    2}}d^2\tilde U(C)(z,z).$$ But, taking the scalar product with
$\bar x(t)= e^{\J \omega_1 t} C$ of the identity
$$\nabla U(C)=\ddot {\bar x}(t)=-\omega_1^2\bar x(t),$$ where the
gradient is relative to the mass metric, one gets $I(C)^{-{3\over
     2}}\tilde U(C)=\omega_1^2$. Finally, one deduces that, if ${\cal
   W}(C)Z_k=-\omega_k^2Z_k$,
$$d^2\tilde  U(C)(Z_k,Z_k)=(\omega_1^2-\omega_k^2)I^{1\over 2}|Z_k|^2.$$
If we write any vertical variation $Z=\sum_{i=1}^{N-1}u_iZ_{(i)}$ in
terms of the orthogonal basis introduced in section~\ref{sec:vvRe}, we
get
$$d^2\tilde U(C)(Z,Z) = \sum_{i=1}^{N-1}
\left(\omega_1^2-\omega(i)^2\right)I^{\frac{1}{2}}u_i^2|Z_{(i)}|^2.$$
Is was proved by Pacella (in the equal-mass case) and Moeckel (in the
general case) that a planar central configuration $C$ of at least 4
bodies is never a local minimum of $\tilde U$: more precisely, there
exists always some $Z$ such that $d^2\tilde U(C)(Z,Z)<0$. This implies
the

\begin{lemma}[Pacella, Moeckel]
   For any central configuration $C$ of at least 4 bodies, at least one
   of the normal frequencies $\omega_k$ is strictly greater than
   $\omega_1$.
\end{lemma}

This has a direct consequence on the Hessian of the Lagrangian action:
let
$$\zeta(t)=Z_k(\frac{\omega_1}{\omega_k}t)=\Re(W_ke^{i\omega_1t}).$$
Then, for the action during time $T=\frac{2\pi}{\omega_1}$, one has
$$d^2{A}(
\bar x(t))(\zeta(t),\zeta(t))=\pi |W_k|^2(\omega_1^2-\omega_k^2),$$
which has the same sign as $d^2\tilde  U(C)(Z_k,Z_k)$.
The proof of this formula is a direct computation: on the one hand,
$$\int_0^{T}|\dot\zeta(t)|^2dt=\pi \omega_1^2|W_k|^2,$$
on the other hand, we have seen that, for any $t$,
$d^2U(C)(\zeta(t),\zeta(t))=-\omega_k^2\zeta(t)$.
This implies the following identity and hence the desired formula:
$$\int_0^{T}{d^2U(
  \bar x(t))(\zeta(t),\zeta(t))dt} =
-\omega_k^2\int_0^{2\pi}{|\zeta(t)|^2dt}=-\omega_k^2\pi |W_k|^2.$$
\smallskip

It follows that, if $N\ge 4$, a relative equilibrium is never a local
minimizer of the action (compare \cite{C1}).

\subsection{Lyapunov families and their lifts}
\label{sec:lyapunov}

\begin{figure}[h]
    \centering
    \includegraphics[width=10cm]{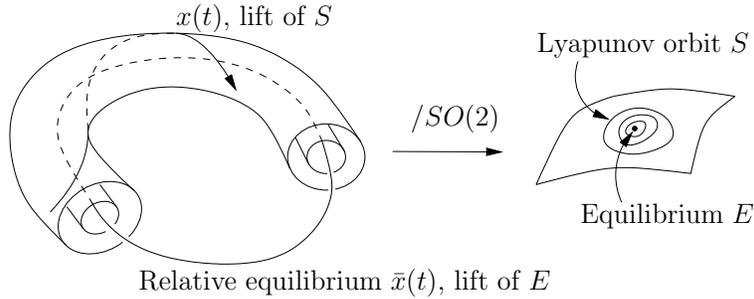}
    \caption{Lyapunov family of a relative equilibrium}
    \label{fig:lyapunov}
\end{figure}

The system may be reduced by fixing the angular momentum and
quotienting by rotations around its axis; if the relative equilibrium
$\bar x(t)$ is horizontal, the angular momentum vector will be
vertical. Besides, by definition of a relative equilibrium, the
solution $\bar x(t)$ projects to an equilibrium point $E$ of the
reduced system. Each pair of conjugate, purely imaginary eigenvalues
of the linearization of the reduced vector field corresponds to a
$2\pi/\omega_k$-periodic eigenmode. In good cases, these eigenmodes
give rise to a family of solutions of the reduced system of period
close to $2\pi/\omega_k$ in the neighborhood of the equilibrium,
called a \emph{Lyapunov family}.

\medskip In the simplest case where no other frequency of the
equilibrium is an integer multiple of $\omega$ (this is for example
the case of the Hip-Hop family of the four-body problem), the Lyapunov
center theorem~\cite{Mos} proves the local existence, uniqueness and
regularity of this family. In resonant cases, one can use either an
argument involving a higher order normal form (such as in \cite{CF2}
for the $P_{12}$ family), or a topological, existence argument (such
as the Weinstein-Moser theorem~\cite{Mos}); for a detailed discussion
on the simplest examples, see section~\ref{sec:localCont}. Global
existence theorems also apply \cite{AY,CMP}, but they give only
general information on the nature of the families.

\medskip Each Lyapunov family $L$ lifts to the original phase space as
a local one-para\-meter (non-reduced, Lyapunov) family of invariant
two-tori foliated by quasi-periodic solutions; on a given torus, any
two solutions differ only by a rotation around the vertical axis
(figure~\ref{fig:lyapunov}). A non-reduced Lyapunov family associated
with an eigenmode $z(t) = \Re Z_k \, e^{i\omega_k t}$ is tangent to
the lift to the phase space of the \emph{cylindrical} family of curves
$$t \mapsto (\bar x(t),Az(t-\varphi)),$$
para\-metrized by a phase $\varphi$ and an amplitude $A$. When
observed in a rotating frame which puts into resonance the horizontal
and vertical frequencies, these curves become a cylindrical family of
loops, whose period is an integer multiple of $2\pi/\omega_k$, say
$s2\pi/\omega_k$, $s\in\Z \setminus 0$.

In a continuous family of rotating frames whose rotation frequency now
depends on $S$ (in general the rotation of the frame will vary, due to
torsion, while it is fixed along the one-parameter family of solutions
in the tangent cylinder), the quasiperiodic Lyapunov solutions $x(t)$
themselves become periodic solutions whose homology class in $\T^2$ is
independant of $S$. More precisely, if $\mu(S) \in \R/\Z$ is their
monodromy, defined by
$$x(t+ 2\pi/\omega(S)) = e^{\J 2\pi \mu (S) / \omega(S)}x(t),$$
where $\omega(S)$ is the frequency of $S$ and $\J$, as before, is the
rotation operator,\footnote{Vertical colinear motions should be avoided
  since they correspond to a singularity of the reduction.} the
condition is that
$$\dfrac{\mu(S) - \varpi}{\omega(S)} \in \Q.$$

The scaling invariance of Newton's equations (if $x(t)$ is a
solution, so is $\lambda^{-\frac{2}{3}}x(\lambda t)$ for any
$\lambda>0$) entails that each family can be freely rescaled. Fixing
the norm of the angular momentum is a way to choose a normalization.
But when moving away from the relative equilibrium, this may
artificially lead to singularities, at elements of the family having
zero angular momentum. Fixing instead the period (in addition to the
direction of the angular momentum) is then a better choice of
normalization, allowing to bypass zero-angular momentum elements. (On
the other hand there is no hope with our techniques to bypass
singularities where the period itself tends to infinity.) This allows
to assume that in the above family of rotating frames the Lyapunov
solutions all have the same period $s2\pi /\omega_k$.

\medskip When the loops in the cylindrical family have some
symmetries, the question arises whether the solutions in the the
Lyapunov family, when observed in the above family of rotating frames,
share the same symmetry. The uniqueness statement in the Lyapunov
theorem implies that the answer is positive locally when the eigenmode
is simple.  However, the first eigenmode is always degenerate, which
requires more analysis (see section~\ref{sec:localCont}).

The really interesting question is to continue the family globally:
if the integrated torsion effect allows continuation up to the
inertial frame, we shall have completed the primary purpose of our
study --to prove the existence of symmetric solutions in the inertial
frame. If the symmetry group is rich enough, this program can be
achieved using minimization of the Lagrangian action. This will be
illustrated in section~\ref{sec:global} with
choreographic and Hip-Hop solutions arising from the regular $N$-gon
relative equilibrium.

\section{Minimizing properties of relative equilibria}
\label{sec:minimizing}

Assume here that the central configuration $C$ is normalized so that
$e^{\J t} C$ is a relative equilibrium.

We study the action minimizing properties of the rescaled relative
equilibrium
$$\bar x(t) = x^\omega_\varpi(t) = e^{\J (\omega+\varpi) t} C_\varpi^\omega,$$
where 
$$C_\varpi^\omega=(\omega+\varpi)^{-\frac{2}{3}}C.$$ 
In a frame rotating with frequency $\varpi$, it becomes the
$T = 2\pi/\omega$-periodic loop
$$\bar y(t) = y_\varpi^\omega(t) = e^{-\J\varpi t} x_\varpi^\omega(t).$$

Let $\Lambda$ be the space of $T$-periodic loops in the configuration
space with $H^1$ regularity. For each value of $\varpi$ we define the
action $\mathcal {A}_\varpi(y)$ of $y(t)\in\Lambda$ as the action of
the path $x(t)=e^{\J\varpi t}y(t)$ in the inertial frame:
$$\mathcal{A}_\varpi(y)=\mathcal{A}(x) =
\int_0^{T}\left(\frac{1}{2}|\dot x|^2+U(x)\right)dt
=\int_0^{T}\left(\frac{1}{2}|\dot y+\J\varpi y|^2+U(y)\right)dt.$$
We call $\mathcal {A}_\varpi$ the action in a frame rotating with
frequency $\varpi$. \smallskip

\subsection{Local minimizing properties}
\label{sec:localMin}

In a frame whose
rotation frequency is $\varpi$, $x^\omega_\varpi(t)$ has frequency
$\omega$.  The Lagrangian action during a given time $T$ of a curve
$x(t)=(x_1(t),\ldots,x_n(t))$ in the configuration space is
$$\mathcal{A}(x) = \int_0^T\left[{1\over 2}|\dot
  x(t)|^2+U(x(t))\right]dt,$$ where, as before, $|\dot
x|^2=\sum{m_i||\dot{x}_i||^2}$ and $U(x) = \sum_{i<j}{m_im_j\over
  ||x_j-x_i||}$ is the potential function. As the kinetic and
potential energies of $x^\omega_\varpi$ are independant of time, its
action during time $T$ is readily computed to be
$$\mathcal{A}(C,\omega,T,\varpi) := \mathcal{A}
(x^\omega_\varpi) = (\omega+\varpi)^\frac{2}{3} \frac{T}{2\pi}a,$$
where $a$ is the action of the normalized solution $x_0^1(t) = e^{\J
  t}C$.

As expected, the action tends to zero when $\varpi$ tends to
$-\omega$: at the limit, the rotation of the frame alone accounts for
the motion and the bodies have only to stay still at infinity in the
inertial frame.  \smallskip

We are interested in the local action minimizing properties under
appropriate symmetry constraints of the members of this family in the
following case:
$$\omega=\frac{r}{s}\omega_k,\quad 
T=s\frac{2\pi}{\omega_k}=r\frac{2\pi}{\omega},$$ where $\omega_1=1,
\omega_2,\cdots,\omega_k,\cdots$ are the frequencies of the vertical
variational equation associated with $x_0^1(t)$ and $r,s$ are mutually
prime integers. Note that $T$ is the minimal period of
$(x_\varpi^\omega(t),Z_k(t))$ in the rotating frame with frequency
$\varpi$, where $Z_k(t)$ has frequency $\omega_k$.  \smallskip

Let us first compute the Hessian
$$d^2\mathcal{A}(x)(\xi,\xi) =
\int_0^T \left[|\dot\xi(t)|^2+d^2U(x(t))(\xi(t),\xi(t))\right]dt.$$ of
the action when $x=x^\omega_\varpi$ and $\xi=(0,0,Z_k)$ with $Z_k$ a
solution of (VVE) of the form $Z_k(t)=\Re(W_ke^{i\omega_k t})$ and
$W_k$ a complex eigenvector of ${\cal W}(C)$ with eigenvalue
$-\omega_k^2$ (see section 1.2).  As $\int_0^{T}{\cos^2(\omega_k
  t)dt}=\int_0^{T}{\sin^2(\omega_k t)dt}$, we get immediately that
$$\int_0^{T}|\dot Z_k(t)|^2dt=\omega_k^2\int_0^{T}| Z_k(t)|^2dt.$$
Hence
$$d^2\mathcal{A}(x^\omega_\varpi)(Z_k,Z_k)=
\int_0^{T}\left[\omega_k^2|Z_k(t)|^2+d^2U(C^\omega_\varpi)
   (Z_k(t),Z_k(t))\right]dt.$$

Now $dU(x)X=X\cdot {\cal W}x$, where $x$ and $X$ belong to ${\cal
   D}\otimes \R^3\equiv{\cal D}^3$ and ${\cal W}$ acts on each
component. Since the mutual distances, hence ${\cal W}$, stay constant
under a vertical variation, we get
$$d^2U(C^\omega_\varpi)(Z_k(t),Z_k(t))=Z_k(t)\cdot{\cal
   W}(C^\omega_\varpi)Z_k(t)=-(C^\omega_\varpi/C)^{-3}\omega_k^2|Z_k(t)|^2,$$
that is
$$d^2\mathcal{A}(x^\omega_\varpi)(Z_k,Z_k)=
\bigl(1-(\omega+\varpi)^2\bigr)\omega^2_k\int_0^T|Z_k(t)|^2dt.$$ In
particular, $d^2\mathcal{A}(x^\omega_\varpi)(Z_k,Z_k)=0$ iff
$\varpi=\pm 1-\omega=\pm 1-\frac{r}{s}\omega_k$.  This is easy to
explain: $\mathcal{A}$ is the action in a frame which rotates in such
a way that the period of the relative equilibrium becomes $s$ times
the one of $Z_k(t)$. It follows that $Z_k(t)$ defines a periodic
Jacobi field and this implies that the corresponding variation belongs
to the kernel of the Hessian of $\mathcal{A}$.

\begin{figure}[h]
    \centering
    \includegraphics[width=9cm]{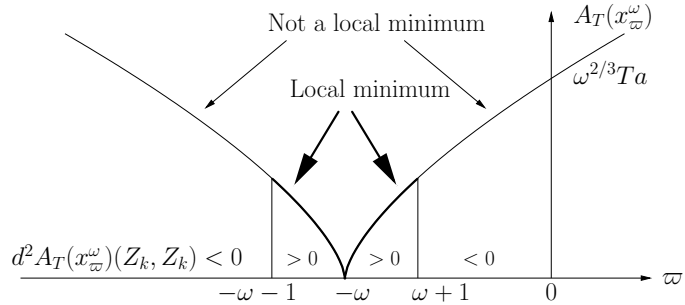}
    \caption{Action of a relative equilibrium in a frame with rotation
      speed $\varpi$}
    \label{fig:actionRE}
\end{figure}

In the rotating frame, the solution $x^\omega_\varpi$ possesses all
the symmetries of the solutions of the (VVE) considered above. If
the symmetry constraints that we impose are strong enough to allow
only one type of solutions of (VVE) and no solution of (HVE),
$x^\omega_\varpi$ will be the only local minimizer of the action under
thee constraints in the whole interval $-(1+\omega)\le\varpi\le
1-\omega$.

In the case of the regular $n$-gon configuration of $n$
equal masses, this will be the case for 3 or 4 bodies; but already for
5 bodies, we shall see an example where the symmetry constraint does
not discriminate between two different solutions of (VVE).  In the
next section we study the global minimization problem.

\goodbreak

\subsection{Global minimizing properties}
\label{sec:globalMin}

This section, which as well as~\ref{sec:NGon} is taken from the
unpublished manuscript \cite{C2}, develops an idea first used by
Barutello and Terracini in \cite{BT}. \smallskip

\textit{We shall now systematically use the shorter notation $\bar
  x(t) = e^{\J(\omega +\varpi)t} \bar x$ in place of $x_\varpi^\omega
  (t) = e^{\J(\omega +\varpi)t} C_\varpi^\omega$.}

Let $G$ be a finite subgroup of $O(2)\times\Sigma_N\times O(3)$. It
has a natural action on the space $\Lambda$ of $T$-periodic loops in
the configuration space of the problem:
$g=(\tau,\sigma,\rho)\in G$ acts on
$y(t)=\bigl(y_1(t),\cdots,y_{N}(t)\bigr)\in\Lambda$ according to the
diagram

$$
\begin{array}[t]{cccccc}
   y : &\R/T\Z &\times &\{1,...,N\} &\rightarrow &\R^{3}\\
   &\tau\downarrow\phantom{\tau} &&\sigma\downarrow\phantom{\sigma}
   &&\rho\downarrow\phantom{\rho}\\
   gy : &\R/T\Z &\times &\{1,...,N\} &\rightarrow &\R^{3}.
\end{array}$$
The transformed loop is
$$gy_j(t)=\rho y_{\sigma^{-1}(j)}(\tau^{-1}(t));$$
this convention is chosen so as to have a left action, i.e.  $(gg') q
= g(g'q)$. \medskip

{\it We suppose that $\bar y(t)=e^{-\J\varpi t}\bar x(t)=e^{\J\omega
    t}\bar x$ belongs to the subset $\Lambda^G\in\Lambda$ of loops
  which are invariant under this action and we look for conditions on
  $G$ and $\varpi$ which insure that the sole absolute minimizer of
  $\mathcal{A}_\varpi$ is $y$.}

\subsubsection{Choice of a new functional}

Following \cite{BT} we are looking for an action functional
$y\mapsto\mathcal{\bar A}_\varpi(y)$ with the following properties:

1) $\mathcal{A}_\varpi(y)\ge\mathcal{\bar A}_\varpi(y)$ for any
$y(t)\in\Lambda^G$,

2) $\mathcal{\bar A}_\varpi$ attains its minimum value at $\bar y(t)$

3) If $y(t)$ is a solution in the rotating frame, then $\mathcal
{A_\varpi}(y)=\mathcal{\bar A}_\varpi(y)$ if and only if $y(t)=\bar
y(t)$.  \smallskip

If such a functional exists, the sole absolute
minimum of $\mathcal{A}_\varpi(y)$ is $\bar y(t)$.
\smallskip

As in \cite{BT}, from which the notations are
borrowed (but with different normalization), one
looks for a functional
of the form
$$\mathcal{\bar A}_\varpi(y)=\mathcal{\bar A}(x) =
\frac{\lambda}{2}\sum_{i<j}\bar\mu_{ij}\xi_{ij}^x+
c \left(\sum_{i<j}\bar\mu_{ij}\xi_{ij}^x \right)^{-\frac{1}{2}},$$
where, if
$x(t)=\left(x_1(t),\cdots, x_{N}(t)\right)=e^{\J\varpi t}y(t)$,
$$\xi_{ij}^x=\int_0^T||x_i(t)-x_j(t)||^2dt=\int_0^T||y_i(t)-y_j(t)||^2dt$$
and all the coefficients $\lambda,c,\bar\mu_{ij}$ are positive.

\subsubsection{Minoration of the potential part of the action}

By Jensen inequality applied to the convex function $f(s)=s^{-\frac{1}{2}}$,
$$\frac{1}{T}\int_0^T\frac{dt}{||x_i(t)-x_j(t)||}\ge
\left(\frac{1}{T}\int_0^T||x_i(t)-x_j(t)||^2\right)^{-\frac{1}{2}}=
\left(\frac{1}{T}\xi_{ij}^x\right)^{-\frac{1}{2}},$$
with equality if and only if
$||x_i(t)-x_j(t)||$
is independent of $t$.
\smallskip

Hence, for any choice of positive $\mu_{ij}$'s,
there exists a maximal $\bar c>0$ such that
$$\int_0^TU(x)dt\ge
T^{\frac{3}{2}}\sum_{i<j}m_im_j\left(\xi_{ij}^x\right)^{-\frac{1}{2}}
\ge \bar c T^{\frac{3}{2}}
\left(\sum_{i<j}\mu_{ij}\xi_{ij}^x\right)^{-\frac{1}{2}}.$$
Indeed, as the function
$$\left(\sum_{i<j}\mu_{ij}\xi_{ij}^x\right)^{\frac{1}{2}}
\left(\sum_{i<j}m_im_j\left(\xi_{ij}^x\right)^{-\frac{1}{2}}\right)$$
is homoge\-neous of degree 0, it is
enough to look for a constrained minimum of
$\sum_{i<j}m_im_j\left(\xi_{ij}^x\right)^{-\frac{1}{2}}$ on
$\sum_{i<j}\mu_{ij}\xi_{ij}^x=1$, i.e. there must exist $\alpha$ such that
$$\forall i<j, \;
\frac{1}{2}m_im_j\left(\xi_{ij}^x\right)^{-\frac{3}{2}}
=\alpha\mu_{ij},\quad{\rm i.e.}\quad \forall i<j, \;
\xi_{ij}^x=\left(\frac{2\alpha\mu_{ij}}{m_im_j}\right)^{-\frac{2}{3}}.$$
{\bf Choice of the $\mu_{ij}$:} On chooses $\mu_{ij}=\bar\mu_{ij}$
such that the $\bar
\xi_{ij}=\left(\frac{\bar\mu_{ij}}{m_im_j}\right)^{-\frac{2}{3}}$ be
equal to the squared mutual distances $\bar r_{ij}^2$ of the central
configuration $\bar x$ (with relative equilibrium motion of period $T$
in the rotating frame). In other words, one chooses
$$\bar \mu_{ij}=m_im_j\bar r_{ij}^{-3}.$$
One checks that $\bar c=\bar U^{\frac{3}{2}}$, where $\bar U=U(\bar x)$.
We have proved the

\begin{lemma} 
  $$\int_0^TU(x(t))dt\ge (\bar{U} T)^{\frac{3}{2}}
  \left(\sum_{i<j}\bar\mu_{ij}\xi_{ij}^x\right)^{-\frac{1}{2}}$$ 
  with equality if and only if there exists $\beta>0$ and $S(t)\in
  O(3)$ such that for all $t$, $x(t)=\beta S(t)\bar x$.
\end{lemma}

The only if part is because the mutual distances of the configuration
$x(t)$ being independent of $t$ and proportional to the corresponding
ones of $\bar x$, this implies the existence of $\beta$ and $S(t)$.

\subsubsection{Minoration of the kinetic part of the action}

Integrating $|x(t)|^2$ by parts between 0 and $T$
does not lead to a boundary term because, if
$x(t)$
itself is not $T$-periodic,  $<x(t),\dot x(t)>$ is $T$-periodic;
one gets
$$\int_0^T|\dot x|^2dt=\int_0^T\sum_{i=1}^{N}m_i\|\dot
x_i\|^2dt=-\int_0^T\sum_{i=1}^{N}m_i<\ddot x_i,x_i>dt.$$ Let $\chi =
\mathcal{D} \otimes \R^3$ be the quotient of $(\R^3)^N$ by the action
of translations, and $\Delta:{\mathcal X}\to{\mathcal X}$ be defined by
$$(\Delta x)_i =
\sum_{j,j\not=i}\frac{\bar\mu_{ij}}{m_i}(x_i-x_j) =
\sum_{j,j\not=i}\frac{m_j}{\bar r_{ij}^3}(x_i-x_j)
\quad\hbox{\rm if}\quad x=(x_1,\cdots, x_{N}).$$
Then
\begin{align*}
   \sum_im_i\int_0^T<(\Delta x)_i(t),x_i(t)>dt
   &=\sum_{i,j,i\not=j}\frac{m_im_j}{\bar
     r_{ij}^3}\int_0^T<x_i-x_j,x_i>dt\\
   &=\sum_{i<j}\bar\mu_{ij}\int_0^T||x_j-x_i||^2dt.
\end{align*}
{\bf Definition.}
Let $\lambda_{\varpi}^G$ be the smallest (eigenvalue) $\lambda$
such that there exists a solution $y(t)\in\Lambda^G$ of the equation $-\ddot
y-2\J\varpi\dot y+\varpi^2 y=\lambda\Delta y$, or equivalently a solution
$x(t)=e^{\J\varpi t}y(t)$ of the equation $-\ddot
x=\lambda\Delta x$ with $y(t)\in\Lambda^G$.
\smallskip

Following \cite{BT}, the Poincar{\'e} inequality of
the Kepler case is replaced by
\begin{lemma} For any $y(t)\in\Lambda^G$, one has
   $$\int_0^T|\dot y(t)+i\J\varpi y(t)|^2dt=\int_0^T|\dot x(t)|^2dt\ge
   \lambda^G_{\varpi}\sum_{i<j}\bar\mu_{ij}\xi_{ij}^x,$$
   with equality if and on if $-\ddot x=\lambda_{\varpi}^G\Delta x$.
\end{lemma}

\begin{proof} Define $\lambda_0$ as the minimum of the positive functional
   $$J(x)=\frac{\sum_{i=1}^Nm_i\int_0^T||\dot
     x_i||^2dt}{\sum_{i<j}\frac{m_im_j}{\bar
       r_{ij}^3}\int_0^T||x_i(t)-x_j(t)||^2dt}$$ defined
   on the set of $x(t)=e^{\J\varpi
     t}y(t)$ in $H^1$ with $y(t)\in\Lambda^G$.  The
   existence of a minimizer $x(t)$ is
   insured as soon as the class of loops $\Lambda^G$
   implies coercivity.  Writing the Euler
   equations, one gets that for such a minimizer,
   $$-\ddot x=J(x)\Delta x.$$
\end{proof}

\begin{lemma}
   $\lambda_0=\lambda_{\varpi}^G$.
\end{lemma}

\begin{proof}
   We have seen that if $x(t)=e^{\J\varpi t}y(t)$
   with $y(t)\in\Lambda^G$ minimizes $J(x)$, it
   satisfies the equation
   $-\ddot x=\lambda\Delta x$ with $\lambda=J(x)$.
   The converse comes from the identity
   $$J(x)=\frac{\sum_{i=1}^Nm_i\int_0^T<-\ddot
     x_i,x_i(t)>dt}{\sum_{i=1}^Nm_i\int_0^T<(\Delta
     x_i)(t),x_i(t)>dt}\;\cdot$$
\end{proof}
\medskip
Being the minimum of a positive functional,
$\lambda_{\varpi}^G$ is necessarily non negative;
this allows to
write its definition as follows:
\begin{lemma}
   $\lambda_{\varpi}^G$ is the smallest $\lambda\ge 0$ such that
   there exists a solution
   $\xi(t)$ of the equation $-\ddot \xi=\Delta \xi$ such that
   $y(t)=e^{-\J\varpi t}\xi(\sqrt{\lambda}t)\in\Lambda^G$ (and in
   particular is $T$-periodic).
\end{lemma}

\noindent{\bf Remark} If $\Lambda$ is the space
of all $H^1$ loops of period $T$,
$\lambda_{min}=0$ because of the lack of coercivity: the min is attained
by motionless bodies at infinity.
\medskip

\noindent{\bf Interpretations of $\Delta$:} up to a harmless factor
2, a transposition changes $\Delta$ into the Wintner-Conley matrix
associated with the central configuration $\bar x$
(cf.~\ref{sec:vvRe}).  In other words, one checks that if
$X=(X_1,\cdots,X_{N})$ is a tangent vector at $\bar x$ to the
configuration space $(\R^3)^N$ (or to its quotient by translations),
$$dU(\bar x)X=\left<X,-\Delta\bar x\right>.$$
Also, if $Z=(Z_1,\cdots,Z_{N})$ is a vertical variation,
$$d^2U(\bar x)(Z,Z)=\left<Z,-\Delta Z\right>,$$
which explains that $\Delta$ is symmetric with respect to the mass
scalar product.  Considered as living in the space $\R^N$ of vertical
variations (resp. in its quotient by translations), the equation
$$\ddot Z=-\Delta Z$$
is the vertical variational equation (VVE) at $\bar x(t)$.

\subsubsection{Minima of $\bar{\mathcal{A}}_\varpi$}

Let $g$ be the real function defined on $\R_+$ by
$$g(s)=\frac{\lambda_{\varpi}^G}{2}s+(\bar UT)^{\frac{3}{2}}s^{-\frac{1}{2}}.$$
The functional $\mathcal{\bar A}_\varpi$ is defined by
$$\mathcal{\bar A}_\varpi(y)=\mathcal{\bar
   A}(x)=g\left(\sum_{i<j}\bar\mu_{ij}\xi_{ij}^x\right).$$
The function
$g$ is strictly convex and its unique minimum is
$$s_{min}=\frac{\bar UT}{(\lambda_{\varpi}^G)^{\frac{2}{3}}}.$$
The minimum of $\mathcal{\bar A}_\varpi(y)=\mathcal{\bar A}(x)$ is attained
when
$\sum_{i<j}\bar\mu_{ij}\xi_{ij}^x=s_{min}$, that is
$$\sum_{i<j}\frac{m_im_j}{\bar
   r_{ij}^3}\int_0^T||x_i(t)-x_j(t)||^2dt=\frac{\bar
   UT}{(\lambda_{\varpi}^G)^{\frac{2}{3}}}\,\cdot$$
This implies the
\begin{lemma}
   A necessary and sufficient condition for the
   relative equilibrium $\bar y(t)$ to be a
   minimizer of the
   action $\mathcal{\bar A}_\varpi$ in $\Lambda^G$,
   is that $\lambda_{\varpi}^G=1$.
\end{lemma}
\begin{proof}
   If $\bar y(t)$ is a minimizer, $\sum_{i<j}\frac{m_im_j}{\bar
     r_{ij}^3}\int_0^T||x_i(t)-x_j(t)||^2dt=\bar UT$ and the above
   identity becomes $\lambda_{\varpi}^G=1$.

   Conversely, if $\lambda_{\varpi}^G=1$, the unique minimum of the
   convex function $g(s)=\frac{1}{2}s+(\bar
   UT)^\frac{3}{2}s^{-\frac{1}{2}}$ is attained at $s_{min}=\bar
   UT$. But $\bar y(t)$ verifies $\forall i,j$, $\xi_{ij}^x=\bar
   r_{ij}^2T$. Hence $\sum_{i<j}\bar\mu_{ij}\xi_{ij}^x=\bar UT$, which
   proves the assertion.
\end{proof}
\smallskip

\begin{lemma} Let $y(t)\in\Lambda^G$ be a solution in the rotating
   frame. Then, the equality
   $\mathcal{A}_\varpi(y)=\mathcal{\bar A}_\varpi(y)$ occurs if and only if
   $\lambda_{\varpi}^G=1$ and $y(t)=\bar y(t)$.
\end{lemma}
\begin{proof}
  $\mathcal{A}_\varpi(y)=\mathcal{\bar A}_\varpi(y)$ if and only if
  the equality holds for both the kinetic and the potential part of
  the action. Equality in the potential part implies that
  $x(t)=e^{\J\varpi t}y(t)=\beta S(t)\bar x$ is a rigid motion. As it
  is a solution, it follows from \cite{AC} that it is a relative
  equilibrium and hence of the form
   $$x(t)=\beta Se^{i\alpha t}\bar x,$$
   where $S$ is a rotation. Now, equality in the
   kinetic part of the action implies that
   $x(t)$ is a solution of the differential equation
   $-\ddot x=\lambda\Delta x$, where $\lambda=\lambda_\varpi^G$. Hence
   $$\beta \alpha^2Se^{i\alpha t}\bar
   x=\lambda\Delta(\beta Se^{i\alpha t}\bar x).$$
   But $\beta Se^{i\alpha
     t}(x_1,\cdots,x_{N})=(\beta Se^{i\alpha
     t}x_1,\cdots,\beta Se^{i\alpha
     t}x_{N})$ and hence
   $$\Delta(\beta  Se^{i\alpha t}\bar
   x)_i=\sum_{j,j\not=i}\frac{m_j}{\bar
     r_{ij}^3}(\beta Se^{i\alpha
     t}x_i-\beta Se^{i\alpha t}x_j)=\beta Se^{i\alpha t}(\Delta\bar x)_i.$$

   Finally,
   $$\beta \alpha^2Se^{i\alpha t}\bar x=\beta Se^{i\alpha
     t}\lambda\Delta\bar x=\beta(\omega+\varpi)^2\lambda Se^{i\alpha t}\bar x,$$
   and $$\alpha^2=(\omega+\varpi)^2\lambda.$$

   So, $x(t)=\beta Se^{i(\omega+\varpi)\sqrt{\lambda}t}\bar x$, and
   $y(t)=e^{-i\varpi t}\beta
   Se^{i(\omega+\varpi)\sqrt{\lambda}t}\bar x$. But
   $y(t)$ is supposed to belong
   to $\Lambda^G$ and, in particular, it must be
   $(T=\frac{2\pi}{\omega})$-periodic, that is for
   all $t$,
   $$e^{-i2\pi\frac{\varpi}{\omega}}e^{-i\varpi
     t}\beta Se^{i(\omega+\varpi)\sqrt{\lambda}t}
   e^{i2\pi\frac{\omega+\varpi}{\omega}\sqrt{\lambda}}\bar x=
   e^{-i\varpi t}\beta Se^{i(\omega+\varpi)\sqrt{\lambda}t}\bar x.$$
   In particular, taking $t=0$, this implies
 
$$e^{i2\pi(\frac{\omega+\varpi}{\omega}\sqrt{\lambda}-\frac{\varpi}{\omega})}=1,
   \quad\hbox{i.e.}\quad  \lambda=1,$$
   provided $\varpi\ne -\omega$. This last condition
   is necessarily satisfied if $y$ is a true
   solution, not
   lying still at infinity.
   From $\lambda=1$ and the fact that $y$ and hence
   also $x$ is a solution, one deduces that
   $\beta=1$
   and hence $y(t)=\bar y(t)$.
\end{proof}
\medskip

We conclude that
\begin{proposition}\label{prop:criterion}
  As long as $\lambda_{\varpi}^G$=1, $\bar y(t)=\bar xe^{\J\omega t}$
  is the sole absolute minimizer of $\mathcal{A}_\varpi$ in
  $\Lambda^G$.
\end{proposition}

\noindent{\bf Remarks} 1) An elementary computation shows that

\begin{eqnarray*}
   \mathcal{A}_\varpi(\bar
   y(t))=\frac{1}{2M}(\omega+\varpi)^2\sum_{i<j}m_im_j\bar
   r_{ij}^2T+\bar UT
   &=&\frac{\sum_{i<j}m_im_j\bar r_{ij}^2}{2MI}\bar
   UT+\bar UT\\
   &=&\frac{3}{2}\bar UT,
\end{eqnarray*}
while
$$\mathcal{\bar A}_\varpi(\bar
y(t))=(\frac{\lambda_\varpi^G}{2}+1)\bar UT$$
and
$$\min\mathcal{\bar A}_\varpi=
(\lambda_\varpi^G)^{\frac{1}{3}}(\frac{\lambda_\varpi^G}{2}+1)\bar
UT= (\lambda_\varpi^G)^{\frac{1}{3}}\mathcal{\bar A}_\varpi(\bar
y(t)).$$
\goodbreak

\noindent 2) In the case of the \textit{choreography symmetry}, that
is of $G=\Z/N\Z$ acting by circularly permuting $N$ equal masses after
one $N$th of the period, this proposition is a rewording of \cite{BT}.
It is proved there that it implies that, for any $N$, the regular
$N$-gon is the sole absolute minimizer of the action among
choreographies.

\subsubsection{Not saying anything on the Italian symmetry}

As soon as relative equilibria with different central configurations
are allowed by the group action, the method cannot work: a criterium
which is essentially local cannot detect the difference in action
between relative equilibria whose configurations are unrelated to each
other. This is coherent with the fact that no conclusion can be drawn
from the property that the criterion for absolute minimization does
not apply.  \medskip

In order to illustrate this remark we look at the Italian symmetry
$$x(t+\frac{T}{2})=-x(t).$$
It is the simplest invariance requirement which implies coercivity and
it possesses two properties needed for our purpose:

\textit{i}) any relative equilibrium satisfies the Italian symmetry;

\textit{ii}) any solution of the (VVE) of a relative equilibrium
solution satisfies the Italian symmetry.  \smallskip

\noindent \textit{In order to compare easily with the formulae in the
  next sections, we consider loops in the configuration space which
  become 1-periodic in a frame rotating with frequency $\varpi$.
  Moreover, when $G=\Z/2\Z$ with the Italian action, we shall use the
  notation $\lambda_\varpi^{It}$ instead of $\lambda_\varpi^G$.}
\smallskip

Let $ x(t)=e^{\J (2\pi+\varpi)t}\bar x$ be some relative equilibrium
($\bar x$ can be any central configuration). Let
$\omega_1,\omega_2,\cdots$ be the \textit{vertical frequencies} (i.e.
the ones of the Vertical Variational Equation $-\ddot z=\Delta z$) of
a relative equilibrium with similar configuration and frequency
$\omega_1$ (see \cite{CF1,Mo}).

\begin{lemma}
   Whatever the central configuration $\bar x$, and whatever the
   dimension (two or three) of the ambiant space, whatever the value of
   $\varpi$,
   $$\sqrt{\lambda_\varpi^{It}}\le\inf_{k}\frac{\omega_1}{\omega_k}.$$
\end{lemma}

\begin{corollary}
   The criterion for being an absolute minimizer does not apply to the
   Italian symmetry as soon as $N\ge 4$: $\lambda_\varpi^{It}$ is
   always strictly smaller than 1.
\end{corollary}

\noindent{\bf Proof of the lemma.} We shall denote by
$\hat\omega_1=2\pi+\varpi,\; \hat\omega_2,\cdots$ the vertical
frequencies associated with the relative equilibrium $\bar x(t)=\bar x
e^{i(2\pi+\varpi)t}$, i.e.
$$\hat\omega_k=\frac{\omega_k}{\omega_1}(2\pi+\varpi).$$
Let us consider a horizontal solution $\xi(t)$ of the equation
$-\ddot\xi=\Delta\xi$, of the form (we identify the horizontal plane
to the complex plane) $\xi(t)=W_ke^{i\hat\omega_k}t$. The periodicity
condition of $y(t)=e^{i\varpi t}\xi(\sqrt{\lambda}t)$ and the Italian
symmetry are satisfied if and only if it is of the form
$y(t)=W_ke^{m2\pi it}$ with $m$ odd, i.e. if
$\hat\omega_k\sqrt{\lambda}-\varpi=m2\pi$, that is
$\frac{\omega_k}{\omega_1}(2\pi+\varpi)=m2\pi,$ that is
$$\sqrt{\lambda}=
\frac{\omega_1}{\omega_k}\left(\frac{m2\pi+\varpi}{2\pi+\varpi}\right).$$
One concludes by choosing $m=1$. \smallskip

{\bf Proof of the corollary.} Let $Z(t)=\Re(W_ke^{i\omega_kt})$ be a
solution of the (VVE) equation associated with the relative
equilibrium $x(t)=A\bar x e^{i\omega_1t}$ with configuration similar
to $\bar x$ and frequency $\omega_1$. Here, $W_k$ is a complex
eigenvector with eigenvalue $\omega_k^2$ of the associated $\Delta$
operator.  Finally, let
$$\zeta(t)=Z(\frac{\omega_1}{\omega_k}t)=\Re(W_ke^{i\omega_1 t}).$$
Then, for the action during time $T=\frac{2\pi}{\omega_1}$, we have
proved in~\ref{sec:known} that
$$d^2\mathcal{A}(
x(t))(\zeta(t),\zeta(t))=\frac{T}{2}|W_k|^2(\omega_1^2-\omega_k^2).$$
The corollary follows from the Pacella-Moeckel lemma recalled
in~\ref{sec:known}.

\paragraph{Remark 1} Not knowing the result of Pacella-Moeckel, one
could have deduced that, in order that the relative equilibrium
solution $x(t)=x_0e^{\J \omega_1 t}$ be the sole absolute minimizer of
the action in $\Lambda^{It}$, it is enough that it be a local
minimizer with respect to the vertical variations, which would have
been a mirific result indeed if it had not been an empty one, except
in the case of the equilateral relative equilibrium of 3 bodies.

\paragraph{Remark 2} In the following sections the choice of the
groups $G_{r/s}(N,k,\eta)$ will eliminate the non trivial horizontal
variations, since the only horizontal loops which are invariant are
relative equilibria of the regular $N$-gon, possibly with a different
frequency and a reordering of the bodies.  The same was true in the
case of choreographies, studied in \cite{BT}: the regular $N$-gon is
the sole central configuration whose relative equilibrium motions are
choreographies.

\textit{The main difference with the study below of the groups
  $G_{r/s}(N,k,\eta)$ lies in the fact that the two integers $m$ and
  $k$ will not be any more independant!} 

\clearpage

\section*{THE CASE OF THE REGULAR $N$-GON}
\addcontentsline{toc}{section}{THE CASE OF THE REGULAR $N$-GON}

{\it CAUTION: It will be convenient from now on to number the bodies
   from 0 to $N-1$ considered as elements of ${\Z}/{n\Z}$. Moreover, we
   suppose that the mass of each body is equal to 1.}

\section{Infinitesimal continuation}
\subsection{Vertical variations}

We focus on the central configuration $C = (1, \zeta, ...,
\zeta^{N-1})$, with $N\geq 3$ and $\zeta=e^{i 2\pi/N}$, in the
horizontal plane identified with $\C$, and on the associated relative 
equilibrium
$$\bar x(t)=(\bar x_0(t),\bar x_1(t),\ldots,\bar
x_{N-1}(t))=e^{\J \omega_1t} \, C$$  
of the equal mass $N$-body problem.

Let
$$\rho_k = r_{i,i+k} =|\zeta^k-1|,\; k=1,\ldots,N-1.$$
In particular, $\rho_{N-k}=\rho_k$ if $k\le[N/2]$. The vertical
variational equation (VVE) reads
$$\ddot z_i=\sum_{j\not=i}{z_j-z_i\over \rho_{|j-i|}^3}$$
or,
$$\begin{pmatrix}
   \ddot z_0\\ \ddot z_1\\ \vdots\\ \ddot z_{N-1}
\end{pmatrix} =
\begin{pmatrix}
   -\sum{1\over \rho_k^3}&{1\over \rho_1^3} &{1\over
     \rho_2^3}&\cdots&{1\over \rho_{N-1}^3}\\
   {1\over \rho_{N-1}^3}&-\sum{1\over \rho_k^3} &{1\over
     \rho_1^3}&\cdots&{1\over \rho_{N-2}^3}\\
   \vdots&\vdots&\vdots&\vdots&\vdots\\
   {1\over \rho_1^3}&{1\over \rho_2^3}&{1\over
     \rho_3^3}&\cdots&-\sum{1\over \rho_k^3}
\end{pmatrix}
\begin{pmatrix}
   z_0\cr z_1\\ \vdots\\ z_{N-1}
\end{pmatrix}.
$$
Such a circulant matrice has an explicit basis of eigenvectors:
a basis of complex eigenvectors is formed by the
$$X_k=(\zeta^k,\zeta^{2k},\ldots,\zeta^{Nk}=1),\quad k=0,\ldots,N-1,$$
with eigenvalues
$$\lambda_k=-\sum{1\over\rho_j^3}+
{1\over\rho_1^3}\zeta^k+{1\over\rho_2^3}\zeta^{2k}+\ldots+
{1\over\rho_{N-1}^3}\zeta^{(N-1)k}= - \sum_{1\leq j \leq N-1}
\frac{1-\zeta^{jk}}{\rho_j^3}.$$
In particular, $\lambda_1=-\omega_1^2$ corresponds to the frequency
of the relative equilibrium. \smallskip

The cases $N$ odd and $N$ even behave slightly differently from one
another.

{\it (i) Case $N=2n+1$ odd:} the $4n$-dimensional phase space ${\cal
   D}^2$ of the variational equation splits into the direct sum of $2n$
invariant eigenplanes.  Indeed, there are exactly $n$ distinct mutual
distances
$$\rho_1=\rho_{N-1},\rho_2=\rho_{N-2},\ldots,\rho_n=\rho_{N-n}=\rho_{n+1}$$
(in increasing order: see lemma~\ref{evOrder} below) and
$\bar\zeta^k=\zeta^{(N-k)}$. Hence, for $k=1,\ldots,n$,
$$\lambda_k=\lambda_{N-k}=\sum_{j=1}^n{{1\over\rho_j^3}
   (\zeta^{jk}+\bar\zeta^{jk}-2)}
=-2\sum_{j=1}^n{{1\over\rho_j^3}(1-\cos{2\pi jk\over N})}.$$ (the
eigenvalue $\lambda_0=0$, with eigenvector $X_0=(1,1,\ldots,1)$,
disappears after the quotient by vertical translations).

To each eigenvalue $\lambda_k$ ($1\le k\le n$) of the circulant
matrix, corresponds the 4-dimensional space of solutions of the (VVE)
consisting in solutions of the following form, where the amplitude $A$
and the phase $t_0$ are parameters and $\omega_k=\sqrt{-\lambda_k}$ :
$$z(t)=\left(A\, \Re(e^{i\omega_k(t-t_0)}),A\,\Re(\zeta^ke^{i\omega_k(t-t_0)}),
   \ldots,A\,\Re(\zeta^{k(N-1)}e^{i\omega_k(t-t_0)})\right)$$ and
$$z(t)=\left(A\, \Re(e^{i\omega_k(t-t_0)}),A\,
   \Re(\bar\zeta^ke^{i\omega_k(t-t_0)}),
   \ldots,A\, \Re(\bar\zeta^{k(N-1)}e^{i\omega_k(t-t_0)})\right).$$

{\it (ii) Case $N=2n$ even:} the $4n-2$-dimensional phase space
splits into the sum of a $4(n-1)$-dimensional invariant space in which
the solutions are of the same form as in the odd case and an invariant
plane which corresponds to the eigenvalue $\lambda_n$ of the circulant
matrix. The eigenvalue $\lambda_n$ plays a special role for the sole
reason that $\zeta^n=-1$ is real and hence corresponds to a
2-dimensional space of solutions and not a 4-dimensional one.

\subsection{The vertical eigenvalues}

Partly following the method of~\cite{PW,Mo}, we prove the following
fact, which will be used in sections~\ref{sec:localCont}
and~\ref{sec:global}. 

\begin{lemma}\label{evOrder}
  The half sequence $(\lambda_k)_{1\leq k \leq N/2}$ is negative and
  decreasing.
\end{lemma}

\paragraph{Case of an odd number of bodies: $N=2n+1$.}

Using the fact that $\zeta^{N-j} = \bar\zeta^j$ and $\rho_{N-j} =
\rho_j$, and introducing $\theta = 2\pi/N$, we get
$$\lambda_k =- \frac{1}{4} \sum_{j=1}^n \frac{1-\cos
   jk\theta}{\sin^3 \frac{j\theta}{2}} \quad (1\leq k \leq n).$$
Three successive discrete derivations yield
$$\begin{array}[t]{rlll}
   \delta\lambda_k &=& \lambda_{k+1}-\lambda_k = -\frac{1}{2}
   \sum_{j=1}^n \frac{\sin j(2k+1)\frac{\theta}{2}}{\sin^2
     j\frac{\theta}{2}} &(1\leq k \leq n-1)\\
   \delta^2\lambda_k &=& \delta\lambda_{k+1} - \delta\lambda_k = -
   \sum_{j=1}^n \frac{\cos j(k+1)\theta}{\sin j\frac{\theta}{2}}
   &(1\leq k \leq n-2) \\
   \delta^3\lambda_k &=& \delta^2\lambda_{k+1} - \delta^2\lambda_k = 2
   \sum_{j=1}^n \sin j(2k+3)\frac{\theta}{2} &(1\leq k \leq n-3)\\
   &=& 2 \frac{\sin \left((2k+3)(n+1) \frac{\theta}{4}\right) \sin
     \left( (2k+3)n \frac{\theta}{4} \right)}{\sin
     (2k+3)\frac{\theta}{4}}.
\end{array}
$$

We want to show that $\lambda_k$ decreases with $k$, or that
$\delta\lambda_k$ is negative. The reason for derivating three times
is that $\delta^3\lambda_k$ is a trigonometric polynomial and that its
sign is related to the convexity of $\delta\lambda_k$.

Notice that these sequences still make sense for larger values of $k$,
provided that the denominators do not vanish. Besides, the property
that $\delta\lambda_k$ is $\leq 0$ for extremal values of $k$ is
easier to check outside the initial interval of definition:
$$\delta \lambda_0 = - \displaystyle\frac{1}{2} \sum_{j=1}^n \frac{1}{\sin
   \left(j \frac{\theta}{2}\right)} <0$$
(because $j\frac{\theta}{2} = j\frac{\pi}{2n+1} < \pi$) and
$$\delta\lambda_n = -\frac{1}{2} \sum_{j=1}^n \frac{\sin j\pi}{\sin^2
   j \frac{\theta}{2}} = 0.$$ Hence it suffices to show that $\delta
\lambda_k$ is convex with respect to $k$, i.e. $\delta^3\lambda_k \geq
0$, over $\{0,...,n-2\}$ (see figure~\ref{fig:dlambda}).

\begin{figure}[h]
    \centering
    \includegraphics[width=7cm]{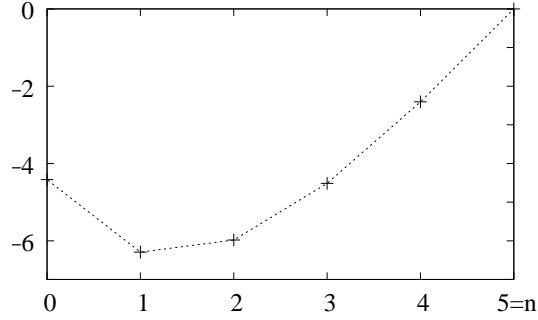}
    \caption{Graph of $\delta\lambda_k$ for $n=5$ ($2n+1=11$ bodies)}
    \label{fig:dlambda}
\end{figure}

In the above expression of $\delta^3\lambda_k$, the denominator is the
sine of
$$0 <  (2k+3) \frac{\theta}{4} \leq \frac{2n-1}{2n+1} \cdot
\frac{\pi}{2} < \frac{\pi}{2} \quad \mbox{if} \quad 0\leq k \leq
n-2,$$ and thus is $>0$. It remains to check that the numerator
$$\nu_k = \sin \left((2k+3)(n+1) \frac{\theta}{4}\right) \sin
     \left( (2k+3)n \frac{\theta}{4} \right)$$
itself is $\geq0$, which follows from linearizing:
$$\nu_k = \frac{1}{2} \cos \frac{k+\frac{3}{2}}{n+\frac{1}{2}}
\frac{\pi}{2} \geq 0 \quad \mbox{if} \quad 0\leq k \leq n-2.$$

\paragraph{Case of an even number of bodies: $N=2n$.}

In the sum defining $\lambda_k$, all terms come pairwise (indexes $j$
and $N-j$), except that of the $n$-th body. Let $\epsilon_j = 1/2$ if
$j=n$ and $\epsilon_j = 1$ otherwise. Then
$$\lambda_k =- \frac{1}{4} \sum_{j=1}^n \epsilon_j \frac{1-\cos
   jk\theta}{\sin^3 \frac{j\theta}{2}} \quad (1\leq k \leq n).$$
It follows that
$$\begin{array}[t]{rlll}
   \delta\lambda_k &=& -\frac{1}{2}
   \sum_{j=1}^n \epsilon_j \frac{\sin j(2k+1)\frac{\theta}{2}}{\sin^2
     j\frac{\theta}{2}} &(1\leq k \leq n-1)\\
   \delta^2\lambda_k &=& - \sum_{j=1}^n \epsilon_j \frac{\cos
     j(k+1)\theta}{\sin j\frac{\theta}{2}} &(1\leq k \leq n-2) \\
   \delta^3\lambda_k &=& 2 \sum_{j=1}^n \epsilon_j \sin
   j(2k+3)\frac{\theta}{2} &(1\leq k \leq n-3)\\
   &=& 2 \frac{\sin \left((2k+3)(n+1) \frac{\theta}{4}\right) \sin
     \left( (2k+3)n \frac{\theta}{4} \right)}{\sin
     (2k+3)\frac{\theta}{4}} + (-1)^k.
\end{array}
$$
Again, we will check that $\delta\lambda_k$ is $<0$ at the boundary of
some interval of integers containing $\{1,...,n-1\}$, and that
$\delta\lambda_k$ is convex inside this interval.

Similarly to the odd case, obviously we have $\delta\lambda_0 <0$ and
$\delta\lambda_{n-1/2}=0$. Unfortunately, the latter equality is not
of any direct use since the index is not an integer. Let us rather
show that
$$\delta\lambda_{n-1} = \frac{1}{2} \left( \sum_{j=1}^n \frac{(-1)^j
     \epsilon_j}{\sin^2 j \frac{\pi}{2n}} \right) \sin \frac{\pi}{2n}$$
is $<0$. It suffices to see that
$$a_n = \sum_{j=1}^n \frac{(-1)^j\epsilon_j}{\sin^2 j
   \frac{\pi}{2n}} = \sum_{j=1}^{n-1} \frac{(-1)^j}{\sin^2 j
   \frac{\pi}{2n}} + \frac{(-1)^n}{2} = \sum_{j=1}^n
\frac{(-1)^j}{\sin^2 j \frac{\pi}{2n}} + \frac{(-1)^{n+1}}{2}$$ is
$<0$. Note that in these sums if $j\leq n-1$ is odd the pair of terms
of indices $j$ and $j+1$ have a negative contribution. Hence, if $n$
is odd, the last but one given expression of $a_n$ shows that $a_n\leq
-1/2$; and, if $n$ is even, the same estimate follows from the last
expression.

It remains to show that
$$\delta^3\lambda_k = \frac{2 \sin \left((2k+3)(n+1)
     \frac{\theta}{4}\right) \sin \left( (2k+3)n \frac{\theta}{4}
   \right) + (-1)^k \sin (2k+3)\frac{\theta}{4} }{\sin
   (2k+3)\frac{\theta}{4}} $$
is positive if $1\leq k \leq n-2$. The denominator being $>0$, focus
on the numerator $\nu_k$. Setting $\alpha = (k+\frac{3}{2})
\frac{\pi}{2n}$, we get
$$\nu_k =2 \sin(n+1)\alpha \sin n\alpha + (-1)^k \sin\alpha;$$
splitting $(n+1)\alpha$ into $n\alpha + \alpha$ and partially
linearizing yields
$$\underbrace{\cos \alpha}_{>0} \left( \underbrace{(1+\cos
       2n\alpha)}_{\geq0}  + \tan \alpha \underbrace{\left( \sin
       2n\alpha + (-1)^k \right)}_{=0} \right) \geq0,$$
which proves the result.

\subsection{The symmetry group $G_{r/s}(N,k,\eta)$}
\label{sec:symmetries}

We will now analyze the symmetries of vertical eigenmodes in the full
phase space. Recall that most of the time a vertical eigenfrequency is
degenerate and has vertical multiplicity 4. We will let the amplitude
and the phase be respectively $A=1$ and $t_0=0$, and consider the most
symmetric generators of its eigenspace i.e.  $2$-frequency motions of
the form
$$x_j(t)=\left( \zeta^je^{i\omega_1t},\Re(\zeta^{\eta kj} e^{i\omega_k
     t}) \right)\in\C\times\R,\quad (j=0,\cdots, N-1), \eqno
S(N,k,\eta)$$ where
$\eta=\pm 1$ and, like in the former section,
$\zeta=e^{i\frac{2\pi}{N}}$ and $k\in\{1,...n\}$. When observed in a
frame rotating around the vertical axis with angular velocity
$\varpi$ such that
$$\frac{\omega_1-\varpi}{\omega_k} = \frac{r}{s} \in \Q,$$
the horizontal and vertical frequencies are set into resonance and the
motion becomes periodic of period $T=\frac{2\pi s}{\omega_k}$:
$$x_j(t)=\left(\zeta^je^{i \frac{r}{s}\omega_kt},
   \Re(\zeta^{\eta kj}e^{i\omega_kt})\right) \quad (j=0,...,N-1) \eqno
S_{r/s}(N,k,\eta).$$

\medskip The discrete symmetry group of such a motion is seeked as
before as a subgroup of
$$G_0=O(\R/T\Z)\times\Sigma(N)\times O(\R^3),$$
where $g=(\tau,\sigma,\rho)\in G_0$ acts naturally on the space of
$T$-periodic loops: $gx_j(t)=\rho x_{\sigma^{-1}(j)}(\tau^{-1}(t))$.

Let $G_{r/s}(N,k,\eta)$ be the stabilizer of $S_{r/s}(N,k,\eta)$. The
group structure of $G_{r/s}(N,k,\eta)$ does not depend on $r$ and it
depends on $s$ only through the fact that we are looking to the
solution during a time interval $s$ times longer than the minimal
period of the relative equilibrium in the inertial frame. We let
$(s,k)$ be the gcd of $s$ and $k$, and $s=(s,k)s',\; k=(s,k)k'$.

\begin{lemma}
   $G_{r/s}(N,k,\eta)$ is a semi-direct product of an Abelian group $H$
   of order $2Ns$ by $\Z/2\Z$. The group $H$ is an extension by
   $\Z/2\Z$ of a group $K$ which is itself an extension of $\Z/Ns'\Z$
   by $\Z/(k,s)\Z$.
\end{lemma}

\proof
\begin{suite}
\item \textit{Restriction to a subgroup $G_1$ of $G_0$.} Certainly
   elements of $G_{r/s}(N,k,\eta)$ stabilize the regular $N$-gon
   relative equilibrium (as the horizontal component of any
   infinitesimal vertical variations), as well as the cylinder of
   infinitesimal vertical variations. Hence $G_{r/s}(N,k,\eta)$ is
   contained in the subgroup $G_1$ consisting of elements
   $g=(\tau,\sigma,\rho)\in G_0$ satisfying the following conditions:
   \begin{Liste}
   \item The isometry $\rho\in O(\R^3)$ is of the form
     $\rho = (\rho_{hor},\rho_{ver}) : \C \times \R \mapsto \C
     \times
     \R$
     where
     $$\rho_{hor}(h) = e^{i 2\pi \alpha}h \quad \mbox{or} \quad
     e^{i2\pi \alpha}\bar h \quad \mbox{and} \quad \rho_{ver}(v) =
     e^{i\pi\beta}v$$ with $\alpha\in\R/\Z$ and $\beta\in \Z/2\Z$.
   \item If we set $\xi=\pm 1$ according to whether
     $\rho_{hor}(h)=e^{i 2\pi \alpha}h$ or $e^{i 2\pi \alpha}\bar
     h$,
     $\tau^{-1}(t)=\xi(t-\theta)$ with $\theta\in \R/T\Z$.
   \end{Liste}
   Hence an element $g\in G_1$ can be identified with a quintuple
   $$(\theta,\sigma,\alpha,\beta,\xi) \in \R/T\Z
   \times\Sigma(N) \times \R/\Z \times
   \Z/2\Z\times \F_2,$$ the multiplication law
   being given by $\tilde g=g'g$ with
   $$\tilde \theta=\theta'+\xi'
   \theta,\;
   \tilde\sigma=\sigma'\sigma,\;
   \tilde\alpha=\alpha'+\xi'\alpha,\;
   \tilde\beta=\beta'+\beta,\;
   \tilde \xi=\xi'\xi.$$
   {\it In the rest of this proof, let us set the time unit so that
     $\omega_k=2\pi$, i.e. $T=s$:} if $x\in S_{r/s}(N,k,\eta)$,
   $$gx_j = \left(
     \begin{array}[c]{l}
       \exp \left( 2\pi i \left[\alpha +
           \frac{1}{N}\xi \sigma^{-1}(j) +
           \frac{r}{s}(t-\theta) \right] \right)\\
       \cos \left(2\pi \left[
           \frac{\beta}{2} + \frac{\eta}{N}k\sigma^{-1}(j) +
           \xi(t-\theta) \right] \right)
     \end{array}\right).$$

\item \textit{Restriction to a subgroup $G_2$ of $G_1$.} Setting
   $\delta = \xi\sigma^{-1}(j)-j \in \Z/N\Z$, the symmetry equation
   $gx = x$ reduces to
   $$\left\{
     \begin{array}[c]{l}
       \alpha+\frac{\delta}{N}- \frac{r}{s} \theta\equiv 0\;
       \pmod{1}\\
       \xi \frac{\beta}{2} +k\eta \frac{\delta}{N} - \theta\equiv
       0\;
       \pmod{1}.
     \end{array} \right.$$
   The first
   equation shows that $\delta$ is independent of $j$,
   i.e. $\xi\sigma$ is a circular permutation. Hence, using the fact
   that $\xi\beta=\beta \pmod{1}$,
   $$
   \left\{
     \begin{array}[c]{l}
       \alpha \equiv \frac{r}{s} \theta-\frac{\delta}{N} \quad
       \pmod{1}\\
       \theta \equiv \frac{\beta}{2} + k\eta \frac{\delta}{N}
       \quad
       \pmod{1}
     \end{array}
   \right.
   $$
   These equations completely determine $\alpha\in \R/\Z$ as a
   function of $\delta$ and
   $\theta$ but $\theta \in \R/s\Z$ is only determined mod 1, as a
   function of
   $(\delta,\beta,\xi)\in\Z_N\times \Z/2\Z\times  \F_2$.
   Let
   $$G_2 = (\R/s\Z \times \Z/N\Z\times \Z/2\Z) \rtimes \F_2  =
   \{(\theta,\delta,\beta,\xi)\},$$
   where the semi-direct product is defined by the law $\tilde g = g'g$,
   with
   $$\tilde\theta = \theta'+\xi' \theta,\; \tilde\delta =
   \delta' +
   \xi'\delta, \; \tilde\beta =
   \beta'+\beta,\; \tilde\xi = \xi'\xi.$$
   Then $G_{r/s}(N,k,\eta)$ identifies with the subgroup of $G_2$
   defined by the equation $$\theta \equiv \frac{\beta}{2} + k\eta
   \frac{\delta}{N} \pmod{1}.$$

\item \textit{Group structure of $G_{r/s}(N,k,\eta)$.}

   \textit(i) The semi-direct product structure of $G_2$ goes down to
   $G_{r/s}(N,k,\eta)$. Indeed, the section $\xi\mapsto (0,0,0,\xi)$ of
   the projection $G_2\to\F_2$ takes its values in the subgroup
   $G_{r/s}(N,k,\eta)$. Hence $G_{r/s}(N,k,\eta)=H\rtimes \Z/2\Z$,
   where $H$ is the subgroup of the abelian group $\R/s\Z \times
   \Z/N\Z\times \Z/2\Z$ defined as the set of triples
   $(\theta,\delta,\beta)$ which satisfy
   $$\theta \equiv \frac{\beta}{2} + k\eta
   \frac{\delta}{N} \pmod{1}.$$
   \smallskip

   \textit(ii) Let $K$ be the kernel of the group homomorphism
   $(\theta,\delta,\beta)\mapsto \beta$ from $H$ to $\Z/2\Z$. It can be
   identified with the subgroup of $\R/s\Z \times \Z/N\Z$ defined as
   the set of pairs $(\theta,\delta)$ which satisfy
   $$\theta \equiv  k\eta
   \frac{\delta}{N} \pmod{1}.$$
   \smallskip

   \textit(iii)  One has the exact sequence
   $$0 \rightarrow \Z/Ns'\Z\rightarrow K \rightarrow
   \Z/(k,s)\Z\rightarrow 0,$$ where the first arrow sends 1 to
   $(k\eta/N,1)$ and the second one sends $(\theta,\delta)$ to the
   class of $\theta-k\eta\delta/N$.  This ends the proof of the lemma.
\end{suite}
\medskip

The precise structure of $G_{r/s}(N,k,\eta)$
depends on $s,N,k,\eta$. We study it in some
cases.

\begin{lemma}
   If one of the following conditions is satisfied, $H$ is isomorphic to
   $K\times
   \Z/2Z$:

   1) $s$ is odd

   2) $s$ and $N$ are even, $k$ is odd.
\end{lemma}
\proof

1) If $s$ is odd, the mapping
$\beta\mapsto(\frac{s\beta}{2},0,\beta)$ defines
a section of the projection from $H$ to
$\Z/2\Z$. Hence the mapping
$(\theta,\delta,\beta)\mapsto
\bigl((\theta-\frac{s\beta}{2},\delta),\beta\bigr)$
is an isomorphism from $H$ to
$K\times\Z/2\Z$.

2) If  $s=2\sigma$ is even, a section must send 1
onto an element of order 2 in $H$ not belonging
to $K$,
i.e. an element of the form $(\theta,\delta,1)$
such that there exist $a,b\in\Z$ with
$$2\delta=bN,\quad \quad
1+2k\eta\frac{\delta}{N}=1+bk\eta=as=2a\sigma.$$
From the second equation, it follows that $b$ and
$k$ must be odd. The first one then implies that
$N$ is even. If all
these conditions are satisfied, $\beta\mapsto
(\beta\frac{s}{2},\beta\frac{N}{2},\beta)$
defines a section and hence the
mapping
$(\theta,\delta,\beta)\mapsto\bigl((\theta-\beta
\frac{s}{2},\delta-\beta\frac{N}{2}),\beta\bigr)$
is an
isomorphism from $H$ to
$K\times\Z/2\Z$.

\begin{lemma}
   If $(k,s')=1$, $K$ is isomorphic to
   $\Z/Ns'\Z\times \Z/(k,s)\Z$ (and hence to $\Z/Ns\Z$ if
   $(Ns',(k,s))=1$); if $(Nk',(k,s))=1$, $K$ is isomorphic to
   $\Z/Ns\Z$.
\end{lemma}

1)  If $(k,s')=1$, the mapping which sends
$1\in\Z/(k,s)\Z$ to $(s',0)\in K$ is a section.

2) If $(Nk',(k,s))=1$, this implies $(N,(k,s))=1$, and hence the
existence of integers $l'$ and $\delta$ such that
$Nl'+\eta\delta(k,s)=1$. Let us set $l=l'k'$ and
$\theta=k\eta\delta/N+l$. We assert that the element
$(\theta,\delta)$ is of order $Ns$. Indeed, $p(\theta,\delta)=0\in
K$ if and only if there exists integers $a,b$ such that
$$pk\eta\delta/N+pl=k\eta b+pl=as,\quad p\delta=bN.$$
From its definition $\delta$ satisfies $(N,\delta)=1$. Then, the
second equation above implies the existence of an integer $p'$ such
that $p=p'N$, hence $b=p'\delta$ and the first equation above
becomes $p'\bigl((k,s)\eta \delta+Nl'\bigr)k'=as$, that is
$p'k'=as$. As $(Nk',(k,s))=1$ implies $(k',(k,s))=1$ and hence
$(k',s)=1$, there exists an integer $p''$ such that $p'=p''s$, and
hence that $p=p''Ns$.

\begin{corollary}
   When $s=1$, $G_{r/s}(N,k,\eta)$ is isomorphic to
   the direct product $D_N\times \Z/2\Z$, where
   $D_N$, of order
   $2N$, is the dihedral group. In particular,
   when $N$ is odd, it is isomorphic to $D_{2N}$.
\end{corollary}

\subsection{Invariant loops}

An $s$-periodic loop of configurations
$x(t)=\bigl(x_1(t),\cdots,x_N(t)\bigr)$ is invariant under the action
of $G_{r/s}(N,k,\eta)$ if and only if, for every
$(\theta,\delta,\beta,\xi)\in G_2$ representing an element of
$G_{r/s}(N,k,\eta)$, i.e. such that
$\theta-\frac{\beta}{2}-k\eta\frac{\delta}{N}=l\in\Z$, one has
$$\forall j\in \Z/{N\Z},\; x_j(t)=\rho
x_{\xi(j+\delta)}\bigl(\xi(t-\theta)\bigr), $$
where the action of $\rho$ on $\R^3=\C\times \R$ is defined by
$$\rho(h,z)=(e^{i2\pi\alpha} \bar h^\xi,e^{i\pi\beta} z) \quad
\mbox{with} \quad \alpha=\frac{r}{s}\theta-\frac{\delta}{N}\pmod{1},$$
where $\bar h^\xi =h$ if $\xi=+1$ and $\bar h^\xi = \bar h$ if $\xi =
-1$. \smallskip

Before looking at remarkable classes of invariant loops, let us make
some general comments:

1. Taking $(\theta=\frac{1}{2}+l,\delta=0,\beta=1,\xi=1)$ and setting
$x_j(t)=(h_j(t),z_j(t))$, we get
$$h_j(t)=e^{i2\pi\frac{r}{s}(\frac{1}{2}+l)}h_j(t-\frac{1}{2}+l),\quad
z_j(t)=-z_j(t-\frac{1}{2}+l).$$ In particular, if $s=2l+1$ and $r$ are
odd, an invariant loop possesses the Italian symmetry. This is the
case of the Hip-Hops and of chains with an odd number of lobes (see
below). \smallskip

2. Taking $(\theta=0,\delta=0,\beta=0,\xi=-1)$, we get
$$h_j(t)=\bar h_{-j}(-t),\quad z_j(t)=z_{-j}(-t).$$
In particular, when $t=0$, the configuration is always symmetric with
respect to the vertical plane containing the first coordinate axis.
\smallskip

3. Taking $(\theta=\frac{k\eta}{N},\delta=1,\beta=0,\xi=1)$, we get
$$h_j(t)= e^{i2\pi\frac{A}{N}}h_{j+1}(t-\frac{k\eta}{N}),\quad
z_j(t)=z_{j+1}(t-\frac{k\eta}{N}).$$
where $A=\frac{r}{s}k\eta-1$.
\medskip

More generally, the smallest non zero value of
$\theta=\frac{\beta}{2}+k\eta\frac{\delta}{N}+l =
\frac{N(\beta+2l)+2k\eta\delta}{2N}$
is
$\theta_{0}=\frac{(N,2k)}{2N}$ because $\beta+2l$
and $\delta$ are arbitrary integers.

\medskip One can distinguish two cases in the action of
$G_{r/s}(N,k,\eta)$ on the space of loops of $N$-body configurations.
In the first one, exemplified by the $P_{12}$ family for 3 bodies, no
a priori spatial symmetry of the configuration exists for all times,
and even $\alpha=0$ for any element of the group. In this case, the
group acts on the full space of similitude classes of $N$-body
configurations.  The condition is that for any $l\in\Z,\,
\delta\in\Z/\N\Z,\, \beta\in\Z,$
$\frac{r}{s}\bigl[\frac{\beta}{2}+k\eta\frac{\delta}{N}+l\bigr]
-\frac{\delta}{N}=0\;\pmod{1}.$ This is equivalent to $s=1,\, r=2r',\,
2r'k\eta-1=0\pmod{N}$ (which implies that $N$ is odd).  The case
$\frac{r}{s}=2r'=N-1, k=1,\eta=-1$ corresponds to the ``unchained
polygons'' with $N-1$ lobes.

In the second one, exemplified by the Hip-Hop family for 4 bodies,
such a symmetry does exist: there is a group element such that
$\theta=0,\, \delta=-1,\, \alpha=1/N$.  In this case, the group acts
on a subspace of the space of similitude classes of $N$-body
configurations.  The condition is that
$\frac{\beta}{2}-\frac{k\eta}{N}=0\pmod{1}$ which, because
$k\leq\frac{N}{2}$) implies $\beta=1$ and $k\eta=\frac{N}{2}$ (the
choice of $\eta=\pm 1$ is then immaterial because
$e^{k\frac{2\pi}{N}}=e^{-k\frac{2\pi}{N}}$, so we set $\eta=1$). The
$N=2N'$-body Hip-Hops studied in \cite{TV,BCPS} fall into this
category; the group element is
$(\theta=0,\delta=-1,\xi=1,\alpha=\frac{1}{N})$, which sends one body
onto the next one by a rotation of $\frac{2\pi}{N}$ followed by a
change of sign in the vertical component. We now study some types of
symmetric loops in the two cases.

\paragraph{The choreographic symmetry} (bodies gather by pairs, at
least, following the same curve in space): this corresponds to a group
element $g$ with $\xi=1, \delta \neq 0, \alpha=0,\beta=0$. The
equations defining the symmetry group become
$$\left\{
   \begin{array}[c]{l}
     \frac{\delta}{N}- \frac{r}{s} \theta\equiv 0\; \pmod{1}\\
     k\eta \frac{\delta}{N} - \theta\equiv 0\;
     \pmod{1}.
   \end{array} \right.$$

Choreographies can be \emph{simple} or \emph{multiple}. They are
\emph{simple} if all the bodies lie on the same curve, i.e. if in
addition $\sigma$ has a unique cycle:
$$(\delta,N)=1.$$

\begin{lemma}
   $S_{r/s}(N,k,\eta)$ is a simple choreography if and only if
   $$s-k\eta r=0\; \pmod{N}.$$
\end{lemma}

\begin{proof} We rewrite more explicitely the conditions for being a
   choreography: there exists integers $l,m$ such that
   $$\left\{
     \begin{array}[c]{l}
       s\delta - Nr\theta=lsN\\
       k\eta \delta - N\theta=mN.
     \end{array} \right.$$
   This is equivalent to
   $$\left\{
     \begin{array}[c]{l}
       (s-k\eta r)\delta=(ls-mr)N\\
       k\eta \delta - N\theta=mN.
     \end{array} \right.$$
   Due to the condition of simple choreography $(\delta,N)=1$, this
   is equivalent to the existence of an integer $a$ such that
   $$\left\{
     \begin{array}[c]{l}
       ls-mr=a\delta\\
       s-k\eta r=aN\\
       k\eta \delta - N\theta=mN.
     \end{array} \right.$$
   This proves the ``only if'' part of the lemma.

   For the ``if'' part, let us suppose that there exists
   an integer $a$ such that $s-k\eta r=aN$. As
   $(r,s)=1$, we can choose $l,m$ such that
   $ls+mr=a$.
   We then define
   $$g =
   \left(\theta=\frac{k\eta}{N}-m,\delta=1,\xi=1,\alpha=0,\beta=0\right)\in
   G_{r/s}(N,k,\eta),$$
   which completes the proof.
\end{proof}

\subparagraph{Remark} The condition of simple choreography could of
course have been obtained directly from the formulae for invariant
loops or by writing that for all $j$, one has $\bigl(\zeta^{\eta
  kj}e^{i\omega_kt}\bigr)^{\frac{r}{s}} =
\zeta^je^{i\frac{r}{s}\omega_kt}$, which expresses that all the bodies
belong to the curve described by the first.

\begin{corollary}\label{cor:dense}
   For any $N,k,\eta$, the set of values of $r/s$ for which
   $S_{r/s}(N,k,\eta)$ is a simple choreography is dense in $\R$.
\end{corollary}

\begin{proof} The integers $N,k,\eta$ being given, $r,s,a$ must be
   such that
   $$(r,s)=1\quad\hbox{and}\quad s-k\eta r=aN.$$
   The primality condition is equivalent to the existence of integers
   $l,m$ such that $ls+mr=1$ , that is $(m+lk\eta)r+l(aN)=1$. As
   $(m+lk\eta,l)$ can be an arbitrary pair of integers, this is
   equivalent to $(r,aN)=1$. But the density of the $r/s$ is
   equivalent to the density of $s/r=k\eta+aN/r$ and hence to the
   density of the $aN/r$ or equivalently of the $r/aN$, with the
   unique condition that $(r,aN)=1$.  As the irreducible fractions of
   the form $r/N^i$ are dense, the corollary follows.
\end{proof}






\paragraph{Eights and maximal chains symmetries}

The two choreographic cases with extreme values of $k$ are
$G_2(N=2n+1,n,-1)$ (which we call the \emph{Eight} symmetry) and
$G_{N-1}(N,1,-1)$ (resp. the \emph{maximal chain} symmetry). In the
first case, the group element $(\theta=\frac{1}{2N},\delta=1, \,
\beta=1,\, \xi=1)$ leads to
$$x_j(t)=x_{j+1}(t-\frac{1}{2N}),\; z_j(t)=-z_{j+1}(t-\frac{1}{2N});$$
in the second case, the group element
$(\theta=\frac{1}{2N},\delta=n, \, \beta=1,\,
\xi=1)$ leads to
$$x_j(t)=x_{j+n}(t-\frac{1}{2N}),\; z_j(t)=-z_{j+n}(t-\frac{1}{2N}).$$
Solutions invariant under such a symmetry are called \emph{unchained
  polygons} because the horizontal rotation of the regular $N$-gon is
unfolded in the vertical direction, into a choreographic chain with 2
and $N-1$ lobes respectively. Examples are shown on
figures~\ref{fig:vve}.2 and~\ref{fig:vve}.1.

\paragraph{The Hip-Hop symmetry}

Writing $k=(N,k)\tilde k,\, N=(N,k)\tilde N$, we notice that
$$\zeta^{\eta kj_1}=\zeta^{\eta kj_2}\quad\hbox{if
   and only if}\quad  j_2-j_1\quad\hbox{is a
   multiple of }\;
\tilde N.$$
Hence the bodies are divided into $\tilde N$ regular $(N,k)$-gons which
remain horizontal at each instant. Independantly
of the chosen rotating frame, the
motion is invariant under a group element $g$ of the form
$$g=\left(\theta=0,\, \delta=\tilde N,\,
  \alpha=-\frac{1}{(N,k)},\, \beta=0\right)\in G_{1,1}(N,k,\eta).$$ If
we make the simplest possible choice $r/s=1$, we get a \emph {Hip-Hop
  motion}. The name is a reference to the original Hip-Hop solution
where the two diagonals of a square stay horizontal while undergoing
vertical oscillations (see figure~\ref{fig:vve}.6). Note that this
symmetry can be combined with a choreographic symmetry.

\medskip A few examples of Lyapunov cylinders are shown on
figure~\ref{fig:vve}. They are respectively tangent at the relative
equilibrium to the following Lyapunov families (the terminology will
be explained in section~\ref{sec:global}):

\begin{tabular}[t]{ll}
  1. $S_4(5,1,-1)$ &Maximal unchained polygon ($\ell=4$)\\
  2. $S_2(5,2,-1)$ &Five-body Eight\\
  3. $S_3(4,1,-1)$ &Unchained polygon (Gerver family)\\
  4. $S_{3/2}(4,2,1)$ &Terracini-Venturelli choreography (Hip-Hop
  family)\\ 
  5. $S_1(6,1,-1)$ &Six-body Hip-Hop\\
  6. $S_2(6,2,1)$ &Yet another six-body Hip-Hop.
\end{tabular}

\begin{figure}[h]
    \centering
    \includegraphics[width=8cm]{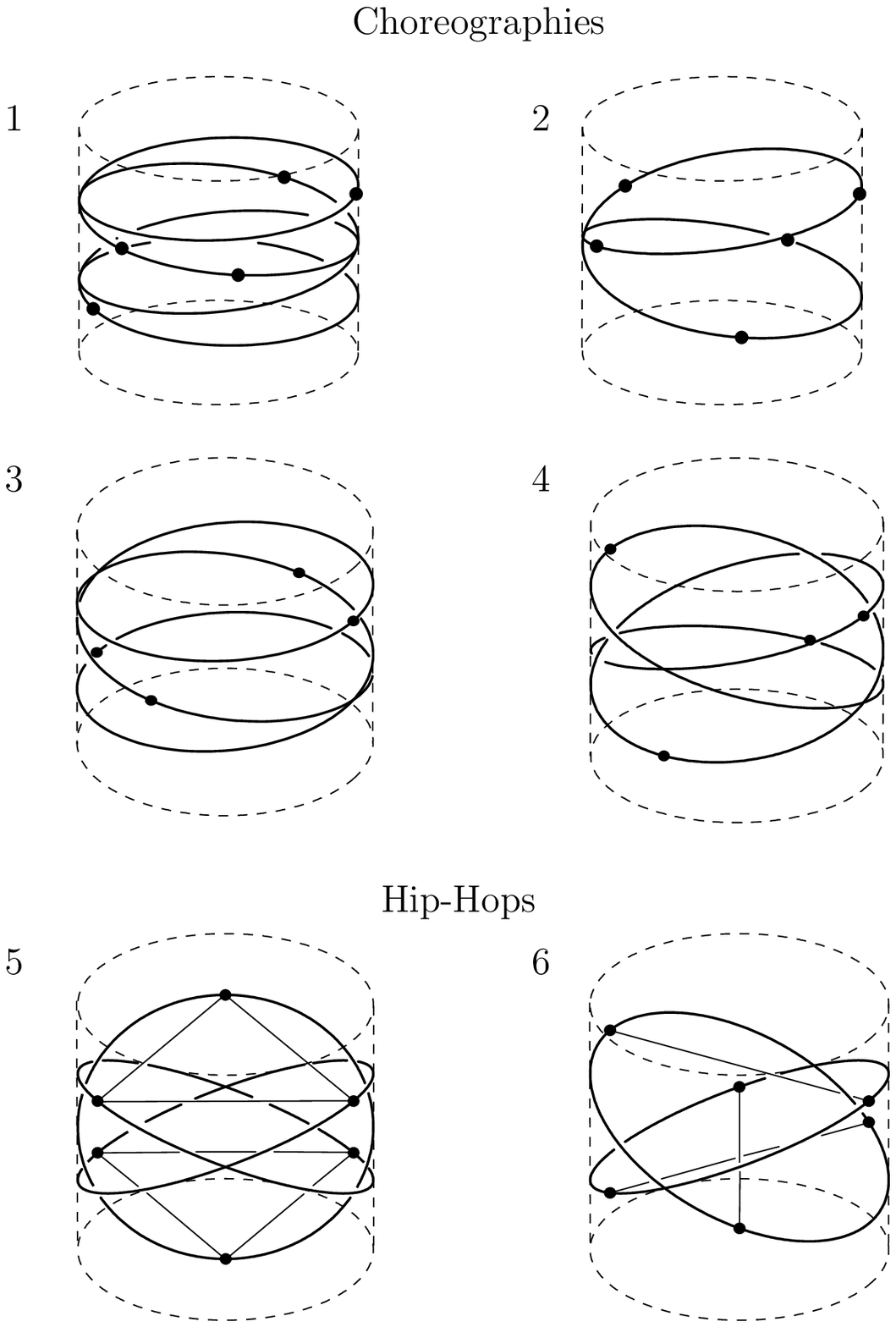}
    \caption{Some examples of first order solutions having special
      symmetries in a resonant rotating frame.}
    \label{fig:vve}
\end{figure}



\subsection{Isomorphic symmetries}
\label{sec:isomorphisms}

In order to better understand minimizers of the Lagrangian action
among loops which in the rotating frame are invariant under the
$G_{r/s}(N,k,\eta)$ action described above, we study the cases when
the actions of two groups $G_{r/s}(N,k,\eta)$ and
$G_{r'/s'}(N,k',\eta')$ coincide up to a relabelling $j\mapsto
j'=\Sigma(j)$ of the bodies. A first condition is that the periods
coincide, i.e. $s'=s$.

{\it We will stick for simplicity to the case $s=s'=1$.}
\smallskip

Let as usual $\zeta=e^{2\pi i/N}$; the horizontal relative equilibrium
whose $j'$th body's motion is $h'_{j'}(t)=\zeta^{j'}e^{2\pi ir't}$ is
$G_{r'}(N,k',\eta')$-symmetric. Hence a necessary condition for the
two actions to coincide up to the relabelling $\Sigma$ is that the horizontal
relative equilibrium whose $j$th body's motion is
$$h_{j}(t)=h'_{\Sigma(j)}(t)=\zeta^{\Sigma(j)}e^{2\pi ir't}$$
be $G_{r}(N,k,\eta)$-symmetric. This means that for any
$(\delta,\beta,\xi)\in G_{r}(N,k,\eta)$,
$$e^{2\pi i\alpha}\zeta^{\xi\Sigma(\xi(j+\delta))}e^{2\pi ir'\xi^2(t-\theta)}
=\zeta^{\Sigma(j)}e^{2\pi ir't}.$$
Using the identities
$\theta=\frac{\beta}{2}+k\eta\frac{\delta}{N}\pmod{1}$ and
$\alpha=\frac{r}{s}\theta-\frac{\delta}{N}\pmod{1}$, this becomes
$$(r-r')\bigl[\frac{\beta}{2}+k\eta\frac{\delta}{N}\bigr]-\frac{\delta}{N}
+\frac{1}{N}\bigl[\xi\Sigma(\xi(j+\delta))-\Sigma(j)\bigr]=0\pmod{1}.$$
$\delta=0,\beta=1,\xi=1$, we get a first necessary condition 
$$\mbox{1)} \quad r-r'=2p\in2\Z.$$
Taking then $\delta=0,\beta=0$, we get $\xi\Sigma(\xi
j)=\Sigma(j)\pmod{N}$. In particular, if $\xi=-1$ and $j=0$ we get
$\Sigma(0)=0\pmod{N}$.  Finally, taking $\delta=1,\beta=0$, we get
$\Sigma(j+1)-\Sigma(j)=1-2pk\eta\pmod{N}$, hence
$\Sigma(j)=(1-2pk\eta)j\pmod{N}$.  Interchanging the roles of the two
groups we see that necessarily
$\Sigma^{-1}(j')=(1+2pk'\eta')j'\pmod{N}$. This yields a second
necessary condition
$$\mbox{2)}\quad (1-2pk\eta) (1+2pk\eta') = 1 \pmod{N}.$$
It turns out that these necessary conditions are essentially
sufficient.  Indeed, it is enough to replace the condition 2), which
can be written $-2p[-k\eta+k'\eta'-2pk\eta k'\eta']=0\pmod{N}$ by the
slightly stronger condition 2') below:

\begin{proposition}
   A necessary and sufficient condition for the existence of a
   permutation $\mathfrak{S}$ of $\Z/N\Z$ such that for any
   $G_{r'}(N,k',\eta')$-invariant loop $x'(t)$ whose $j'$th body motion
   is $(h'_{j'}(t),z'_{j'}(t))$, the loop $x(t)$ whose $j$th body
   motion is
   $(h_j(t),z_j(t)):=\bigl(h'_{\mathfrak{S}(j)}(t),z'_{\mathfrak{S}(j)}(t)\bigr)$
   be $G_{r}(N,k,\eta)$-invariant is that the following conditions be
   satisfied:

   1) $r-r'=2p\in2\Z$,

   2') $-k\eta+k'\eta'-2pk\eta k'\eta'=0\pmod{N}$.

   The permutation $\mathfrak{S}$ and its inverse are then
   respectively $j\mapsto j'=(1-2pk\eta)j$ and $j'\mapsto
   j=(1+2pk'\eta')j'$ in $\Z/N\Z$.
\end{proposition}

\begin{proof} We already know that the asserted form of the
  permutation $\mathfrak{S}$ is necessary. What remains is a direct
  computation: in order to show the invariance of the loop $x(t)$
  under the element of $G_{r}(N,k,\eta)$ represented by
  $(\delta,\beta,\xi)$, one must find an element of
  $G_{r'}(N,k',\eta')$ repesented by $(\delta',\beta',\xi')$ such that
  $$e^{2\pi
    i(\alpha-\alpha')}
  \overline{h'_{(1-2pk\eta)\xi(j+\delta)}(\xi(t-\theta))}^{\xi}
  =\overline{h'_{(1-2pk\eta)\xi'j+\xi'\delta'}(\xi'(t-\theta'))}^{\xi'}$$
  and
  $$e^{\pi i(\beta-\beta')}z'_{(1-2pk\eta)\xi(j+\delta)}(\xi(t-\theta))
  =z'_{(1-2pk\eta)\xi'j+\xi'\delta'}(\xi'(t-\theta')).$$

  Straightforward computations show that this reduces to condition
  2'), and that the unique solution is
  $$\delta'=(1-2pk\eta)\delta,\; \beta'=\beta,\; \xi'=\xi,$$
\end{proof}
\medskip

\paragraph{Examples} We note $G\equiv G'$ for two groups which satisfy
the conditions of the proposition. We have

$G_{N-1}(N,1,-1) \equiv G_{-(N-1)}(N,1,1)$, with $\mathfrak{S}(j) = -j
\pmod{N}$.

$G_2(N,n,-1) \equiv G_{-2}(N,n,1)$, with $\mathfrak{S}(j) = -j
\pmod{N=2n+1}$.

$G_{N-1}(N,1,-1)\equiv G_2(N,n,-1)$ for $N=2n+1$. This is the
isomorphism between the symmetries of the maximal chains (see
\ref{sec:maximalChain}) and the symmetries of the Eight with an odd
number of bodies. Here $k\eta=-1,\; k\eta'=-n,\; p=n-1,\;
\mathfrak{S}(j)=(2n-1)j=-2j\pmod{N}.$

$G_{3}(4,1,-1)\equiv G_1(4,1,1)$. This is the isomorphism between the
symmetries of the Gerver solution (see~\ref{sec:C}) and the symmetries of
the relative equilibrium of the square. Here $k\eta=-1,\; k\eta'=1,\;
p=1,\; \mathfrak{S}(j)=3j\pmod{4}=-j\pmod{4}.$

$G_{3}(5,2,1)\equiv G_1(5,1,1)$. This is the isomorphism between the
symmetries of the 3 lobes chain for 5 bodies (see ....) and the
symmetries of the relative equilibrium of the pentagon. Here
$k\eta=2,\; k\eta'=1,\; p=1,\;
\mathfrak{S}(j)=-3j\pmod{5}=2j\pmod{5}.$ 

\paragraph{Remark} The difference between conditions 2) and 2') is
most easily understood on the example of the groups $G_1(4,1,1)$ and
$G_1(4,2,\pm 1)$. The horizontal relative equilibrium of the square
possesses both symmetries but the two actions are not isomorphic:
indeed, relative equilibria in an inclined plane are symmetric under
$G_1(4,1,1)$ but not under the symmetry group of the Hip-Hop
$G_1(4,2,\pm 1)$. This comes from the fact that condition 2) is
trivially satisfied because $p=0$ while condition 2') is not.

\section{Local continuation}
\label{sec:localCont}

After reduction by translations and rotations, the
$(6N-10)$-dimensional reduced phase space is the sum of a
$(4N-6)$-dimensional ``horizontal'' subspace and a
$(2N-4)$-dimensional ``vertical'' subspace, both invariant under the
linearized flow. The eigenvalues of this flow are that of the
variational equation, the only difference being that the generalized
eigenspace corresponding to $\pm i\omega_1$ has lost four dimensions,
two horizontally (corresponding to fixing the value of the angular
momentum and to rotations around it), and two vertically
(corresponding to rotations around a horizontal axis). \medskip

We have already studied the vertical part of the spectrum. A thourough
study of the horizontal part appears in~\cite{Mo}. See
figures~\ref{fig:spectrum} and~\ref{fig:horizontalSpectrum}. Note that
purely imaginary horizontal eigenvalues exist only for $N\geq 6$
bodies.

\begin{figure}[h]
    \centering
    \includegraphics[width=12cm]{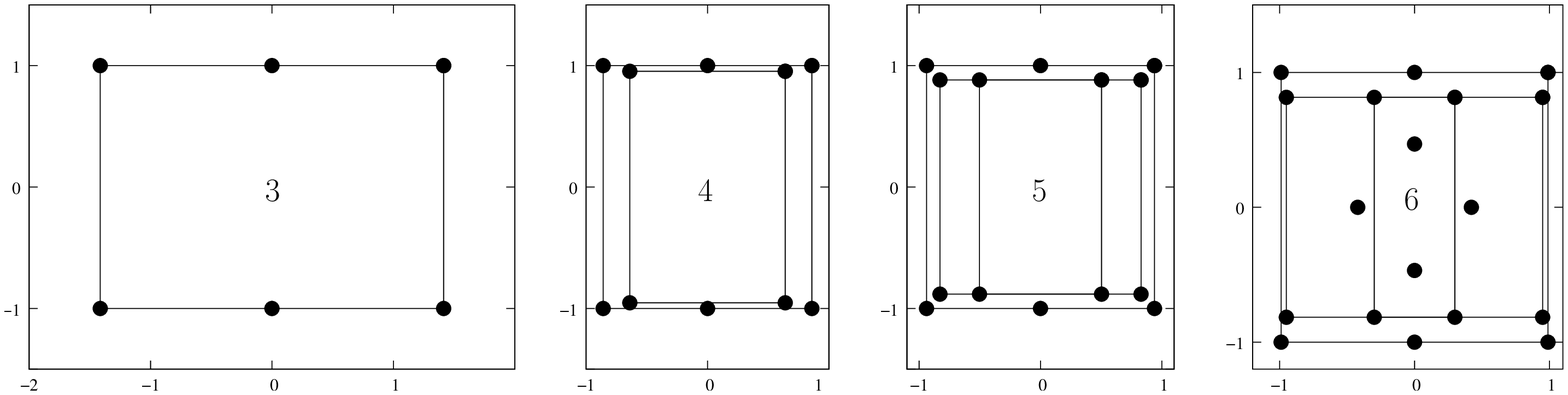}
    \caption{Horizontal spectrum for three to six bodies}
    \label{fig:spectrum}
\end{figure}

\begin{figure}[h]
    \centering
    \includegraphics[width=8cm]{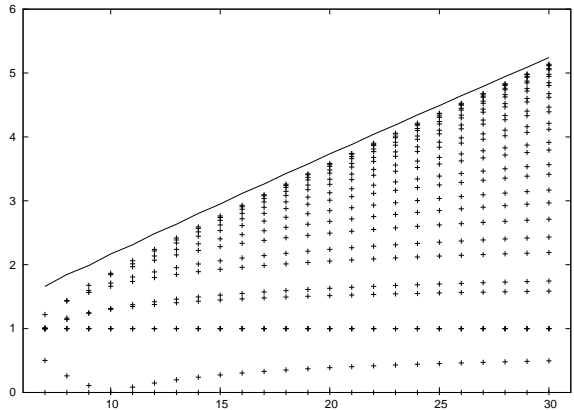}
    \caption{Imaginary parts of purely imaginary horizontal
      eigenvalues of the regular $N$-gon, as functions of $N \in \{7, 
      ..., 30\}$ are upper bounded by $\omega_n$ (solid line).}
    \label{fig:horizontalSpectrum}
\end{figure}

Recall that for each vertical frequency $\omega_k$ the vertical
eigenspace, most of the time, has dimension~4. The only exception is
when $N=2n$ is even and $k=n$; if moreover no other vertical or
horizontal frequency is an integer multiple of $\omega_n$, Lyapunov's
classical theorem implies the local existence and uniqueness of the
Lyapunov family. In the other cases, one can try to get existence from
the following theorem of Weinstein, here in the form given by Moser,
of whom we have kept the notations.

\begin{theorem}[{\cite[p.~743]{Mos}}]
  Let $H$ be a Hamiltonian on $R^{2n}$ with an equilibrium point at
  the origin, and $C$ be the linearization of the Hamiltonian vector
  field at the origin. Assume that $\R^{2n} = E \oplus F$, where $E$
  and $F$ are invariant subspaces of $C$ such that all solutions of
  $C$ in $E$ have the same period $T>0$ while no nontrivial solution
  in $F$ has this period. Moreover, assume that the Hessian $D^2H(0)$
  restricted to $E$ is positive definite. Then, for sufficiently small
  $\varepsilon$, on each energy surface $H(z) = H(0) + \varepsilon^2$
  the number of periodic orbits of $H$ is at least $\frac{1}{2}
  \mbox{dim} \, E$.
\end{theorem}

We will see in the next section that the theorem applies to the
vertical frequencies provided that no unexpected resonance occurs. The
cases $N=3$ and $N=2n$ with Hip-Hop symmetry are respectively studied
in \cite{CF2} and \cite{BCPS} (see also \cite{MS} for the case of the
regular $n$-gon with a central mass at its center).

\subsection{Partial convexity of the energy}

In order to apply the Weinstein-Moser theorem to some vertical
frequency $\omega_\ell$, one must check that the energy level sets,
restricted to the space of $2\pi/\omega_\ell$-periodic solutions of
the linearized vector field, are compact in the neighborhood of the
relative equilibrium.

Let $\mathcal{V}_\ell$ be the vertical eigenspace of the frequency
$\omega_\ell$ ($1 \leq \ell \leq N/2$), and $\mathcal{H}_1$ be the
plane tangent to the homographic motions. (Recall that the total
eigenspace of $\omega_1$ contains $\mathcal{V}_1 \oplus
\mathcal{H}_1$.) If no other frequency, horizontal or vertical, is an
integer multiple of (possibly equal to) $\omega_\ell$, in order to
apply the Weinstein-Moser theorem it suffices to prove that the energy
is convex on $\mathcal{V}_1 \oplus \mathcal{H}_1$ for $\ell=1$ and on
$\mathcal{V}_\ell$ for $2 \leq \ell \leq N/2$.  Below we will prove
the fact that the quadratic part $H$ is positive definite on the whole
vector space
$$\mathcal{F} = \mathcal{H}_1 \bigoplus \oplus_{1\leq \ell \leq N/2}
\mathcal{V}_\ell.$$ 

The first cases, studied in the following section, all satisfy the
required non resonance condition. Furthermore, numerical experiment
suggests that the purely imaginary horizontal eigenvalues, in general,
cannot resonate with the vertical eigenvalue $i \omega_n$, at least,
for they are smaller in module (see
figure~\ref{fig:horizontalSpectrum}, where the horizontal frequencies
have been computed using the factorization of the stability polynomial
which is described in~\cite{Mo}).\footnote{Note however that the
  largest horizontal frequency seems asymptotic to $\omega_n$, when
  $N$ tends to infinity.} When this is true, the proposition below and
the Weinstein-Moser theorem show in a weak sense the local existence
of Lyapunov families associated with $\omega_n$, in particular.

\begin{proposition}
  After reduction by rotations, the restriction of the quadratic part
  of the energy to $\mathcal{F}$ is definite positive.
\end{proposition}

We first prove two lemmas. The frequency $\omega_1$ plays a special
role, for its eigenspace always includes a horizontal plane, namely
the plane tangent to the homographic motions.

\begin{lemma}\label{lm:E1}
  After reduction by rotations, the restriction of the energy to
  $\mathcal{V}_1 \oplus \mathcal{H}_1$ is definite positive.
\end{lemma}

This lemma allows to apply the Weinstein-Moser
theorem~\cite[p.~743]{Mos} to the frequency $\omega_1$ provided that
the reduced linearized equation have no solution of period
$2\pi/\omega_1$ outside $\mathcal{V}_1 \oplus \mathcal{H}_1$. It is a
straightforward generalization of \cite[lemma~2.1]{CF2}.

\begin{proof}
  For the same reason as in \cite[lemma~2.1]{CF2}, the lift of
  $\mathcal{V}_1 \oplus \mathcal{H}_1$ to the non-reduced phase space
  is tangent to the submanifold
  $$\mathcal{E}_1 = \{(x,\dot x)\in (\R^3)^N \times (\R^3)^N, \; \exists
  \rho, \sigma>0 \; \exists R,S\in SO_3, \; x = \rho R C, \; \dot x =
  \sigma S \J C \}$$ where, as before, $C$ is the central
  configuration $(1,\zeta, ..., \zeta^{N-1})$ and $\J$ is the
  horizontal-rotation operator.  It suffices to prove that the
  restriction of $H$ to $\mathcal{E}_1$ is definite. For the sake of
  completeness, we shortly explain the computation.

  Before reduction, $(\rho,\sigma,R,S)\in(\R_+)^2\times SO(3)\times
  SO(3)$ are (generalized) coordinates on $\mathcal{E}_1$ and the
  restriction of $H$ is
  $$H|_{\mathcal{E}_1}=\frac{\sigma^2}{2}\sum_{0\leq j 
    \leq N-1}||\omega_1 \J \zeta^j||^2 -\frac{1}{\rho} \sum_{0\leq j<k
    \leq N-1} \frac{1}{||\zeta^j- \zeta^k||}.$$ 
  As in~\ref{sec:known}, the mass dot product by $C$ of Newton's
  equation shows that 
  $$N \omega_1^2 = \sum_{0\leq j<k\leq N-1} \frac{1}{||\zeta^j-
    \zeta^k||}.$$ Hence
  $$H|_{\mathcal{E}_1} = N \omega_1^2
  \left(\frac{\sigma^2}{2}-\frac{1}{\rho} \right).$$
  
  We compute the reduced system by first quotienting by the full group
  $SO(3)$ and then fixing the length of the angular momentum $L$.
  This amounts to replacing $(\rho,\sigma,R,S)$ by
  $(\rho,\sigma,R^{-1}S)$ and imposing the relation
  $$\rho\sigma ||\sum_{i=0}^{N-1} (\zeta^j,0) \wedge R^{-1}S(i \omega_1
  \zeta^j,0)||:=||L||.$$ Any element of a neighborhood of the Identity
  in $SO(3)$ can be uniquely written as $\exp A$, where $A$ is an
  antisymmetric $3\times 3$ matrix. In particular,
  $$R^{-1}S=\exp\begin{pmatrix}
    0&-c&b\\
    c&0&-a\\
    -b&a&0\end{pmatrix}$$
  $$=Id+\begin{pmatrix}
    0&-c&b\\
    c&0&-a\\
    -b&a&0\end{pmatrix}+\frac{1}{2}\begin{pmatrix}
    -(b^2+c^2)&ab&ac\\
    ab&-(c^2+a^2)&bc\\
    ac&bc&-(a^2+b^2)\end{pmatrix}+\cdots,$$ where the dots represent
  terms of order higher than or equal to 3 in $a,b,c$. Locally, the
  angular momentum equals
  $$L = \rho\sigma\omega_1 \left(
    \begin{pmatrix}
      0\\ 0\\ N
    \end{pmatrix} + \dfrac{N}{2} 
    \begin{pmatrix}
      b \\-a\\ 0
    \end{pmatrix} - \dfrac{N}{4}
    \begin{pmatrix}
      ac\\ bc\\ a^2 + b^2 + 2c^2
    \end{pmatrix} + \cdots \right),$$
  and its norm
  $$\|L \| = \rho\sigma\omega_1 N \left[ 1 - \frac{1}{2} \left(
      a^2 + b^2 + 4 c^2 \right) + \cdots \right].$$ 
  In the coordinates $(a,b,c,d=\sigma-1)$ on the space $\mathcal{E}_1$
  reduced by the rotations, the reduced Hamiltonian is
  $$\dfrac{N\omega_1^2}{2} \left( -1 + a^2 + b^2 + 4c^2 + d^2
  \right) + \cdots,$$ which proves the lemma.
\end{proof}

We will now study the other eigenspaces $\mathcal{V}_\ell$. Usually
$\mathcal{V}_\ell$ is $4$-dimensional, generated by the sum of the
Lyapunov ``cylinders'' $S(N,\ell,\pm 1)$. As we will see in examples,
the first interesting case is that of $N=5$, $\ell=2$. The case where
$N=2n$ is even and $\ell=n$ is special:
$\mathcal{V}_\ell=\mathcal{V}_n$ is then $2$-dimensional (and the
Lyapunov center theorem applies with no need of convexity property,
provided that no other frequency is an integer multiple of
$\omega_n$).

\begin{lemma}\label{lm:El}
  If $2\leq \ell \leq N/2$, $\mathcal{V}_\ell$ lies in a submanifold
  of fixed vertical angular momentum and in the reduced space the
  restriction of the energy to $\mathcal{V}_\ell$ is definite
  positive.
\end{lemma}

\begin{proof}
  A point on $\mathcal{V}_\ell$ can be uniquely represented by
  $(x,\dot x)$ with
  \begin{eqnarray*}
    x_j &=&(\cos j\theta, \sin j\theta, a \cos j\ell\theta + b\sin
    j\ell\theta)\\
    \dot x_j &=&(-\omega_1 \sin j\theta, \omega_1 \cos j\theta,
    c\omega_\ell \sin j\ell\theta + d \omega_\ell \cos j\ell\theta),
  \end{eqnarray*}
  with $\theta = 2\pi/N$, as before, and where $a,b,c,d\in \R$;
  $(a,b,c,d)$ is a coordinate system on $\mathcal{V}_\ell$, except in
  the case $N=2n$ and $\ell=n$, where $(a,c)$ alone is a coordinate
  system (with $b=d=0$).

  The statement on the angular momentum comes from the fact that the
  angular momentum of the $j$-th body is
  $$L_j = 
  \begin{pmatrix}
    \cos j\theta \\ \sin j\theta \\ a \cos j\ell\theta + b \sin
    j\ell\theta 
  \end{pmatrix}
  \wedge
  \begin{pmatrix}
    -\omega_1 \sin j\theta\\
    \omega_1 \cos j\theta\\
    \omega_\ell c\cos j\ell\theta + \omega_\ell d \sin j\ell\theta
  \end{pmatrix};$$ its vertical component depends only on the radius
  of the horizontal circle, and its two horizontal components boil
  down to degree-one trigonometric polynomials in $j(\ell+1)\theta$
  and $j(\ell-1)\theta$, whose sums with respect to $j \in
  \{0,...,N-1\}$ are zero if $\ell > 1$.

  We now need to compute the quadratic part of $H$ with respect
  to $a,b,c,d$. A straightforward computation shows that the kinetic
  energy is
  $$\frac{K}{2} = \frac{1}{2} \sum_{j=0}^{N-1} \|\dot x_j\|^2 =
  \frac{N}{2} \left( \omega_1^2 + \omega_\ell^2 (c^2 + d^2) \right),$$
  hence positive definite with respect to $c$ and $d$. 
  On the other hand, the potential part depends only on $a$ and
  $b$ and, at the second order,
  \begin{eqnarray*}    
  -U &=& - \sum_{0\leq j < k \leq N-1} \dfrac{1}{\| x_k - x_j \|}\\
  &=& - \sum_{j<k} \dfrac{1}{4 |\sin (j-k) \frac{\ell \theta}{2}|}
  \left(   
    \begin{array}[c]{l}
      1 -  a^2 \sin^2 (j+k) \frac{\ell\theta}{2} -   
      b^2 \cos^2 (j+k) \frac{\ell\theta}{2} +\\
      2 ab \sin (j+k) \frac{\ell\theta}{2} \cos (j+k)
      \frac{\ell\theta}{2} + \cdots
    \end{array} \right)\\
  &=& -N \omega_1^2 + \sum_{j<k} \dfrac{\left( a \sin (j+k)
      \frac{\ell\theta}{2} - b \cos (j+k) \frac{\ell\theta}{2}
    \right)^2}{4 |\sin (j-k) \frac{\ell \theta}{2}|} + \cdots,
  \end{eqnarray*}
  which is positive, definite if $\ell \neq N/2$ and null if $\ell =
  N/2$. 
\end{proof}

We now complete the proof of the proposition.

\begin{proof}
  We want to show that in the neighborhood of the Lagrange relative
  equilibrium the quadratic part of the energy $H$ is positive
  definite on the space
  $$\mathcal{F} = \mathcal{H}_1 \bigoplus \oplus_{1\leq \ell \leq N/2}
  \mathcal{V}_\ell.$$

  The two lemmas above show that the quadratic part of the energy $H$
  is positive definite in restriction to $\mathcal{H}_1\oplus
  \mathcal{V}_1$ and to each $\mathcal{V}_\ell$, $\ell \geq 2$. As
  eigenspaces of a symplectic linear operator (namely, the
  linearization of the reduced vector field) are orthogonal, it
  follows that the aforementioned spaces are pairwise symplectically
  orthogonal. Since additionally they are invariant by the flow of the
  quadratic part $Q$ of $H|_\mathcal{F}$, $Q$ does not have
  off-diagonal terms.  Hence it is positive definite.
\end{proof}

\subsection{The first cases}

We will now describe the simplest families. It is interesting that new
qualitative features appear for each successive value of $N$ from~3
to~6.

A natural name for each Lyapunov family we study would be the label
$S(N,k,\eta)$ of the tangent Lyapunov cylinder in the inertial frame.
On the other hand, for the sake of concreteness we like to look at a
given family in rotating frames where some fixed symmetry
$G_{r/s}(N,k,\eta)$ holds. This amounts to considering the Lyapunov
family as a deformation of the Lyapunov cylinder $S_{r/s}(N,k,\eta)$.
Among the infinitely many possible choices, we have selected the
symmetry of the most remarkable periodic orbit of the (global) family
in the inertial frame. Of course, this is a subjective choice and a
family may bear several names.

\paragraph{Three bodies}


The case of $N=3$ bodies is fully analyzed in~\cite{CF2}, which we now
summarize. The eigenvalues of the linearized reduced vector field
(which can be computed following~\cite{Mo}) are:

\begin{tabular}[t]{ll}
   Horizontally: &$\pm i \omega_1, \left(\pm \sqrt{2}
     \pm i\right) \, \omega_1$\\
   Vertically: &$\pm i \omega_1$
\end{tabular}

The eigenvalue $\pm i \omega_1$ has multiplicity 2 (1 horizontal and 1
vertical), corresponding to two Lyapunov families:

-- The horizontal family of homographic motions; as a curiosity, see
figure~\ref{fig:choreoHomog} for a view of a member of the family in a
choreographic rotating frame. The so-called proper degeneracy of the
Newtonian potential shows up here in that the family has no torsion.

\begin{figure}[h]
    \centering
    \includegraphics[width=5cm]{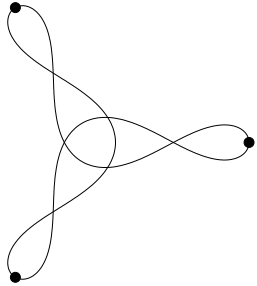}
    \caption{Homographic solution in a choreographic frame, rotating
      twice per period}
    \label{fig:choreoHomog}
\end{figure}

-- The vertical $P_{12}$-family, tangent to the Lyapunov cylinder
$S(3,1,-1)$. 

\subparagraph{The $P_{12}$-family ($S(3,1,-1)$)} Because of the $1-1$
resonance between the vertical and horizontal eigenvalues $\pm i
\omega_1$, one has to understand the third order normal form of the
vector field in order to prove that there exists a unique Lyapunov
family tangent to a non-trivial solution of the (VVE); in an
adequately rotating frame, this unique family possesses a
$G_2(3,1,-1)=D_{6}$ symmetry, which is the $\Gamma_1$-symmetry
described in \cite{CFM}. See figure~\ref{fig:p12}. As elsewhere, the
solutions have been computed using 

-- the Dopri-8 algorithm~\cite{HNW}, as implemented in the C language by
F. Joutel and M. Gastineau (Astronomie et Syst{\`e}mes dynamiques,
IMCCE), for integrating Newton's equations

-- a hybrid Newton algorithm, implemented in the Gnu Scientific
Library \cite{GSL}, for solving symmetry equations.

The germ of the $P_{12}$-family was described in \cite{Ma1} as the
family of relatively periodic solutions of the three-body problem with
the highest possible symmetry; it is immediately after he learnt about
the existence of the Eight that Christian Marchal noticed \cite{Ma2}
that these solutions turned into choreographies in a rotating frame.
This is the origin of \cite{CF1}, where we give its full scope to this
remark, and prove in particular that:

-- the flow of the reduced three-body problem in the neighborhood of
the Lagragange relative equilibrium has a unique 4-dimensional
center manifold $\mathcal{N}$,

-- the energy levels within $\mathcal{N}$ are 3-spheres,

-- the restriction of the flow to an energy level has a Poincar{\'e}
section diffeomorphic to an annulus, whose return map is a monotone
twist map which is the identity on the boundary corresponding to the
homographic motions.

\paragraph{Four bodies}

$$
\begin{array}[t]{ll}
   \mbox{Horizontal eigenvalues}\\
   \quad \pm i \omega_1\\

   \quad \left( \pm \frac{1}{7}\,\sqrt {56-14\,\sqrt {2}} \pm \,
     i\right) \, \omega_1
   &\simeq \left( \pm 0.8595325038 \pm \, i \right) \, \omega_1\\

   \quad \big(\pm \sqrt {42\,\sqrt {-16+18\,\sqrt {2}}-49}
   &\simeq \big( \pm 0.6394812009\\

   \qquad \pm \, i \,\sqrt {42\, \sqrt {-16+18\,\sqrt
       {2}}+49}\big) \, \frac{\omega_1}{14}
   &\qquad \pm 0.9533814590 \, i \big) \,
   \omega_1 \\

   \mbox{Vertical eigenvalues} \\
   \quad \pm i\omega_1 = \pm \frac{i}{2}\,\sqrt {2\,\sqrt {2}+1} \\

   \quad \pm i\omega_2 = \pm i\sqrt[4]{2}\omega_1 &\simeq \pm
   1.2155625241 \, i \, \omega_1
\end{array}$$
There are now two non-trivial vertical Lyapunov families.

\subparagraph{The 4-body, 3-lobe chain family ($S(4,1,-1)$)}

The Weinstein-Moser theorem applied to the reduced vector field and to
the eigenvalue $\pm i \omega_1$ (using lemma~\ref{lm:E1} and the fact
that $\omega_1/\omega_2 = 2^{-1/4} \notin \N$) implies the existence
of at least one more Lyapunov family corresponding to this frequency
in addition to the horizontal homographic family.

Numerical computations show that the Lyapunov cylinder $S(4,1,-1)$ is
tangent to a family which is choreographic in a rotating frame which
starts making two full turns per period in the negative direction (see
the top part of figure~\ref{fig:famGerver}). This does not follow from
the Lyapunov theorem, for the frequency $\omega_1$ is also the
frequency of the horizontal homographic family.

Yet without uniqueness it is theoretically not obvious that the family
shares the $G_2(4,1,-1)$-symmetry with $S(4,1,-1)$, in the rotating
frame making one turn per period, at the limit of the regular
pentagon. The proof of the local existence and uniqueness of this
family would require similar computations (of a normal form of the
third order) and arguments as for the $P_{12}$-family described above.

\subparagraph{The 4-body Hip-Hop family ($S(4,2,\pm 1)$)}

Since $\omega_2$ has multiplicity one and since the only other
frequency, $\omega_1$, is not an integer multiple of $\omega_2$, it
follows from Lyapunov's theorem that there exists a unique family of
relatively periodic orbits bifurcating from the square with vertical
frequency close to $\omega_2$. The eigenmode $S_1(4,2,\pm 1)$ is
tangent to the family and, by uniqueness, the family shares its
symmetry. Hence the family is the Hip-Hop family, with symmetry group
$G_1(4,2,1)=\Z/2\Z\times \Z/4\Z$ (see figure~\ref{fig:hh}). It is
studied in \cite{BCPS}.

\paragraph{Five bodies}

$$
\begin{array}[t]{ll}
   \mbox{Horizontal eigenvalues}\\
   \quad \pm i \, \omega_1\\

   \quad ... \qquad \simeq (\pm 0.9391304549 \pm i) \, \omega_1\\

   \quad ...\qquad \simeq (\pm 0.8281366700 \pm 0.8822431635 \, i) \, \omega_1\\

   \quad ...\qquad \simeq (\pm 0.5028535236 \pm 
0.8822431635 \, i ) \, \omega_1\\

   \mbox{Vertical eigenvalues} \\

   \quad \pm i \omega_1{} = \pm \frac{i\sqrt 
{2}}{4}\, \sqrt {{\frac {\sqrt {2}\sqrt
         {5-\sqrt {5}} \left( \sqrt {5}+2 \right) }{1/4\,\sqrt
         {5}-3/4+2\, \left( 1/4\,\sqrt {5}+1/4 \right) ^{2}}}}\\

   \quad \pm i \omega_2{} = \pm i \omega_1 \sqrt{2} \,\sqrt
   {{\frac {\sqrt {2}\sqrt {5-\sqrt {5}} \left( 11/4+3/4\, \sqrt {5}
         \right) } {\sqrt {2}\sqrt {5-\sqrt {5}} \left( \sqrt {5}+2
         \right) }}} &\simeq \pm 
   1.3281310261 \, i \, \omega_1
\end{array}$$

\subparagraph{The 5-body, 4-lobe chain family ($S(5,1,-1)$)}
Figure~\ref{fig:famChaine5c4b} depicts it in rotating frames which
start by making three full turns in the negative direction (symmetry
$G_4(5,1,-1)$).

\medskip Contrary to its analogue with four bodies, the maximal
vertical frequency $\omega_2$ now has multiplicity two. Using
lemma~\ref{lm:El} and the above expression of $\omega_1/\omega_2
\notin \N$, the Weinstein-Moser theorem proves the existence of at
least two Lyapunov families. Numerically, there are exactly two
families, which thus display the same symmetries as their
corresponding Lyapunov cylinders:

\subparagraph{The 5-body Eight family ($S(5,2,-1)$)} It owes its name to
the $G_2(5,2,-1)$-symmetry it acquires in a frame with $r/s=2$
(figure~\ref{fig:famHuit5c}).

\subparagraph{The 5-body, 3-lobe chain family ($S(5,2,1)$)} It owes
its name to the non-maximal chain $G_3(5,2,1)$-symmetry in a frame
with $r/s=3$ (see figure~\ref{fig:fam5cChaine3b}).

\paragraph{Six bodies}

$$
\begin{array}[t]{ll}
   \mbox{Horizontal eigenvalues}\\

   \quad \pm \, i\, \omega_1\\

   \quad ...&\simeq \pm 0.4687282051 \, i \, \omega_1\\

   \quad ...&\simeq \pm 0.4211614102 \, \omega_1\\

   \quad ...&\simeq \left( \pm 0.9499800584 \pm 
0.8151022048 \, i \right) \, \omega_1\\

   \quad ...&\simeq \left( \pm 0.2986755303 \pm 
0.8151022048 \, i \right) \, \omega_1\\

   \quad ...&\simeq \left( \pm 0.9893611078 \pm \, i \right) \,
   \omega_1\\ 

   \mbox{Vertical eigenvalues} \\

   \quad \pm i \omega_1{} = \pm i \,\sqrt {\frac{5}{4}+ \frac{\sqrt
       {3}}{3}}\\ 

   \quad \pm i \omega_2{} = \pm i \,\sqrt {3+ \frac{\sqrt {3}}{3}}
   &\simeq \pm 1.3991678967 \, i \, \omega_1\\

   \quad \pm i \omega_3 = \pm \frac{i}{2}\,\sqrt {17}
   &\simeq \pm 1.5250481798 \, i \, \omega_1
\end{array}$$

From six bodies on, some horizontal eigenvalues are purely imaginary,
which could give rise to horizontal Lyapunov families. For example,
could the choreography shown in figure~3 (4th of first column) of
\cite{S} belong to such a family? Moreover the purely imaginary
horizontal eigenvalues can resonate with the vertical frequencies, on
which we focus (see the remark before lemma~\ref{lm:El} however).

\subparagraph{The 6-body, 5-loop chain ($S(6,1,-1)$)} It is the $6$-body
instance of maximal chains. Its proof of existence has the same status
as with $4$ and $5$ bodies. 

\subparagraph{Some 6-body Hip-Hops} Since $6$ is not prime, the $6$-body
problem has a richer set of Hip-Hop families than problems with fewer
bodies.

The Weinstein-Moser theorem and numerical experiments indicate the
existence of a Hip-Hop family sharing the Hip-Hop symmetry of and
tangent to the Lyapunov cylinder $S(6,1,-1)$; it belongs to an
invariant problem having as few dimensions as the two-body problem,
where bodies are symmetric with respect to each other within two
groups of three (see part~5 of figure~\ref{fig:vve}).

The second frequency gives rise to two Lyapunov families $S(6,2,1)$
and $S(6,2,-1)$, each belonging to an invariant problem having the
dimensions of the three-body problem, where bodies are symmetric with
respect to each other within three groups of two (see part~6 of
figure~\ref{fig:vve} for $S_1(6,2,-1)$).

Since all other frequencies (vertical or horizontal) are smaller than
$\omega_3$, by the Lyapunov theorem the third vertical frequency gives
rise to a unique Hip-Hop family sharing the symmetries of $S(6,3,\pm
1)$. In a frame rotating once per period (at the limit of the
hexagon), this family is both a Hip-Hop (bodies are symmetric with
respect to each other within two groups of three) and a partial
choreography (bodies chase each other by pairs, along three distinct
closed curves).

\section{Global continuation}
\label{sec:global}

We are interested in the following two questions:

-- Existence: Does the range of frequency rotation $\varpi$ of the
frame over which the family exists contain $0$?

-- Uniqueness: Can one take the frequency $\varpi$ as a monotonous
continuous parameter over the whole family, i.e. has the torsion
constant sign?

Minimization under the $G_{r/s}(N,k,\eta)$-symmetry is a natural tool
for the existence question. It will turn out that, because of
isomorphisms between the actions of different groups
$G_{r/s}(N,k,\eta)$ (cf. section~\ref{sec:isomorphisms}), its use is
essentially restricted to the case $k=n$, i.e. to the largest vertical
frequency.

As a preliminary study we apply the results of
section~\ref{sec:globalMin} to estimate intervals of the rotation
frequency of the frame over which the relative equilibrium family
itself is the sole absolute minimizer under the
$G_{r/s}(N,k,\eta)$-symmetry.

\subsection{Minimization properties of the $N$-gon family under
  $G_{r/s}(N,k,\eta)$-symmetry} 
\label{sec:NGon}

Consider solutions $x(t)=e^{\J\varpi t}y(t)$ of Newton's equations
which, in a frame rotating with frequency $\varpi$, become
$s$-periodic loops $y(t)\in\Lambda^G$, where $G=G_{r/s}(N,k,\eta)$.  A
relative equilibrium solution with the required symmetry accomplishes
one turn in time $\frac{s}{r}$ in the rotating frame. Hence it has
frequency $\hat\omega_1=2\pi\frac{r}{s}+\varpi$ in the inertial frame.

\medskip

Referring to proposition~\ref{prop:criterion}, we now study conditions
under which the inf of the positive $\lambda$'s for which there is a
solution $x(t)=e^{\J\varpi t}y(t)$ with $y(t)\in\Lambda^G$ of the
equation
$$-\ddot x=\lambda\Delta x$$
is equal to 1, where $\Delta$ is computed with
the mutual distances $\bar r_{ij}$ of the
relative equilibrium of frequency
$2\pi\frac{r}{s}+\varpi$ in the inertial frame.
\smallskip

Let $\xi(t)$ be defined by $x(t)=\xi(\sqrt{\lambda}t)$. As $\xi(t)$ is
a solution of $-\ddot\xi=\Delta\xi$, each component of $\xi(t)$ is of
the form $\sum_ka_ke^{\J \hat\omega_kt}$, where
$\omega_1,\omega_2,\cdots,\omega_k,\cdots$ are the frequencies of the
(VVE) and
$$\hat\omega_k=\frac{\omega_k}{\omega_1}(2\pi\frac{r}{s}+\varpi).$$

The linearity of the differential equation and the structure of the
action of $G_{r/s}(N,k,\eta)$ allows us to study separately solutions
lying in the horizontal plane and solutions lying on the vertical
axis: if $y(t)$ is invariant, so are the loops $h(t)$ and $v(t)$ of
$N$-body configurations in $\R^3$ obtained from $y(t)$ by projecting
each body on the horizontal plane and the vertical axis respectively.
\medskip

{\bf 1) VERTICAL SOLUTIONS.}
A solution $x(t)$ of the
(VVE) of the form
$$(x_0(t),\cdots,x_{N-1}(t))\; \hbox{with}\;
\forall j\in\Z/N\Z,
x_j(t)=\bigl(0,0,x_j^z(t)\bigr),$$  can be
written
$$x_j^z(t)=\sum_l \Re
\left[(u_l\zeta^{jl}+v_l\bar\zeta^{jl})
  e^{i\hat\omega_l\sqrt{\lambda}t}\right],\; j\in\Z/N\Z.$$

The symmetry conditions impose not only $s$-periodicity but even
$1$-perio\-dicity. Indeed, if we choose $\beta=0,\delta=0, q=1$, we
get $\theta=1\in\R/s\Z$, hence
$$x^z_j(t)=y^z_j(t)=y^z_j(t-1)=x^z_j(t-1),$$
that is
$$\sum_l \Re
\left[(u_l\zeta^{jl}+v_l\bar\zeta^{jl})(1-e^{-i\hat\omega_l\sqrt{\lambda}})
    e^{i\hat\omega_l\sqrt{\lambda}t}\right]\equiv0.$$
It follows that for any $l$,
$(u_l\zeta^{jl}+v_l\bar\zeta^{jl})(1-e^{-i\hat\omega_l\sqrt{\lambda}})=0$,
that is
$$\hbox{either}\quad u_l=v_l=0\quad\hbox{or}\quad
\exists m\in\Z,\; \hat\omega_l\sqrt{\lambda}=2\pi m.$$
The second possibility occurs for a single index
$l$ provided no relation of the form
$m\hat\omega_{l'}=m'\hat\omega_{l}$ exists with $m,m'\in\Z$.
In any case, as we are only interested
in the values that $\lambda$ can take, we can
restrict the attention to solutions corresponding
to
a single index $l$ such that $\hat\omega_l\sqrt{\lambda}=2\pi m$.
\smallskip

\noindent Now, the invariance of $x(t)$ under the
action of $G_{\frac{r}{s}}(N,k,\eta)$ means that
for every
$(\theta,\delta,\beta,\xi)\in \Z/s\Z\times
\Z/2\Z\times \Z/2\Z\times \{-1,+1\}$,
such that $\theta-\frac{\beta}{2}-k\eta\frac{\delta}{N}=q\in\Z$,
$$\Re
\left[(u_l\zeta^{jl}+v_l\bar\zeta^{jl})e^{i\hat\omega_l\sqrt{\lambda}t}-
    (u_l\zeta^{\xi(j+\delta)l}+v_l\bar\zeta^{\xi(j+\delta)l})
    e^{i[\pi\beta+\hat\omega_l\sqrt{\lambda}\xi(t-\theta)]}\right]\equiv 0.$$
Choosing $\delta=0,\beta=1,\xi=1$, hence
$\theta=1/2$, we get that $m=1+2p$ is odd.

\noindent Choosing $\delta=0,\beta=0,\xi=-1$,
hence $\theta=0$, we get that for all
$j\in\Z/N\Z$,
$(u_l-\bar u_l)\zeta^{jl}+(v_l-\bar
v_l)\bar\zeta^{jl}=0$, which implies that $u_l$
and $v_l$
must be real.

\noindent Finally, for $\xi=1$ the condition is that for every $j,\delta$,
$$u_l\zeta^{jl}\left(1-e^{2\pi
      i\left(-\frac{mk\eta\delta+\delta
          l}{N}\right)}\right)
+v_l\bar\zeta^{jl}\left(1-e^{2\pi
      i\left(-\frac{mk\eta\delta-\delta
          l}{N}\right)}\right)=0,$$
which implies that one of the two coefficients
$u_l$ and $v_l$ must vanish and that, according to
which one does vanish,
$$l=\pm mk\eta\mod N=\pm (1+2p)k\eta\mod N.$$
Finally, the inf of the corresponding $\sqrt\lambda$'s is the inf of the
$\left|\frac{(1+2p)2\pi}{\hat\omega_l}\right|$, that is
$$\inf_p\frac{\omega_1}{\omega_{(1+2p)k\eta}}\left|
    \frac{(1+2p)2\pi}{2\pi\frac{r}{s}+\varpi}\right|,$$
where the frequencies
$\omega_{(1+2p)k\eta}=\omega_{-(1+2p)k\eta}$ involved are the ones of
bifurcating vertical solutions with the required
symmetry (in particular,
$\omega_{k\eta}=\omega_k$ is
always among them).    One deduces immediately
that the condition $\inf \lambda\ge 1$ reads
$$\left|\varpi+2\pi\frac{r}{s}\right|\le
V=\inf_{p>0}\frac{\omega_1}{\omega_{(1+2p)k\eta}}|(1+2p)2\pi|,$$
where the retriction to the $p\ge0$ comes from
the fact that if $p'=-(1+p)$, one has
$|1+2p'|=|1+2p|$, hence
$\omega_{1+2p'k\eta}=\omega_{1+2pk\eta}$.
\medskip

{\bf 2) HORIZONTAL SOLUTIONS.}  The general ``horizontal'' solution
$$\xi=(\xi_0,\cdots,\xi_{N-1}), \;\hbox{with}\;
\xi_j(t)=(\xi_j^h(t),0),$$
of $ -\ddot\xi=\Delta\xi$ is such that,
identifying $\R^2$ with $\C$, there exists
$k_0,a_l, b_l,c_l,d_l\in\C$ such that
$$\forall j,\; \xi_j^h(t)=k_0+\sum_{l\not=0}\left[\bigl(a_le^{i\hat\omega_lt}
    +b_le^{-i\hat\omega_lt}\bigr)\zeta^{jl}+\bigl(c_le^{i\hat\omega_lt}
    +d_le^{-i\hat\omega_lt}\bigr)\bar\zeta^{jl}\right].$$
If $N=2n+1$ is odd, to each of the $n$ pairs
$\lambda_l=\lambda_{N-l}=-\hat\omega_l^2$ of
non-zero
eigenvalues of the matrix which defines $\Delta$,
corresponds an eight dimensional space of
solutions
parametrized by the four complex constants
$a_l,b_l,c_l,d_l$. This gives the dimension $4n$
of the phase
space after quotienting by the translations.
\smallskip

\noindent If $N=2n$ is even, to the eigenvalue
$\lambda_n$ corresponds only a four dimensional
space
of solutions because $\zeta^n=\bar\zeta^n=-1$. In
this case the two complex parameters are
$a_l+c_l$ and $b_l+d_l$.
\smallskip

\noindent We shall now study the conditions
imposed on horizontal solutions by the invariance
under
$G_\frac{r}{s}(N,k,\eta)$.
\smallskip

1) The condition of $s$-periodicity on $y(t)$
imposes that, in addition to $\hat\omega_0=0$,
only one frequency
$\hat\omega_l$ is present: indeed, for all $j$ and all $t$,
$$0=\sum_{l\ne
    0}\bigl[(e^{i\alpha_ls}-1)(a_l\zeta^{jl}+c_l\bar\zeta^{jl})e^{i\alpha_lt}
+(e^{i\beta_ls}-1)(b_l\zeta^{jl}+d_l\bar\zeta^{jl})e^{i\beta_lt}\bigr],$$ where
$$\alpha_l=\hat\omega_l\sqrt\lambda-\varpi,\;
\beta_l=-\hat\omega_l\sqrt\lambda-\varpi.$$
Hence, as $\omega_l\pm\omega_{l'}\not=0$ if $l\not= l'$,
$$\forall j,l,\; (e^{i\alpha_ls}-1)(a_l\zeta^{jl}+c_l\zeta^{jl})=
(e^{i\beta_ls}-1)(b_l\zeta^{jl}+d_l\bar\zeta^{jl})=0,$$
As $a_l\zeta^{jl}+c_l\bar\zeta^{jl}$ (resp.
$b_l\zeta^{jl}+d_l\bar\zeta^{jl}$) cannot be
equal to 0 for all
$j=0,\cdots,N-1$ except if $a_l=c_l=0$ (resp.
$b_l=d_l=0$), one deduces that there exists a
(necessarily unique) $l\ne 0$ such that
$e^{i\alpha_ls}=1$ or $e^{i\beta_ls}=1$, that is
$$\alpha_ls=(\hat\omega_l\sqrt\lambda-\varpi)s=m2\pi, \quad\hbox{\rm or} \quad
\beta_ls=(-\hat\omega_l\sqrt\lambda-\varpi)s=m2\pi,$$
i.e.
$$y_j(t)=(k_0+(a_l\zeta^{jl}+c_l\bar\zeta^{jl})e^{i2\pi
    \frac{m}{s}t},0)\quad\hbox{or}\quad
y_j(t)=(k_0+(b_l\zeta^{jl}+d_l\bar\zeta^{jl})e^{i2\pi \frac{m}{s}t},0).$$
Both cases are identical hence we suppose that the first identity holds.
Hence we are looking for the invariance of
$$y_j(t)=\left(k_0+(u_l\zeta^{jl}+v_l\bar\zeta^{jl})e^{i2\pi
      \frac{m}{s}t},0\right),$$
corresponding to
$$\sqrt{\lambda}=\frac{\omega_1}{\omega_l}
\left|\frac{\varpi+2\pi\frac{m}{s}}{\varpi+2\pi\frac{r}{s}}\right|.$$

This means that, for all $(\theta,\delta,\beta,\xi)\in \R/s\Z
\times\Z/N\Z\times\Z/2\Z\times\{-1,+1\}$ with
$$\theta-\frac{\beta}{2}-k\eta\frac{\delta}{N}=q\in\Z,$$
one has for all $j\in \Z/{N\Z}$ and for all $t$,
$y_j(t)=e^{2\pi i\alpha}{\bar y}^\xi _{\xi(j+\delta)}(\xi(t-\theta))$,
where
$\;\alpha=\frac{r}{s}\theta-\frac{\delta}{N}\pmod{1}$.
As $\xi^2=1$, this is equivalent to
the following:
$$k_0+(u_l\zeta^{jl}+v_l\bar\zeta^{jl})e^{i2\pi
    \frac{m}{s}t}=\left[\bar{k_0}^\xi+({\bar
      u_l}^\xi\zeta^{(j+\delta)l}+{\bar v_l}^\xi\bar\zeta^{(j+\delta)l})e^{2\pi
      i\frac{m}{s}(t-\theta)}\right]e^{2\pi i\alpha}\;
.$$  Let us first take $\xi=1$ and
hence ${\bar u_l}^\xi=u_l$ and ${\bar
    v_l}^\xi=v_l$.  The condition becomes the
vanishing of the expression
$$k_0(1-e^{2\pi i\alpha})+
u_le^{2\pi i(\frac{jl}{N}+\frac{mt}{s})}\left(1-e^{2\pi
      i(\frac{\delta l}{N}+\alpha-\frac{m\theta}{s})}\right)$$
$$\phantom{k_0(1-e^{2\pi i\alpha})} +v_le^{2\pi
    i(-\frac{jl}{N}+\frac{mt}{s})}\left(1-e^{2\pi
      i(-\frac{\delta l}{N}+\alpha-\frac{m\theta}{s})}\right)$$
for all $j$ and all $t$. This
implies the vanishing of one or the
other of the following complex numbers~:
$$A=1-e^{2\pi
    i(\frac{\delta}{N}l+\alpha-\frac{m}{s}\theta)},\quad\hbox{or}\quad
B=1-e^{2\pi i(-\frac{\delta}{N}l+\alpha-\frac{m}{s}\theta)}.$$
\goodbreak

\indent -- First case: $A=0$, that is
$y_j(t)=(k_0+u_l\zeta^{jl}e^{2\pi
    i\frac{m}{s}t},0)$.
\smallskip

-- Second case: $B=0$, that is
$y_j(t)=(k_0+v_l\bar\zeta^{jl}e^{2\pi
    i\frac{m}{s}t},0)$.
\smallskip

In both cases, $k_0=0$ unless $\alpha=0\mod 2\pi$
but $k_0$ does not play any role in the
determination of
$\lambda$.

\medskip
Taking now $\xi=-1$, we see that $u_l$ in the
first case, $v_l$ in the second one, must be real.
Replacing $\alpha$ by
$\frac{r}{s}\theta-\frac{\delta}{N}\mod 1$, we get
$$\frac{\delta}{N}(l-1)+\frac{r-m}{s}\theta\in\Z\quad\hbox{or}
\quad -\frac{\delta}{N}(l+1)+\frac{r-m}{s}\theta\in\Z.$$

Finally, replacing $\theta$ by
$\frac{\beta}{2}+k\eta\frac{\delta}{N}+q$, we
obtain the
following conditions:
$$\left[l-1+\frac{r-m}{s}k\eta\right]\frac{\delta}{N} +
\frac{r-m}{s}(\frac{\beta}{2}+q)\in\Z,$$
or
$$\left[-(l+1)+\frac{r-m}{s}k\eta\right]\frac{\delta}{N} +
\frac{r-m}{s}(\frac{\beta}{2}+q)\in\Z.$$

Taking $\beta=0, q=0,\delta=1$, this implies that
\smallskip

$l-1+\frac{r-m}{s}k\eta=0\mod N$ in the first case,

$-(l+1)+\frac{r-m}{s}k\eta=0\mod N$ in the second one.
\smallskip

Taking now $\beta=1$, this implies in both cases that, for any
$q\in\Z/s\Z$,

$$\frac{r-m}{s}(\frac{1}{2}+q)=0\mod
1,\quad\hbox{\rm i.e.}\quad \exists p\in\Z,\;
m=r-2ps.$$

It follows that $ l=1-2pk\eta\mod N$ in the first
case, $l=-(1-2pk\eta)\mod N$ in the second one.
Notice that both cases correspond to the same type of solutions
$$y_j(t)=(k_0+a_l\zeta^{j(1-2pk\eta)}e^{2\pi i(\frac{r}{s}-2p)t},0),$$
(and that, when in addition $1-2pk\eta$ is prime to $N$, these
solutions are up to translation and scaling, {\it a relative
  equilibrium of the regular $N$-gon with a different period and a
  relabelling of the bodies}, that is one of the relative equilibrium
solutions found in section \ref{sec:isomorphisms}). This shows that
the case where $A=B=0$ will not lead to different solutions and hence
need not be studied separately.  In both cases, as
$\omega_l=\omega_{-l}$,
$$\sqrt{\lambda}=\frac{\omega_1}{\omega_{1-2pk\eta}}
\left|\frac{\varpi+2\pi\frac{r}{s}-4p\pi}
   {\varpi+2\pi\frac{r}{s}}\right|.$$ Notice that the value $p=0$
corresponds to $\lambda=1$. Hence, in order to insure that the minimum
value of $\lambda$ is 1, it is enough to insure that $\lambda$ is
always $\ge 1$.  This amounts to the following conditions, for which
{\it one should remember that the values of $p$ must be such that
   $1-2pk\eta\ne 0\mod N$.}  \goodbreak

\smallskip
{\it i)} If  $p>0$ and $\varpi+2\pi\frac{r}{s}\ge 4p\pi$,  the condition is
$$(\omega_{1-2pk\eta}-\omega_1)(\varpi+2\pi\frac{r}{s})\le -4p\pi\omega_1,$$
\indent which can never be satisfied because for every $l$,
$\omega_l-\omega_1\ge 0$; this implies that $$\varpi+2\pi\frac{r}{s}<
4\pi;$$ \smallskip

{\it ii)} If  $p>0$ and $0<\varpi+2\pi\frac{r}{s}<4p\pi$, the condition is
$$(\omega_1+\omega_{1-2pk\eta})(\varpi+2\pi\frac{r}{s})\le 4p\pi\omega_1;$$

{\it iii)} If  $p>0$ and $\varpi+2\pi\frac{r}{s}<0$, the condition is
$$(\omega_1-\omega_{1-2pk\eta})(\varpi+2\pi\frac{r}{s})\le 4p\pi\omega_1;$$

{\it iv)} If  $p<0$ and $\varpi+2\pi\frac{r}{s}\le 4p\pi$, the condition is
$$(\omega_1-\omega_{1-2pk\eta})(\varpi+2\pi\frac{r}{s})\le 4p\pi\omega_1,$$
\indent which can never be satisfied for the same
reason as above; and this implies
\indent that $$\varpi+2\pi\frac{r}{s}>-4\pi;$$
\smallskip

{\it v)} If  $p<0$ and $4p\pi\le
\varpi+2\pi\frac{r}{s}\le 0$, the condition is
$$-(\omega_1+\omega_{1-2pk\eta})(\varpi+2\pi\frac{r}{s})\le -4p\pi\omega_1;$$

{\it vi)} If  $p<0$ and $\varpi+2\pi\frac{r}{s}\ge 0$,  the condition is
$$(\omega_{1-2pk\eta}-\omega_1)(\varpi+2\pi\frac{r}{s})\le -4p\pi\omega_1.$$

Finally, the horizontal conditions are that
$-H-\le \varpi+2\pi\frac{r}{s}\le +H_+$, where
$$H_+=\inf\left\{4\pi,
    {\inf_{p>0}}^*\left(\frac{4p\pi\omega_1}{\omega_1+\omega_{1-2pk\eta}},
      \frac{4p\pi\omega_1}{\omega_{1+2pk\eta}-\omega_1}
    \right)\right\},$$
$$H_-=\inf\left\{4\pi,
    {\inf_{p>0}}^*\left(\frac{4p\pi\omega_1}{\omega_1+\omega_{1+2pk\eta}},
      \frac{4p\pi\omega_1}{\omega_{1-2pk\eta}-\omega_1}
    \right)\right\},$$
\textit{where  the ${{\inf}^*}$ means that the
    index $k$ of any $\omega_k$ involved can never be
    0}.

Recalling the vertical condition
$$\left|\varpi+2\pi\frac{r}{s}\right|\le
V=\inf_{p\ge 0}\frac{\omega_1}{\omega_{(1+2p)k\eta}}|(1+2p)2\pi|,$$
we obtain the

\begin{theorem}\label{theorem:absolute}
  The following condition implies that the relative equilibrium
  solution of the equal mass regular $N$-gon with frequency
  $2\pi\frac{r}{s}$ in a frame rotating with frequency $\varpi$ is the
  sole absolute minimizer of the action among paths which in the
  rotating frame are $s$-periodic loops with the
  $G_{r/s}(N,k,\eta)$-symmetry of some solution of the vertical
  variational equation:
  $$-\inf(V,H-)\le \varpi+2\pi\frac{r}{s}\le \inf(V,H_+).$$
\end{theorem}

When $k=n$, the horizontal contribution disappears; in this case, the
corresponding Lyapunov families can be searched for as absolute
minimizers of the action with the given symmetry constraints:

\begin{corollary}
  The following condition implies that the relative equilibrium
  solution of the equal mass regular $N$-gon with frequency
  $2\pi\frac{r}{s}$ in a frame rotating with frequency $\varpi$ is the
  sole absolute minimizer of the action among paths which in the
  rotating frame are $s$-periodic loops with the
  $G_{r/s}(N,n,\eta)$-symmetry:
  $$-V\le \varpi+2\pi\frac{r}{s}\le V.$$
\end{corollary}

\begin{proof} As $\omega_1\le\omega_n$, it is enough to prove that
$$\frac{2p}{\omega_{1\pm 2pn}\pm \omega_1}\ge
\frac{1}{\omega_n}\quad\hbox{and}\quad
\frac{2p}{\omega_{1\mp 2pn}\pm \omega_1}\ge \frac{1}{\omega_n},$$
in all the cases where no $\omega_0$ is implied.
But all these inequalities are implied by
$2\omega_n\ge\omega_{1\pm 2pn}+\omega_1$.
\end{proof}
\smallskip

In the following sections, where the numerical side plays an important
role, we first study the two families associated with the highest
frequency $\omega_n$: the Eight families when $N$ is odd and the
Hip-Hop families when $N$ is even.  The first ones generalize
Marchal's $P_{12}$-family (\cite{CM,Ma2,CFM,CF1,S}) which links the
Lagrange equilateral relative equilibrium to the Eight, the second
ones generalize the family which links the relative equilibrium of the
square to the original Hip-Hop (\cite{CV,TV}).  These are precisely
the families for which global minimization of the action under the
$G_{r/s}(N,k,\eta)$-symmetry constraint may be used for global
continuation.  But even then, no unicity of minimizers is proved;
hence the continuity of the families so obtained is in question, even
if global topological continuation results can be used.

We then study the families associated with the lowest frequency
$\omega_1$, which lead to Sim{\'o}'s chains with the maximal number of
lobes ($N-1$ lobes for $N$ bodies): only numerical results are given
here because these families are only local minimizers of the action.
This is related to the fact that, starting with $N=4$, there exist
isomorphisms between the actions of some of the groups
$G_{r/s}(N,k,\eta)$ (see section~\ref{sec:isomorphisms}). We also give
examples of chains with a non maximal number of lobes, where similar
phenomena do occur.

\paragraph{Remark} It follows from corollary~\ref{cor:dense} that
absolute choreographies are dense in the Lyapunov families as soon as
torsion is present (i.e. as soon as $\varpi$ does vary along the
family). This results from the fact that a path wich is
$G_{r/s}(N,k,\eta)$-symmetric in a frame rotating with frequency
$\varpi$ is $G_{r/s+\varpi}(N,k,\eta)$-symmetric in the inertial
frame. (In the action of the group, this corresponds to the
replacement of $\alpha$ by $\alpha + \varpi \theta$.)

\subsection{Families associated with the highest frequency $\omega_n$}

They satisfy $V=2\pi\frac{\omega_1}{\omega_n}$ and it follows that the
relative equilibrium remains the unique global minimizer as long as
$$-2\pi\frac{\omega_1}{\omega_n}\le \varpi+2\pi \le
2\pi\frac{\omega_1}{\omega_n},$$ i.e. until one reaches the vertical
bifurcation of the corresponding family.

\subsubsection{The Eight families: $G=G_2(2n+1,n,-1)$}

\begin{figure}[h]
     \centering
     \includegraphics[width=11cm]{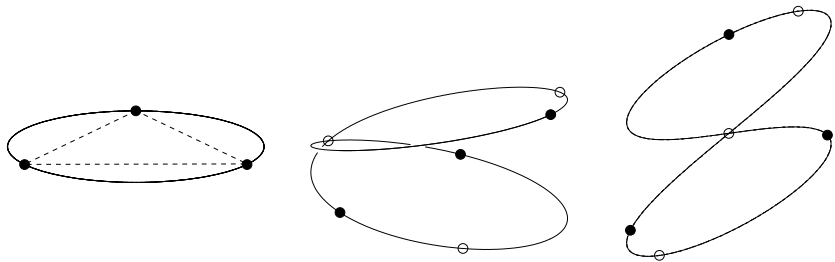}
     \caption{The $P_{12}$ family. Filled and hollow circles represent
       bodies at times $t=0$ and $t=T/12$}
     \label{fig:p12}
\end{figure}

When $n=1$ one gets the original $P_{12}$-family of Christian Marchal
(see figure~\ref{fig:p12}). In this case, the minimization property
of the relative equilibrium family was already proved in \cite{BT}
under the weaker assumption of the choreography symmetry.  Note that
for planar solutions, $2\pi$ is replaced by $4\pi$. We leave to the
reader the pleasure to check that this was a priori obvious.

When $n\ge 2$, a proof of the existence of a $G_2(2n+1,n,-1) =
D_{2n}$-symmetric Eight is still missing. The existence of a
$D_n$-symmetric Eight is proved in~\cite{FT} but, while highly
probable, the fact that it is automatically $D_{2n}$-symmetric is not
proved. Figures~\ref{fig:famHuit5c} and~\ref{fig:action5bEightChain}
depict the case $n=2$.

\subsubsection{The Hip-Hop families: $G=G_1(2n,n,\pm 1)$ }

The original Hip-Hop (figures~\ref{fig:hh} and~\ref{fig:actionFamHH})
corresponds to the case $n=2$.  For all values of $n$, the existence
of the Hip-Hop family is proved in \cite{TV}, with the usual proviso
that no uniqueness, hence no continuity, is proved.  It seems very
likely that the natural end of the family is a pair of simultaneous
$n$-tuple collisions for $\varpi=\omega_2$, but this not proved
either.  The heuristic explanation is that the phenomenon is the same
as for a fixed center pb with angle $\alpha=2\pi$: in the inertial
frame, each body must turn by exactly $2\pi$ during time $T/2$).
Moreover, all along the family, the minimizing solutions should
possess the brake symmetry.

\begin{figure}[h]
     \centering
     \includegraphics[width=10cm]{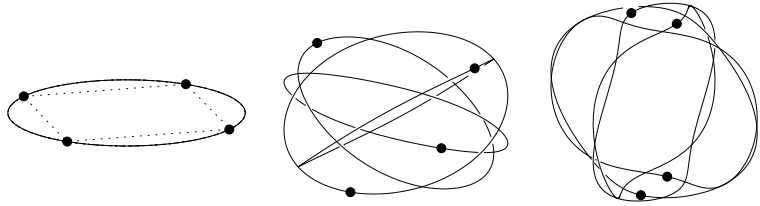}
     \caption{The original Hip-Hop family (Hip-Hop in the middle)}
     \label{fig:hh}
\end{figure}

\paragraph{Remark} Illustrating the remark at the end
of~\ref{sec:NGon}, figure~\ref{fig:actionFamHH} shows the two simplest
spatial choreographies (in the inertial frame) in the Hip-Hop family
for $n=2$. First described by S. Terracini and A.  Venturelli, they
correspond to the symmetry groups $G_\frac{3}{2}(4,2,\pm 1)$ and
$G_\frac{5}{6}(4,2,\pm 1)$. On the other hand, figure~\ref{fig:halo}
shows the original Hip-Hop in a frame rotating with its frequency
$\omega_2$.

\begin{figure}[halo]
  \centering
  \includegraphics[width=6cm]{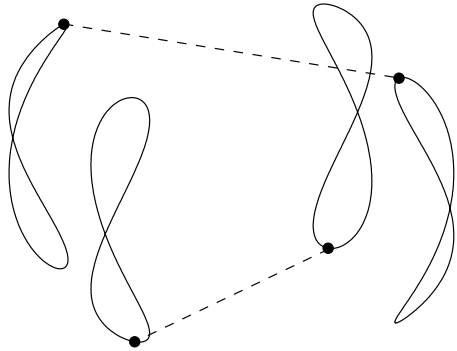}
  \caption{The Hip-Hop in a frame rotating with its very frequency}
  \label{fig:halo}
\end{figure}

\subsection{Chains families:  the role of angular momentum}
\label{sec:C}

The chains are planar choreographies of the equal-mass $N$-body
problem, similar to the Eight but with a number $\ell \geq 2$ of
lobes. The first example with $\ell > 2$ is Gerver's chain with 3
lobes and 4 bodies~\cite{CGMS}. Because of their symmetries, chains
with an even number $\ell$ of lobes have a vanishing angular momentum
while this is not the case when $\ell$ is odd. This fundamental
difference accounts for a totally different behaviour of the
corresponding Lyapunov families. We have studied numerically the cases
of 4 and 5 bodies and have observed:

-- when $\ell$ is even, a complete unfolding of the corresponding
unchained polygon into a vertical chain;

-- when $\ell$ is odd, a complete unfolding into a horizontal chain
(in a frame which is still rotating), followed by a plane family which
continues to a horizontal chain. 

Due to the existence of isomorphisms of the actions of different
symmetry groups $G_{r/s} (n,k,\eta)$, global minimization of the
action may succeed only for the Eights ($\ell=2$).


\subsubsection{Maximal chain families: $G=G_{N-1}(N,1,-1)$, $N=2n+1$} 
\label{sec:maximalChain}

When observed in a frame which rotates $N-1$ times per period in the
negative direction, the non-trivial solutions of the (VVE)
corresponding to the frequency $\omega_1$ give rise, to an
infinitesimal choreography which unfolds the $N-1$ circles.  When
$N=2n+1$ is odd, the full family hopefully continues up to the fixed
frame into a vertical zero angular momentum planar chain with $N-1$
lobes.  We recall that it was proved in section~\ref{sec:isomorphisms}
that the action of $G_{N-1}(N,1,-1)$ is, for $N=2n+1$, isomorphic to
the one of the Eight with $N$ bodies $G_2(N,n,-1)$.

We study now the simplest case where the maximal chain family and the
Eight family differ.

\subsubsection{The 4-lobe chain and the Eight for 5 bodies: two
  isomorphic actions of the symmetry groups}

\paragraph{Four-lobe chain: $G=G_4(5,1,-1)$}

When $N=5$, we have checked numerically that, indeed, the $G_4(5,1-1)$
family continues up to a vertical planar chain with four lobes.
During the unfolding, the two central lobes become smaller and flatten
more rapidly while the the two exterior lobes remain during a long
time almost vertical before flattening down as the Big Ears of Big
Brother.

We have
$$V=2\pi,\quad H_+=4\pi\frac{\omega_1}{\omega_1+\omega_2},\;
H_-=\inf \left(2\pi,4\pi\frac{\omega_1}{\omega_2-\omega_1} \right)
=2\pi.$$ Hence, on the left hand side ($\varpi+8\pi<0$), the estimate
goes all the way till the bifurcation of the chain family but on the
right hand side, it stops before:
$-2\pi\le\varpi+8\pi\le4\pi\frac{\omega_1}{\omega_1+\omega_2}$.  This
is due to the existence of the Eight family (see
figure~\ref{fig:famHuit5c}) which, as we have recalled, has the same
symmetries up to reordering (see figure~\ref{fig:action5bEightChain}).

\begin{figure}[h]
     \centering
     \includegraphics[width=11cm]{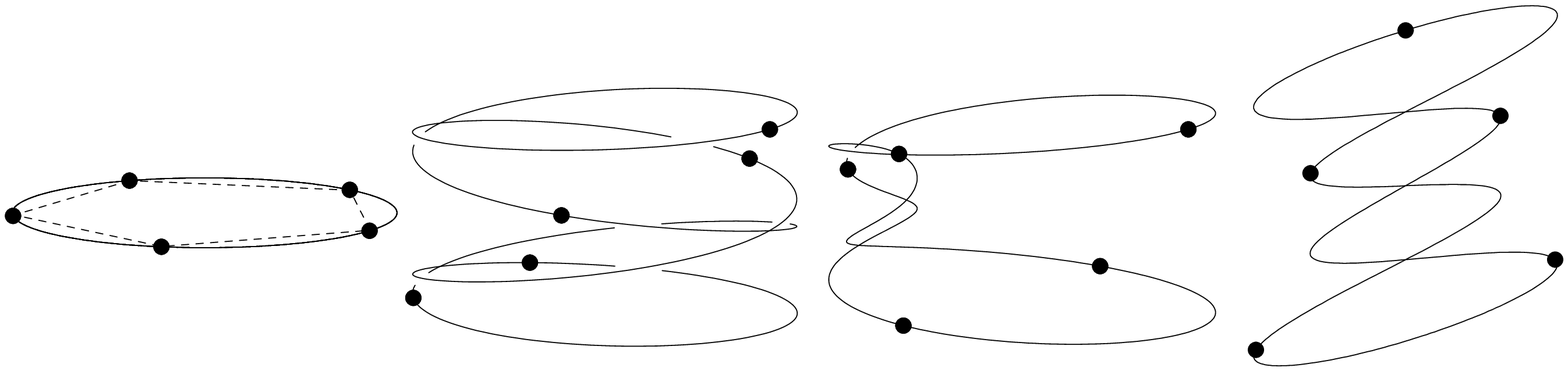}
     \caption{The family of the 5-body chain with 4 loops}
     \label{fig:famChaine5c4b}
\end{figure}

\paragraph{Five-body Eight: $G=G_2(5,2,-1)$}

The solutions of the (VVE) coresponding to the 2 families with
frequency $\omega_2$ (recall that they differ only by the replacement
of $\zeta^k$ by $\bar\zeta^k$) give rise to chains with respectively 3
lobes and 2 lobes, the order of the bodies in the choreography being
respectively the one of the retrograde and the direct stellated
pentagon.  We are interested now in the second one.

We have
$$V=2\pi\frac{\omega_1}{\omega_2},\quad
H_+=\inf(4\pi,8\pi\frac{\omega_1}{\omega_2-\omega_1},
12\pi\frac{\omega_1}{\omega_1+\omega_2})=4\pi,\;
H_-=8\pi\frac{\omega_1}{\omega_1+\omega_2} \cdot$$

It follows that $\lambda_\varpi^G=1$ as long as
$-2\pi\frac{\omega_1}{\omega_2}\le \varpi+4\pi\le
2\pi\frac{\omega_1}{\omega_2}$, i.e. until we reach the vertical
bifurcation of the Eight family (see
figure~\ref{fig:action5bEightChain}).

\begin{figure}[h]
     \centering
     \includegraphics[width=11cm]{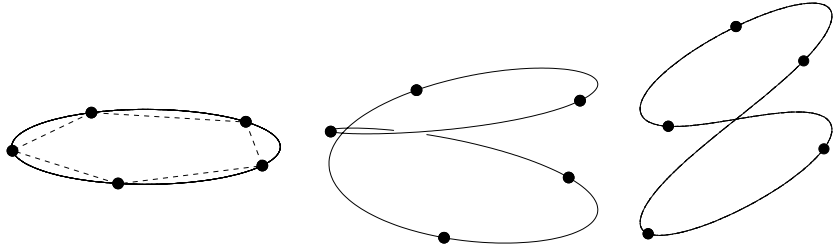}
     \caption{The family of the 5-body Eight.}
     \label{fig:famHuit5c}
\end{figure}

\noindent {\bf Remarks.} 1) The case of $G=G_{4}(5,1,-1)$ is the first
one where $H_+<V$.  This is because the second frequency $\omega_2$,
while ruled out from the $G$-invariant vertical solutions, is present
into the $G$-invariant horizontal solutions.

2) The estimate $H_+$ yields a lower bound of the action of the member
of the Lyapunov family which bifurcates at $\varpi = (-2\omega_2 -
\omega_1) \frac{\omega_1}{\omega_2}$. In particular, at this point,
the action is higher than it is at the bifurcation point, because,
$\frac{\omega_1}{\omega_2}<\frac{2\omega_1}{\omega_1+\omega_2}$.

\subsubsection{The 3-lobe chains for 4 or 5 bodies}

\paragraph{Four bodies} 

When $N=4$, the non-trivial solutions of the (VVE) corresponding to
the frequency $\omega_1$, more precisely to the group $G_3(4,1,\pm
1)$, give rise to 3 lobes chains when one starts with a frame which
rotates two full turns in the negative direction.  When the rotation
decreases, the chain starts opening but the central lobe decreases and
when it reaches approximately $-0.9$ turns by period, an almost
collision occurs. After that, the solution flattens to the horizontal
plane (with $\varpi$ not yet vanishing) where a second bifurcation
leads to the Gerver solution (figures~\ref{fig:famGerver}
and~\ref{fig:action4cChaine3b}).

\begin{figure}[h]
     \centering
     \includegraphics[width=11cm]{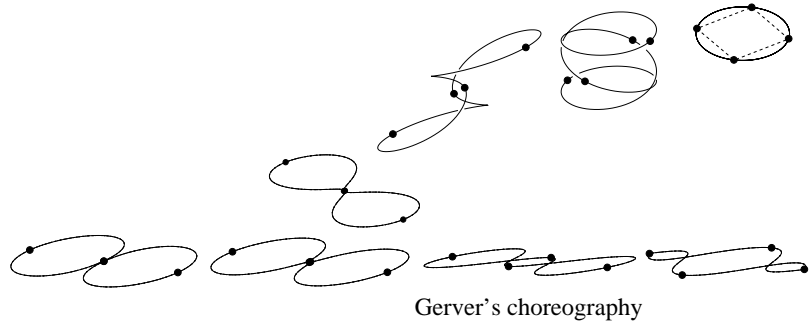}
     \caption{The planar Gerver family as a secondary bifurcation}
     \label{fig:famGerver}
\end{figure}

\paragraph{Five bodies} 

An analogous scenario is observed for $N=5$ for the frequency
$\omega_2$, more precisely with the group $G_3(5,2,1)$. but in this
case, the family never gets close to collision
(figures~\ref{fig:fam5cChaine3b} and~\ref{fig:action5cChaine3b}) and
the central lobe is bigger than the two extreme ones.

\begin{figure}[h]
     \centering
     \includegraphics[width=11cm]{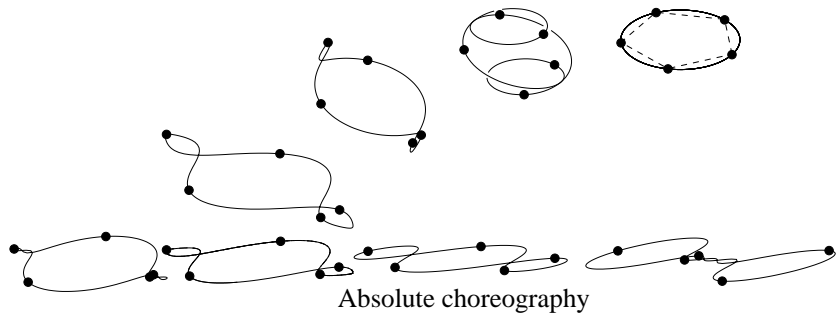}
     \caption{The $5$-body $3$-lobe chain family and its planar secondary bifurcation}
     \label{fig:fam5cChaine3b}
\end{figure}

\clearpage
\subsection{Action diagrams} 

Figures~\ref{fig:actionP12}--\ref{fig:action5cChaine3b} display
the action of families of solutions as a function of the frequency
$\varpi$ of the rotation of the frame. Each figure corresponds to a
fixed symmetry in the rotating frame (recall that the isomorphisms
between group actions are associated with relabelling of the bodies as
described in theorem~\ref{theorem:absolute}):

\begin{tabular}[t]{llll}
  Figure~\ref{fig:actionP12} &$G_2(3,1,-1)$ &$\equiv$ &$G_{-2}(3,1,1)$\\
  Figure~\ref{fig:actionFamHH} &$G_1(4,2,\pm 1)$\\
  Figure~\ref{fig:action4cChaine3b} &$G_3(4,1,-1)$ &$\equiv$ &$G_1(4,1,1)$\\
  Figure~\ref{fig:action5bEightChain} &$G_4(5,1,-1)$ &$\equiv$
  &$G_2(5,2,-1)$\\ 
  Figure~\ref{fig:action5cChaine3b} &$G_3(5,2,1)$ &$\equiv$ &$G_1(5,1,1)$
\end{tabular}

The solutions are represented in perspective in $\R^3$. This holds in
particular for solutions lying in the horizontal plane (relative
equilibria and horizontal secondary families in
figures~\ref{fig:action4cChaine3b} and~\ref{fig:action5cChaine3b}) or
in a vertical plane (Eights or chains in figures~\ref{fig:actionP12}
and~\ref{fig:action5bEightChain}). 

\paragraph{Display of the implications of
  theorem~\ref{theorem:absolute}}
Fat segments on the $\varpi$-axis and the corresponding fat part on
the graph of the action indicate intervals on which the relative
equilibrium of the regular $N$-gon is the unique absolute minimizer
for the given symmetry in the rotating frame. It appears that the
theorem detects global phenomena i.e., the presence of different
branches of Lyapunov families. 

\paragraph{Symmetry with respect to $\varpi=0$ of figure~\ref{fig:actionP12}}
This symmetry corresponds to the isomorphism $G_2(3,1,-1) \equiv
G_{-2}(3,1,1)$. Analogous symmetries could have been shown on the
other figures. 

On the $\varpi<0$-half (resp. $\varpi >0$-half) of the figure, the
configuration of the relative equilibrium is a positively (resp.
negatively) oriented equilateral triangle. In the rotating frame, it
makes two turns in the positive (resp. negative) direction. Notice
that in the inertial frame the four possible combinations of
orientation of the configuration and direction of motion occur at the
four bifurcation points $-3\omega_1$, $-\omega_1$, $\omega_1$ and
$3\omega_1$.

Also, in figure~\ref{fig:actionP12}, the continuity of the Lyapunov
family implies a change of sign of the $x$-coordinate of the body 0
when $\varpi$ goes through 0. This is consistent with the
$G_r(N,k,\eta)$-symmetry: if a loop is $G_r(N,k,\eta)$-symmetric, the
same is true for its image under the rotation of angle $\pi$ around
the vertical axis. 

\paragraph{Cases without torsion} The $G_{N-1}(N,1,-1)$-Lyapunov
family starting at $\varpi = -N \omega_1$ (resp. $+N\omega_1$) from
the relative equilibrium which completes minus (resp. plus) one turn
per period in the inertial frame has no torsion: indeed, it is merely
the family obtained by rotating the horizontal relative equilibrium
around the $y$-axis. As this family has constant action, this fact
accounts for the end of the action diagrams in
figures~\ref{fig:actionP12}, \ref{fig:action4cChaine3b} and
\ref{fig:action5bEightChain}.

\paragraph{About figure~\ref{fig:action5bEightChain}} The family of
solutions which in figure~\ref{fig:action5bEightChain} bifurcates at
$\varpi= (-2\omega_2-\omega_1)\frac{\omega_1}{\omega_2}$ is, up to
scaling, symmetric of that which, in
figure~\ref{fig:action5cChaine3b}, bifurcates at
$\varpi=\omega_1-3\omega_1$. More precisely, the latter family
transforms into the former one by:

1) reflection about $\varpi=0$, which corresponds to the isomorphism
$G_3(5,2,1) \equiv G_{-3}(5,2,-1)$;

2) translation of $-5\omega_2$ along the $\varpi$-axis, which
transforms the $G_{-3}(5,2,-1)$-symmetry into the symmetry group
$G_2(5,2,-1)$ of figure~\ref{fig:action5bEightChain} (see the remark
at the end of~\ref{sec:NGon});

3) scaling by $\omega_1/\omega_2$.

In the rotating frame, the family starts as the $P_{12}$-family but,
at some point, two more loops develop on the supporting curve which
then flattens to a horizontal planar (relative) choreography with two
loops, described twice per period. The figure below shows this
(quasiperiodic) solution in the inertial frame.

\vfill 

\begin{figure}[h]
     \centering
     \includegraphics[height=9cm]{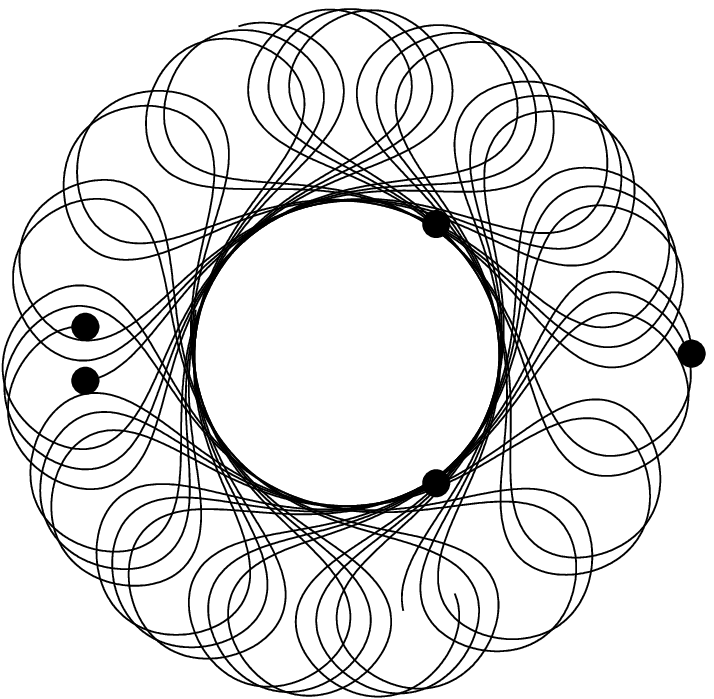}
     \label{fig:assiette}
\end{figure}

\vspace{1cm}

\begin{figure}[h]
     \centering
     \includegraphics[angle=90,height=17cm]{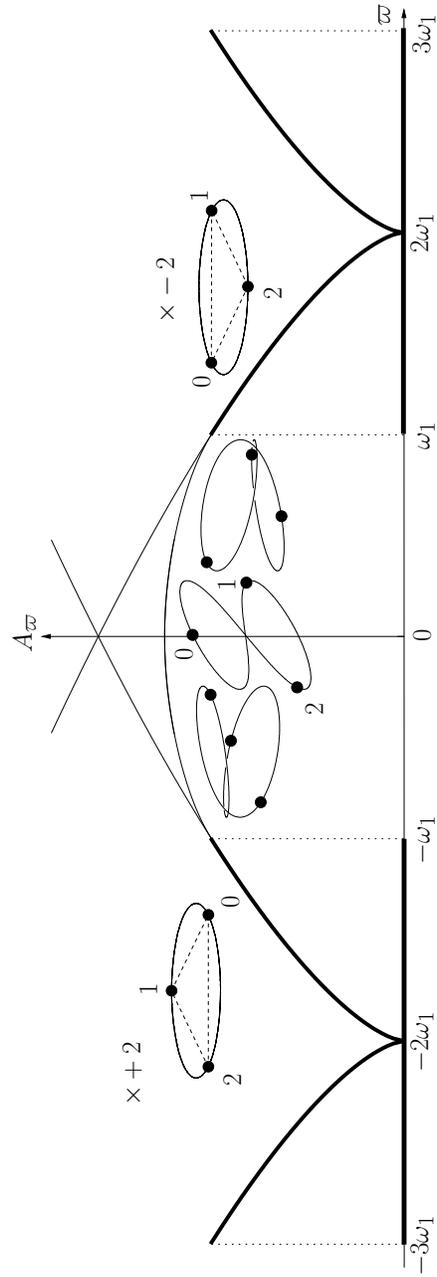}
     \caption{The action of the $P_{12}$ family and of two times
       Lagrange solution in the rotating frame}
     \label{fig:actionP12}
\end{figure}

\begin{figure}[h]
     \centering
     \includegraphics[angle=90,height=17cm]{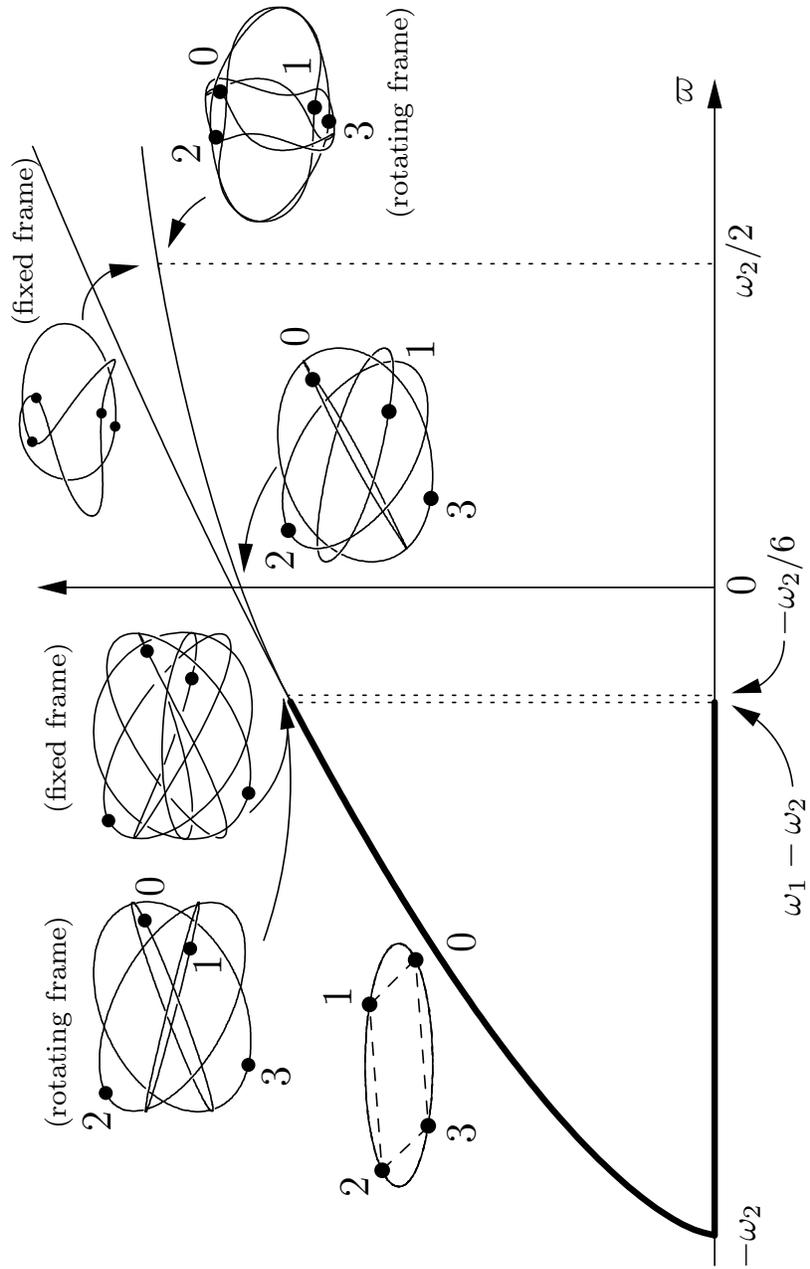}
     \caption{The action of the original Hip-Hop family (including two
       absolute choreographies discovered in~\cite{TV})}
     \label{fig:actionFamHH}
\end{figure}

\begin{figure}[h]
   \centering
   \includegraphics[angle=90,height=18cm]{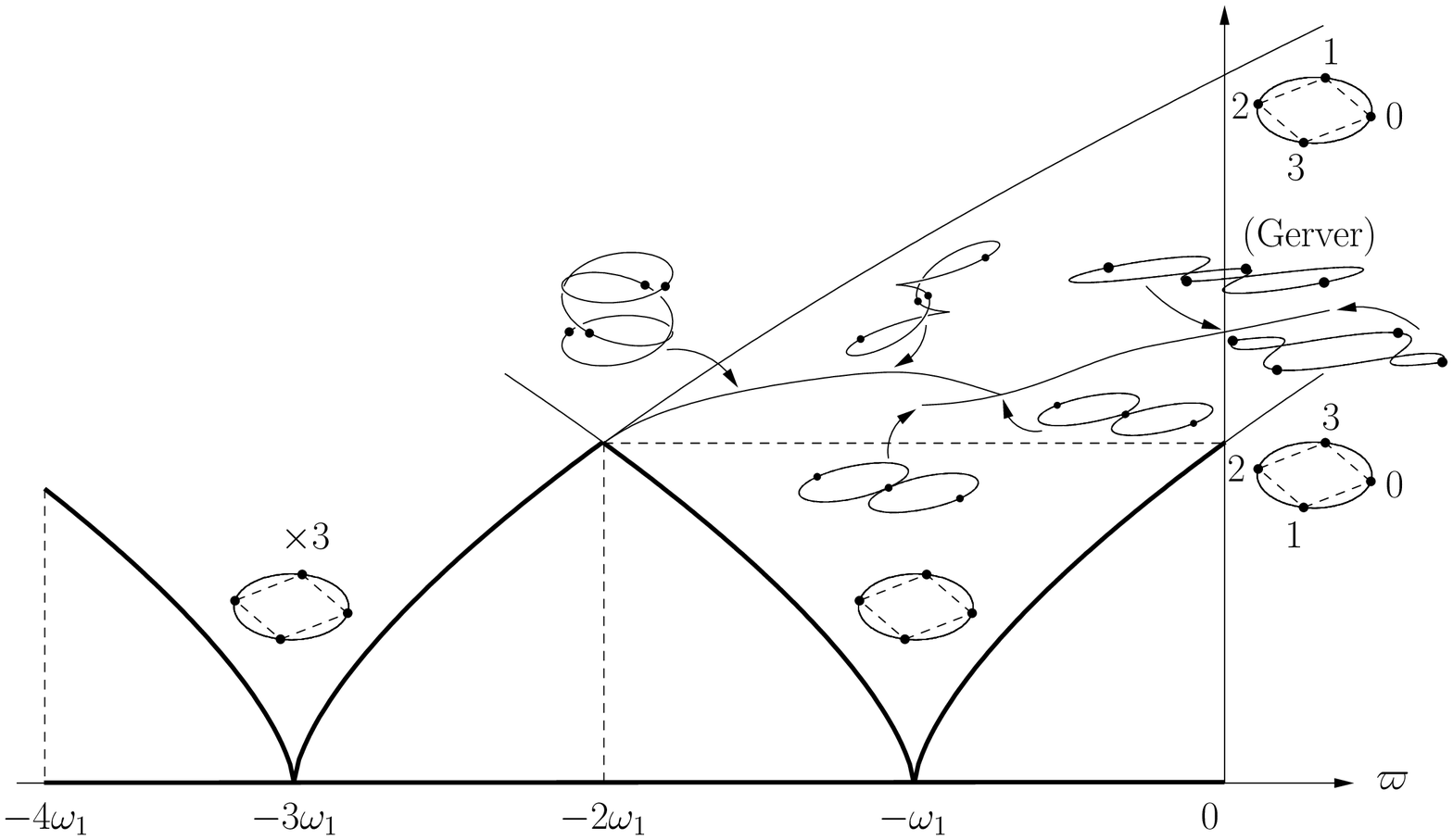}
   \caption{Action of the 4-body, 3-loop chain family, and of the
     planar Gerver family}
    \label{fig:action4cChaine3b}
\end{figure}

\begin{figure}[h]
     \centering
     \includegraphics[angle=90,height=17cm]{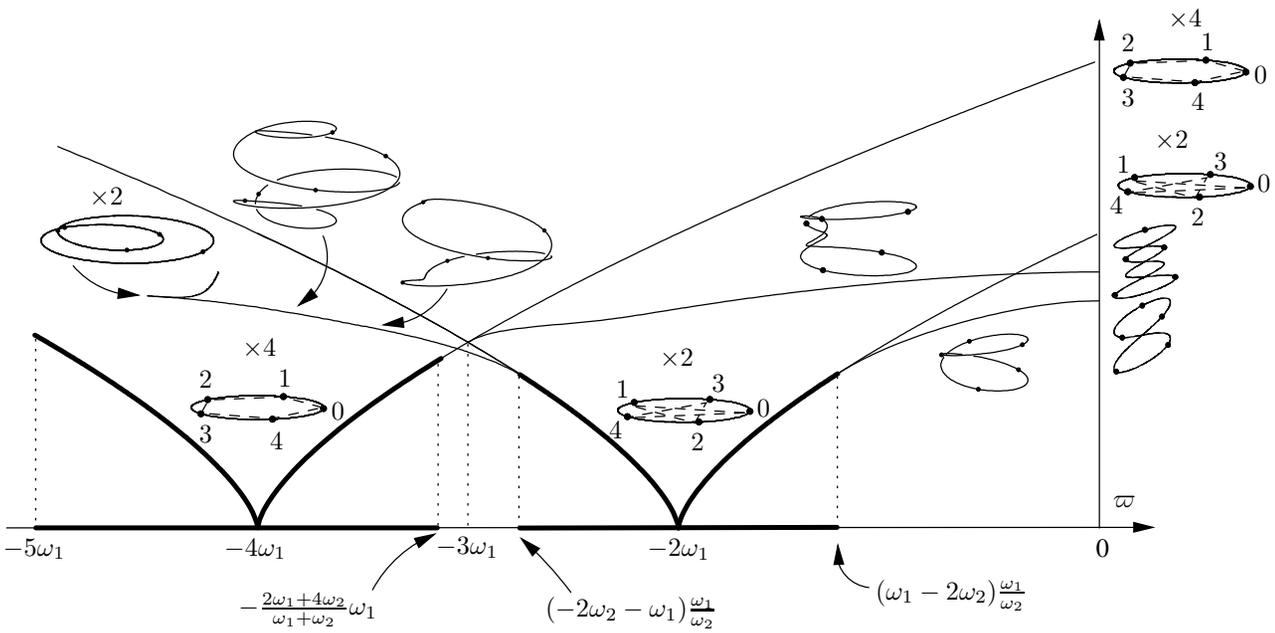}
     \caption{Action of the 5-body Eight and 4-loop chain
       families} 
     \label{fig:action5bEightChain}
\end{figure}

\begin{figure}[h]
   \centering
   \includegraphics[angle=90,height=18cm]{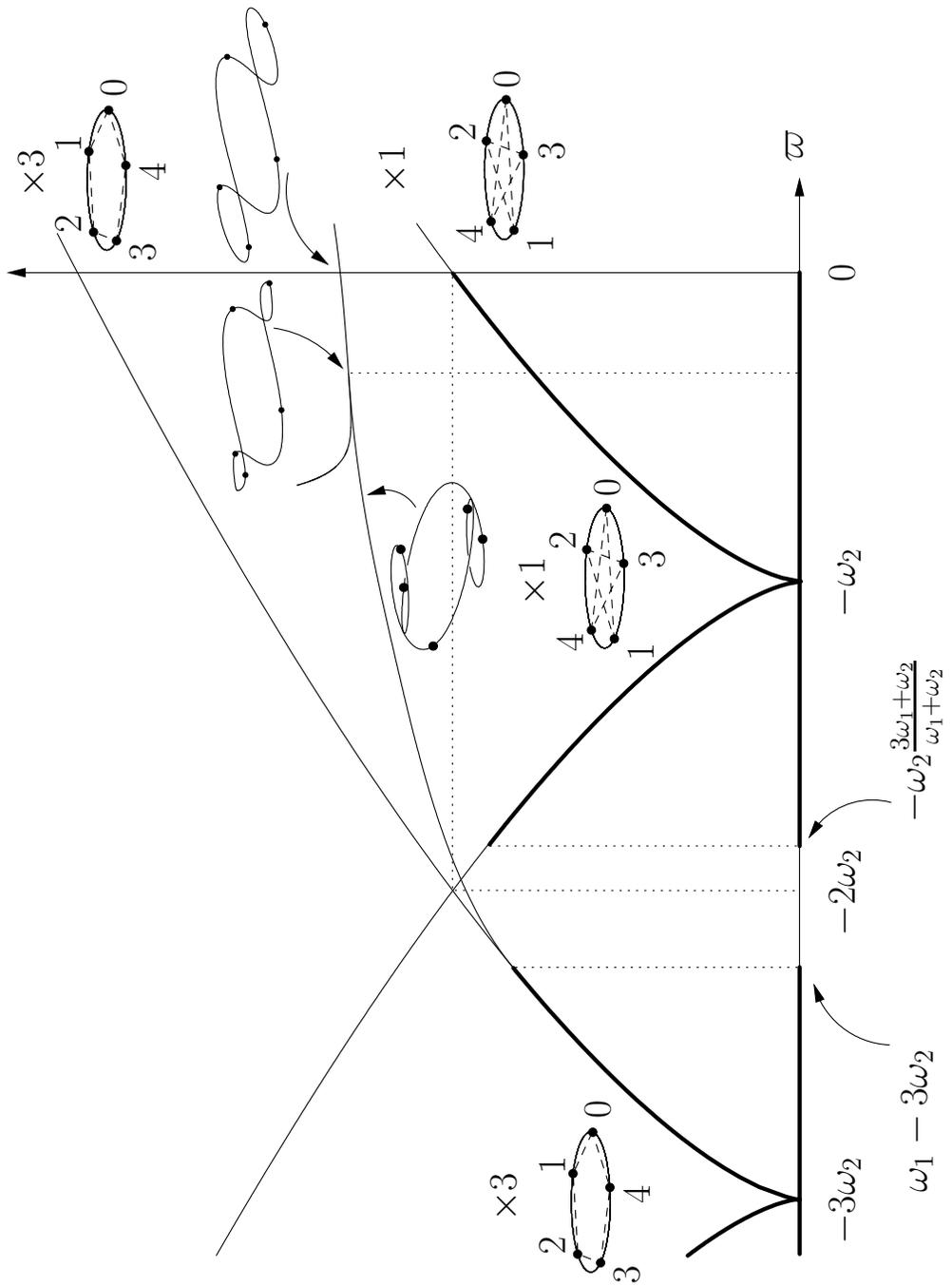}
   \caption{Action of the 5-body, 3-loop chain family, and of the
     corresponding planar family (cf.~\cite{S})}
   \label{fig:action5cChaine3b}
\end{figure}

\clearpage
\section{Appendix: Fourier expansions and the torsion}

Approximate Fourier expansions of $G_{r/s}(N,k,\eta)$-symmetric
Lyapunov families, in the same spirit as Marchal's computations in
\cite{Ma1,Ma2}, allow to evaluate the torsion under a regularity
hypothesis.

\subsection{The symmetry ansatz}

We are looking for local one-parameter families of solutions of the
$N$-body problem which, in a family of frames rotating with frequency
$$\varpi=\omega_1-\frac{r}{s}\omega_k+\tilde\varpi,$$ 
are periodic of period $T=s\frac{2\pi}{\omega_k}$.
{\it We will suppose that $\omega_k=2\pi$, and hence $T=s$.}

\medskip Such solutions are of the form
$$x_j(t)=(h_j(t),z_j(t)),\;
h_j(t)=e^{i(\omega_1-\frac{r}{s}\omega_k+\tilde\varpi)t} \tilde
h_j(t), \quad j=0,\cdots, N-1,$$ with
$$\tilde h_j(t)=\sum_{l=-\infty}^{+\infty}a_l^je^{i2\pi\frac{l}{s}t},\; 
z_j(t)=\Re\bigl(\sum_{l=-\infty}^{+\infty}b_l^je^{i2\pi\frac{l}{s}t}\bigr).$$
Moreover, we ask the solutions in the rotating frame $\tilde
x_j(t)=(\tilde h_j(t),z_j(t))$ to be symmetric under the action of
$G_{r/s}(N,k,\eta)$ described in section 3. Recall that an
$s$-periodic loop of configurations
$x(t)=\bigl(x_1(t),\cdots,x_N(t)\bigr)$ is invariant under the action
of $G_{r/s}(N,k,\eta)$ if and only if, for every
$(\theta,\delta,\beta,\xi)\in G_2$ representing an element of
$G_{r/s}(N,k,\eta)$, i.e. such that
$\theta-\frac{\beta}{2}-k\eta\frac{\delta}{N}=m\in\Z$, one has
$$\forall j\in \Z/{N\Z},\; x_j(t)=\rho
x_{\xi(j+\delta)}\bigl(\xi(t-\theta)\bigr), $$
where the action of $\rho$ on $\R^3=\C\times \R$ is defined by
$$\rho(h,z)=(e^{i2\pi\alpha} \bar h^\xi,e^{i\pi\beta} z) \quad
\mbox{with} \quad \alpha=\frac{r}{s}\theta-\frac{\delta}{N}\pmod{1},$$
where $\bar h^\xi =h$ if $\xi=+1$ and $\bar h^\xi = \bar h$ if $\xi =
-1$. 

\smallskip Translated in terms of Fourier coefficients, this reads:
$$a_l^j=e^{i2\pi(\alpha-\frac{l}{s}\theta)}
\overline{a_l^{\xi(j+\delta)}}^\xi,$$ 
for the horizontal components and
$$\Re (b_0^j-e^{i\pi\beta\xi} \overline{b_0^{\xi(j+\delta)}}^\xi)=0
\quad \hbox{and}\quad
b_l^j=e^{i(\pi\beta\xi-2\pi\frac{l}s{\theta})}\overline{b_l^{\xi(j+\delta)}}^\xi
\quad\hbox{if}\quad l\not=0$$ 
for the vertical components, implying restrictions on the coefficients
$a_l^j$ and $b_l^j$: 

\paragraph{Horizontal coefficients} We have
$$\alpha-\frac{l}{s}\theta = \frac{r-l}{s}\theta-\frac{\delta}{N} =
\frac{r-l}{s}(\frac{\beta}{2}+m) + \left(\frac{r-l}{s}k\eta-1\right)
\frac{\delta}{N}\cdot$$  
Fixing $\xi$ and $\delta$ and changing $\beta$, the angle
$e^{i2\pi(\alpha-\frac{l}{s}\theta)}$ takes at least two different
values as soon as $\frac{r-l}{s}$ is not even. Hence
$$a_l^j=0\quad\hbox{if}\quad \frac{r-l}{s}\quad\hbox{is not an even
  integer}.$$ 
 
If $\frac{r-l}{s}=2p$ is even, the symmetry conditions become
$$a_l^j=e^{i2\pi(2pk\eta-1)\frac{\delta}{N}}\overline{a_l^{\xi(j+\delta)}}^\xi\cdot$$
Applying these identities with both values $\xi=1$ and $\xi=-1$, we
get by difference
$$\forall p,j, a_{r-2ps}^{j}=\overline{a_{r-2ps}^{-j}}, \quad\hbox{hence}\quad 
\forall p,\; a_{r-2ps}^0=\overline{a_{r-2ps}^0}\cdot$$
So, the coefficients $a_{r-2ps}^0$ are real and
$$\forall p,j,\; a_{r-2ps}^j=e^{-i2\pi(2pk\eta-1)\frac{j}{N}}a_{r-2ps}^0\cdot$$
Finally, as the center of mass is at the origin, we have that for all
$l$, $\sum_{j=0}^{N-1}a_l^j=0$, hence 
$$\left(\sum_{j=0}^{N-1}e^{-i2\pi(2pk\eta-1)\frac{j}{N}}\right)a_{r-2ps}^0=0,$$
which implies
$$a_{r-2ps}^j=0\quad\hbox{for all}\quad p\quad\hbox{such that}
\quad2pk\eta-1=0\pmod N.$$

\medskip To summarize,
$$a_l^0=0\quad\hbox{unless possibly for}\quad l=r-2ps \quad\hbox{with}\quad 2pk\eta-1\not=0\pmod N$$
$$\forall p,j,\; a_{r-2ps}^j=e^{-i2\pi(2pk\eta-1)\frac{j}{N}}a_{r-2ps}^0\cdot$$

\paragraph{Vertical coefficients} Giving successively its two possible
values 0 and 1 to $\beta$, we get that
$$\forall j,\;\; \Re b_0^j=0.$$ Moreover, 
$$\frac{\beta}{2}\xi-\frac{l}{s}\theta =  
\frac{\beta}{2}(\xi-\frac{l}{s})-m\frac{l}{s}-k\eta\frac{l}{s}\frac{\delta}{N}$$
takes two different values for $\beta=0$ and $\beta=1$ as long as
$\xi-\frac{l}{s}$ is not an even integer, which implies that, for
$l\not=0$, $b_l^j=0$ can be different from 0 only if $l$ is of the
form $l=(2p+1)s$, with $p$ an integer.

Now, if $l\not=0$, choosing $\delta=-j$ and giving its two possible
values $\pm 1$ to $\xi$ and noticing that $\beta\xi=\beta\pmod 2$, we
get
$$b_l^j=e^{i2\pi(\frac{\beta}{2}(1-\frac{l}{s})-m\frac{l}{s}+k\eta\frac{l}{s}\frac{j}{N})}b_l^0
=e^{i2\pi(\frac{\beta}{2}(1-\frac{l}{s})-m\frac{l}{s}+k\eta\frac{l}{s}\frac{j}{N})}
\overline{b_l^0},$$
hence $b_l^0\in\R$. Finally, as the center of mass is at the origin,
$\forall l,\;\; \sum_{j=0}^{N-1}b_l^j=0$,
that is
$$e^{i2\pi(\frac{\beta}{2}(1-\frac{l}{s})-m\frac{l}{s})}
\left(\sum_{j=0}^{N-1}e^{i2\pi
    k\eta\frac{l}{s}\frac{j}{N}}\right)b_l^0=0.$$ 
Finally, $b_l^0$ (and hence all $b_l^j$) must vanish unless $k\eta l\not=0\pmod {sN}$.

\medskip To summarize,
$$b_l^0=0\quad\hbox{unless possibly for}\quad l = (2q+1)s
\quad\hbox{with}\quad (2q+1)k\eta\not=0\pmod N,$$ 
$$\forall q,j,\;\; b_{(2q+1)s}^j=e^{i2\pi k\eta(2q+1)\frac{j}{N}}b_{(2q+1)s}^0\cdot$$

In particular,
$$z_j(t)=\sum_{q\ge 0}c_{2q+1}\cos 2\pi(2q+1)(t+k\eta\frac{j}{N}),$$
where we have used the notation
$$c_{2q+1}=b_{(2q+1)s}^0+b_{-(2q+1)s}^0,\quad q=0,1,\cdots$$

\medskip For example, these conditions for the $P_{12}$ family (that
is for the group $G_{2}(3,1,-1)$) become
$$a_l^0\in\R,\quad a_l^j=0\quad\hbox{if either}\quad l=1\pmod 2\quad\hbox{or}\quad l=0\pmod 3,$$
$$b_l^0\in\R,\quad b_l^j=0\quad\hbox{if either}\quad l=0\pmod 2\quad\hbox{or}\quad l=0\pmod 3,$$ 
which coincide with the conditions found by Marchal in \cite{Ma1}. 
 
\subsection{The regularity ansatz}

In the rotating frame, the solutions we are interested in are the
ones tangent to the cylinder of solutions of (VVE)
$$(A_0\zeta^je^{i2\pi\frac{r}{s} t}, \epsilon\Re{\zeta^{k\eta
    j}e^{i2\pi t}}),\; j=0,\cdots,N-1,$$ 
where as usual $\zeta=e^{i\frac{2\pi}{N}}$.   

\medskip In the inertial frame, these solutions are of the form
$$x_j(t)=(h_j(t),z_j(t))\qquad j=0,\cdots, N-1,$$
$$h_j(t)=e^{i(\omega_1+\tilde\varpi)t}\sum_{p}a_{r-2ps}^0
e^{i2\pi[-2pt-(2pk\eta-1)\frac{j}{N}]},$$ 
$$z_j(t)=\sum_{q\ge 0}c_{2q+1}\cos 2\pi(2q+1)(t+k\eta\frac{j}{N}),$$
where the coefficients $a_l^0(\epsilon)$ and $c_l(\epsilon)$ are real
and the summations are respectively over all integers $p\in\Z$ and
$q\in\Z_+$ such that $$2pk\eta-1\not=0\pmod N\quad\hbox{and}\quad
(2q+1)k\eta\not=0\pmod N.$$ 

We will assume that, in the reduced phase space where the relative
equilibrium becomes an equilibrium, the one-parameter family of
$s$-periodic solutions we are looking at generates an analytic
foliation of an analytic 2-dimensional surface, more precisely, that
it is the image under an analytic embedding of the trivial planar
family $u(t)=\epsilon\cos 2\pi t/s,\, v(t)=\epsilon\sin 2\pi t/s$.  As
$(\cos^k 2\pi t/s)(\sin^l 2\pi t/s)$ has no harmonics of order larger than
$k+l$, this implies in particular that the Fourier expansion of the
coefficient of $\epsilon^n$ in $x(t)$ does not contain harmonics of
order larger than $n$. This justifies taking $$\epsilon=c_1$$ as our
main local parameter and making the following ansatz:
$$a_{r-2ps}^0=A_{2p}\epsilon^{|2p|}+O(\epsilon^{|2p|+2}),\quad
c_{2q+1}=C_{2q+1}\epsilon^{2q+1}+O(\epsilon^{2q+3}),$$ which in the
case of the $P_{12}$ family coincides with the one made in~\cite{Ma1}.

\smallskip Using both ansatz, an identification in the equations of
motion allows us to determine the leading coefficients $a_l^0$ and
$c_l$, as well as the dominant term in $\tilde\varpi$, in terms of
$\epsilon$.  The dominant terms in the components of $z_j(t)$ are
$$z_j(t)=\epsilon\cos 2\pi(t+k\eta\frac{j}{N})+ C_3\epsilon^3\cos
6\pi(t+k\eta\frac{j}{N})+O(\epsilon^5),$$ while those of $h_j(t)$ are
$$h_j=e^{i(\omega_1+\tilde\varpi)t} \sum_{p=0,\pm
  1}a_{r-2ps}^0e^{i2\pi (-2pt+\frac{j[p]}{N})}+O(\epsilon^4),$$ 
that is
$$h_j=e^{i(\omega_1+\gamma\epsilon^2)t} \left[
  \begin{array}[c]{l}
    (A_0+\alpha\epsilon^2) e^{i2\pi\frac{j}{N}}+ \\
    \left(A_{2}e^{i2\pi
        (-2t+\frac{j[1]}{N})}+A_{-2}e^{i2\pi
        (2t+\frac{j[-1]}{N})}\right)\epsilon^2
  \end{array}
\right] +O(\epsilon^4),$$
where we have used the following notations (we have directly set to
zero the coefficient of $\epsilon$ in $a_r^0$ and $\tilde\varpi$
because this is an immediate consequence of the equations in the next
section):
$$a_r^0=A_0+\alpha\epsilon^2+O(\epsilon^4),\quad
\tilde\varpi=\gamma\epsilon^2+O(\epsilon^4).$$ 

\paragraph{The Eight families $G_2(N=2n+1,n,-1,)$}
$$h_j(t)=e^{i(\omega_1+\tilde\varpi)t}
\left(a_{2}^0e^{i2\pi\frac{j}{N}}+a_{4}^0e^{i2\pi(2t+(-2n+1)\frac{j}{N})} 
+\cdots\right),$$

\paragraph{The Hip-Hop families $G_1(N=2n,n,\pm 1)$}
$$h_j(t)=e^{i(\omega_1+\tilde\varpi)t}\left(a_{1}^0e^{i2\pi\frac{j}{N}}+
a_{3}^0e^{i2\pi(2t+(2n+1)\frac{j}{N})}+a_{-1}^0e^{i2\pi(-2t-(2n-1)\frac{j}{N})}
+\cdots\right),$$

\paragraph{The maximal chain families $G_{2n}(2n+1,1,-1)$}
$$h_j(t)=e^{i(\omega_1+\tilde\varpi)t}\left(a_{2n}^0e^{i2\pi\frac{j}{N}}
  +a_{2n+2}^0e^{i2\pi(2t-\frac{j}{N})}+a_{2n-2}^0e^{i2\pi(-2t+3\frac{j}{N})}
+\cdots\right),$$
except if $N=3$ where the last term is absent.
 
\subsection{Identification of dominant coefficients }

It is remarkable how well the two ans{\"a}tze above make it possible to
identify Fourier expansions of the Lyapunov families: the regularity
ansatz makes the system of equations block-triangular, while the
symmetry ansatz makes the solution unique for a given choice of group.

\paragraph{Notations}

$$e^{i2\pi\frac{j}{N}}-e^{i2\pi\frac{l}{N}} = \vec
r_{jl}=u_{jl}+iv_{jl} = \rho_{jl}e^{i4\pi \theta_{jl}},\quad
-(2pk\eta-1)j=j[p].$$ 
$$A_{jl}=A_0\rho_{jl},\quad  B_{jl}=\frac{1}{A_{jl}}\sin^2
2\pi(k\eta\frac{j-l}{2N}), \quad \Theta_{jl;p}=\theta_{jl}-\theta_{j[p]l[p]}.$$
Finally, we recall the restrictions on the values of $p$ and $q$ which
come into the sums; in particular, $\sum_{p=\pm 1}$ will mean the sum
restricted to those values of $p=\pm 1$ such that $2pk\eta-1\not=0
\pmod{N}$. 

\paragraph{Mutual distances} With the notations above,
$$h_j-h_l=e^{i(\omega_1+\tilde\varpi)t} \sum_{p}a_{r-2ps}^0e^{-i4\pi
  pt}\vec r_{j[p]l[p]}\; ;$$ 

$$z_j-z_l=\sum_{q\ge 0}c_{2q+1}\left[\cos
  2\pi(2q+1)(t+k\eta\frac{j}{N}) -\cos 2\pi(2q+1)(t+k\eta\frac{l}{N})\right].$$
Mod $O(\epsilon^4)$ we obtain
$$||h_j-h_l||^2=(A_0\rho_{jl})^2+2A_0\left[\rho_{jl}\alpha+\sum_{p=\pm
    1} A_{2p}<\vec r_{jl},e^{-i4\pi pt}\vec r_{j[p]l[p]}>\right]\epsilon^2;$$
$$|z_j-z_l|^2=\left[\cos 2\pi(t+k\eta\frac{j}{N})- \cos
  2\pi(t+k\eta\frac{l}{N})\right]^2\epsilon^2,$$ 
that is
$$||h_j-h_l||^2=(A_0\rho_{jl})^2+2A_0\rho_{jl}\left[\alpha+\sum_{p=\pm
    1}A_{2p}\rho_{j[p]l[p]}\cos
  4\pi(pt+\theta_{jl}-\theta_{j[p]l[p]})\right]\epsilon^2\;
;$$ 
$$|z_j-z_l|^2=4\sin^2 2\pi(t+k\eta\frac{j+l}{2N})\sin^2
2\pi(k\eta\frac{j-l}{2N})\epsilon^2,$$ 
which, using the notations above, becomes
$$||h_j-h_l||^2=A_{jl}^2\left[1+\frac{2}{A_{jl}}
  \left(\alpha+\sum_{p=\pm 1}A_{2p}\rho_{j[p]l[p]}\cos
    4\pi(pt+\Theta_{jl;p})\right)\epsilon^2\right];$$ 
$$|z_j-z_l|^2=2A_{jl}B_{jl} \left(1-\cos
  4\pi(t+k\eta\frac{j+l}{2N})\right) \epsilon^2.$$

Finally, $\|x_j-x_l\|^{-3}$ equals
$$A_{jl}^{-3}\left[1-\frac{3}{A_{jl}} \left(
    \begin{array}[c]{l}
      \alpha+ \sum_{p=\pm 1}A_{2p}\rho_{j[p]l[p]}\cos
      4\pi(pt+\Theta_{jl;p})+\\
      B_{jl} \left(1-\cos 4\pi(t+k\eta\frac{j+l}{2N})\right)
    \end{array}
  \right) \epsilon^2\right]+O(\epsilon^4).$$

We will now plug in the obtained expressions into the equations of
motion
$$\ddot h_j=\sum_{l\not=j}\frac{h_l-h_j}{||x_l-x_j||^3},\quad \ddot
z_j=\sum_{l\not=j}\frac{z_l-z_j}{||x_l-x_j||^3}.$$ 

\paragraph{Horizontal equations} Recall that
$$h_j=e^{i(\omega_1+\tilde\varpi)t} \sum_{p=0,\pm 1}a_{r-2ps}^0
e^{i2\pi (-2pt+\frac{j[p]}{N})}+O(\epsilon^4),$$ 
hence
$$
  \ddot h_j=e^{i(\omega_1+\tilde\varpi)t}\sum_{p=0,\pm 1}
  \left(
    \begin{array}[c]{c}
      a_{r-2ps}^0 e^{i2\pi (-2pt+\frac{j[p]}{N})} \times\\
      \left[-(\omega_1+\tilde\varpi)^2+8\pi
        p(\omega_1+\tilde\varpi)-16\pi^2p^2\right]
    \end{array}
  \right)
  +O(\epsilon^4),
$$
that is
$$e^{-i(\omega_1+\tilde\varpi)t}\ddot
h_j=-(A_0+\alpha\epsilon^2)(\omega_1+\gamma\epsilon^2)^2
e^{i2\pi\frac{j}{N}}$$  
$$+\sum_{p=\pm 1}A_{2p}\left[-(\omega_1+\gamma\epsilon^2)^2+8\pi
  p(\omega_1+\gamma\epsilon^2)-16\pi^2p^2\right]e^{i2\pi
  (-2pt+\frac{j[p]}{N})} \epsilon^2+O(\epsilon^4),$$ 
or
\begin{eqnarray*}
  e^{-i(\omega_1+\tilde\varpi)t}\ddot h_j
  &=&-\omega_1^2A_0e^{i2\pi\frac{j}{N}}+\\
  &&\left[
    \begin{array}[c]{l}
      -\omega_1(2A_0\gamma+\omega_1\alpha)e^{i2\pi\frac{j}{N}}+ \\
      \sum_{p=\pm 1}A_{2p}\left[-\omega_1^2+8\pi
        p\omega_1-16\pi^2p^2\right] e^{i2\pi (-2pt+\frac{j[p]}{N})}
    \end{array}\right] \epsilon^2+ \\
  &&O(\epsilon^4),
\end{eqnarray*}
Finally, as
$$e^{-i(\omega_1+\tilde\varpi)t}(h_l-h_j)=
(A_0+\alpha\epsilon^2)\vec r_{lj}+ \sum_{p=\pm
  1}A_{2p}\epsilon^2e^{-i4\pi pt} \vec r_{l[p]j[p]}+O(\epsilon^4),$$
the horizontal part of the equations splits into the 0-th order 
equation, which is nothing but the equation satisfied by the relative
equilibrium:
$$-\omega_1^2A_0e^{i2\pi\frac{j}{N}}=
\sum_{l\not=j}\frac{A_0}{(A_0\rho_{jl})^3}\vec r_{lj}\;,$$ 
and the second order equation
$$-\omega_1(2A_0\gamma+\omega_1\alpha)e^{i2\pi\frac{j}{N}}
+\sum_{p=\pm 1}A_{2p}\left[-\omega_1^2+ 8\pi
  p\omega_1-16\pi^2p^2\right] e^{i2\pi (-2pt+\frac{j[p]}{N})} =$$ 
\begin{eqnarray*}
  &&\sum_{l\not=j}A_{lj}^{-3}\left(\alpha\vec r_{lj}+\sum_{p=\pm
      1}A_{2p} e^{-i4\pi pt}\vec r_{l[p]j[p]}\right)+\\
  &&-3A_0\sum_{l\not=j}A_{jl}^{-4}\left(
    \begin{array}[c]{l}
    \alpha+\sum_{p=\pm 1} A_{2p}\rho_{j[p]l[p]}\cos
    4\pi(pt+\Theta_{jl;p})+\\
    B_{jl}\left(1-\cos 4\pi(t+k\eta\frac{j+l}{2N})\right)
    \end{array}
  \right) \vec r_{lj}. 
\end{eqnarray*}
We will call $(H_j)$ this equation. 

\smallskip \noindent Now, because of the symmetry ansatz, the $N$
equations $(H_j)$ are equivalent to one of them, for example $(H_0)$,
which is of the following form
$$U+Ve^{i4\pi t}+We^{-i4\pi t}=0,$$
where the complex numbers $U, V,W$ are affine functions of the 4 real
unknowns
$$\alpha,\; A_{2},\; A_{-2},\; \gamma.$$ Moreover, it turns out that,
because of the invariance of the expressions $A_{0l},B_{0l},\cdots$
under the change of $l$ into $-l$, the coefficients of $U,V,W$ are
indeed real.  As $(H_0)$ has to be satisfied for all values of $t$, it
is equivalent to the three real affine equations  
$$U=V=W=0.$$

\paragraph{Vertical equations} At the order of approximation
$O(\epsilon^4)$, we have
$$z_j(t)=\epsilon(\cos 2\pi(t+k\eta\frac{j}{N})+C_3\epsilon^3 \cos
6\pi(t+k\eta\frac{j}{N})+O(\epsilon^5),$$ 
hence
$$z_l-z_j=\epsilon\left[u_{jl}\sin 2\pi(t+k\eta\frac{j+l}{2N})+
  C_3v_{jl}\epsilon^2\sin
  6\pi(t+k\eta\frac{j+l}{2N})+O(\epsilon^4)\right],$$ 
where
$$u_{jl}=-2\sin 2\pi(k\eta\frac{j-l}{2N}),\quad v_{jl}= -2\sin
6\pi(k\eta\frac{j-l}{2N}).$$ The vertical equation $(V_j)$ is obtained
by identifying the $\epsilon^2$-terms in the identity $\pmod
{O(\epsilon^4)}$
$$-4\pi^2\cos 2\pi(t+k\eta\frac{j}{N})-36\pi^2C_3\epsilon^2 \cos
6\pi(t+k\eta\frac{j}{N})=$$ 
$$\sum_{l\not=j}||x_l-x_j||^{-3}
\left(u_{jl}\sin 2\pi(t+k\eta\frac{j+l}{2N})+C_3v_{jl}\epsilon^2\sin
  6\pi(t+k\eta\frac{j+l}{2N})+O(\epsilon^4)\right)$$ The
identification of the terms of order 0 give the vertical variational
equation (VVE) which is already satisfied by $z_j(t)=\cos
2\pi(t+k\eta\frac{j}{N})$.

\noindent The identification of the terms of order 2 gives the two
remaining relations beween the five unknowns
$\alpha,A_2,A_{-2},C_3,\gamma$. They are of the form
$$\Re(X_je^{i2\pi t})=0,\quad \Re(Y_je^{i6\pi t})=0,$$
where $X_j$ is an affine function of $\alpha,A_2,A_{-2}$ and $Y_j$ is
an affine function of $A_2,A_{-2},C_3$. As above, the symmetry ansatz
implies that the $N$ equations $(V_j)$ are equivalent to one of them,
for example $(V_0)$ and the invariance of the expressions
$A_{0l},B_{0l},\cdots$ under the change of $l$ into $-l$ implies that
$X_0$ and $Y_0$ are real, hence that $(V_0)$ reduces to exactly two
real equations. 


\paragraph{Identification of coefficients} Finally, we have 5 real
equations which are affine in the 5
unknowns $$\alpha,A_2,A_{-2},C_3,\gamma.$$ They can be solved in the
following order (assuming non degeneracies which will prove true in
the cases investigated below):

-- In $(H_0)$ the term in $\epsilon^0$ depends only on $A_0$, which it
determines. 

-- In $(H_0)$ the term $U$ in $\epsilon^2$ constant with respect to
time depends only on $\alpha$ and $\gamma$ and can be used to
eliminate $\alpha$.

-- In $(H_0)$ the coefficients $V$ and $W$ of $\epsilon^2 e^{\pm i4\pi
  t}$ allow to eliminate $A_{\pm 2}$. (For the Eight families, it can
help to remember or it can be checked, that $a_0^0=A_2\epsilon^2 +
... =0$.)

-- In $(V_0)$, the coefficient of $\epsilon^2 e^{i6\pi t}$ (after
simplification by $\epsilon$) allows to eliminate $C_3$. (In the case
of the $P_{12}$ family, $C_3=0$.)

-- In $(V_0)$, the coefficient of $\epsilon^2 e^{i2\pi t}$ eventually
allows to compute $\gamma$. 


\paragraph{Remark} The fact that, in the rotating frame, the
horizontal period is half the vertical one follows from the invariance
of the problem under the symmetry with respect to the horizontal plane
(materialized in the action of $\beta$). This fact accounts for the
generic aspect displayed by the solutions in a frame which accompanies
the rotation of the regular $N$-gon, as exemplified on
figure~\ref{fig:halo}. Notice that the same type of behaviour, with
the same explanation, is observed in the restricted problem for the
vertical solutions originating from the Lagrange points (see
\cite{Za}).

\subsection{Torsion of the first cases}

Recall that we are computing the torsion of the Lyapunov families
which bifurcate at $\varpi = -\frac{r}{s}\omega_k + \omega_1$. Since
$\omega_1>0$, $\gamma$ is necessarily $\geq 0$.

\paragraph{Three bodies}

\subparagraph{The family of rotated Lagrange solutions ($S_2(3,1,-1)$)} 
$$\gamma = 0 \quad \mbox{as expected.}$$

\subparagraph{The $P_{12}$-family ($S_2(3,1,-1)$)} 
$\gamma = \displaystyle\frac{12}{19} \left( 6\pi^7 \right)^{1/3}
\simeq 16.589$

\noindent This value is consistent with Marchal's
computation~\cite{Ma1}.  Indeed, writing with primes quantities
introduced in his book, he finds that the frequency shift of the
$P_{12}$-family is 
$$\frac{\alpha'}{2\pi} + h.o.t. = \frac{3}{19} c_1'^2 + h.o.t.,$$
to be compared with our $\gamma \epsilon^2$. But $\omega_1' = 1$ while
$\omega_1 = 2\pi$. Hence
$$2 \pi \frac{3}{19} c_1'^2 = \gamma\epsilon^2$$
(that Marchal's masses are not the same as ours does not come into
play in this equality between frequencies). Besides, his $c_1'$ is the
amplitude of oscillations of $z_1-z_2$ at the first order and the
edges of the equilateral triangle he considers have length $1$. On the
other hand, our $\epsilon$ is the amplitude of oscillations of $z_1$
and we consider a triangle whose edges have length $a_0\sqrt 3$, with
$a_0^{-3} = 4\sqrt{3}\pi^2$. Hence
$$\left(c_1'/{\scriptstyle
    \frac{1}{\sqrt{3}}}\right):{\scriptstyle \frac{1}{\sqrt{3}}} =
\epsilon : a_0,$$
whence the above expression of $\gamma$.

\paragraph{Four bodies}

\subparagraph{The 4-body, 3-lobe chain family ($S_2(4,1,-1)$)}

$$\gamma = \frac{48\, \pi^{7/3}}{41 \times 2^{1/3}7^{2/3}}
\left( 9\sqrt 2 -11 \right) \left( 2\sqrt2 +1 \right)^{2/3} > 9$$

\subparagraph{The 4-body Hip-Hop family ($S_1(4,2,\pm 1)$)}

$$\gamma = \frac{4692\;\pi^{7/3}}{2446096\times 2^{1/4} \sqrt{2\sqrt{2} +
    1}} \left(440 \sqrt 2 + 981\right) > 19$$

\paragraph{Five bodies}

\subparagraph{The 5-body, 4-lobe chain family ($S_4(5,1,-1)$)} 

$$\gamma = {{6\,2^{{{1}\over{3}}}\,\left(17995-7823\,\sqrt{5}\right)\,\pi^{{{7
 }\over{3}}}}\over{5981 \times 5^{{{1}\over{3}}}}} > 5.$$

\subparagraph{The 5-body Eight family ($S_2(5,2,-1)$)} 

$$\gamma =
\frac{\pi^{{{7}\over{3}}}(1+\sqrt{5})^{1/6}(1+3\sqrt{5})^{1/6}}{2^{11} 
  \times 5^{1/3} \times 11^{55/6} \; 761\, 809} \times $$
$$\left(
  \begin{array}[c]{l}
    \left(230356305630646272\, \sqrt{5}+954953942246092800\right)
    \sqrt{\sqrt{5}-1}\,\sqrt{3\,\sqrt{5}-1}\\
    +10427244637440119808\,\sqrt{5}-3532108049365632000
  \end{array}\right) > 19.
$$

\subparagraph{The 5-body, 3-lobe chain family ($S_3(5,2,1)$)} 

\begin{eqnarray*}
  \gamma &=& 
  \frac{\pi^{{{7}\over{3}}} \left(\sqrt{5}+1\right)^{{{1}\over{6}}}
    \left(3\,\sqrt{5}-1\right)^{{{1}\over{6}}}}{69129994912 \times
    5^{1/3} \times 11^{11/6}} \times \\
  &&\left(
  \begin{array}[c]{l}
    \left(522541527\, \sqrt{5} - 177004875\right) - \\
    \left(11543868\,\sqrt{5}+ 47855700\right) \,
    \sqrt{\sqrt{5}-1}\,\sqrt{3\,\sqrt{5}-1} 
  \end{array}\right) > 12
\end{eqnarray*}

\paragraph{Six bodies}

\subparagraph{The 6-body, 5-loop chain ($S_5(6,1,-1)$)}

$$\gamma = \frac{48 \, \pi^{7/3}}{781199 \times 6962^{1/3}}
(32634\sqrt{3}-39889) (15-4\sqrt{3})^{2/3} > 3$$

\subparagraph{A 6-body Hip-Hop ($S_1(6,2,-1)$)} $\gamma > 14$ (long expression)

\clearpage \bigskip\emph{Acknowledgments. } The authors warmly thank
Carles Sim{\'o} for his invaluable advice on numerical computations, and
Richard Montgomery for his comments and questions. Thanks also to
Michel H{\'e}non who made us aware of his researches~\cite{H} on families
of spatial periodic solutions of the restricted circular three-body
problem which bifurcate from planar solutions. In the same way as his
families were continued to the full problem, it would be natural to
continue ours by varying the masses.

\end{document}